\documentclass[a4paper,12pt]{amsart}
\usepackage[utf8]{inputenc}
\usepackage{amsthm,stmaryrd}
\usepackage{amssymb,enumerate}
\usepackage{epsfig}
\usepackage[usenames,dvipsnames]{color}
\usepackage{verbatim}
\usepackage{hyperref}
\usepackage{mathrsfs}
\usepackage{dsfont}
\usepackage[all]{xy}
\newcommand{\rel}{\ell}

\newcommand{\any}{{\diamond}}
\newcommand{\s}{{\mathsf s}}

\newcommand{\Gf}{{\Gamma}}
\newcommand{\T}{{\mathcal T}}
\newcommand{\VV}{{\mathcal V}}
\newcommand{\Mat}{\mathsf{Mat}}
\newcommand{\Ch}{\ensuremath {\operatorname{Cb}}}
\newcommand{\Conj}{\ensuremath {\operatorname{Conj}}}
\newcommand{\Env}{\ensuremath {\operatorname{Env}}}
\newcommand{\Alg}{\ensuremath {\operatorname{Alg}}}

\newcommand{\W}{\ensuremath {W}}

\newcommand{\SL}{\mathsf{SL}}
\newcommand{\GL}{\mathsf{GL}}
\newcommand{\GB}{\mathsf{B}}
\newcommand{\GH}{\mathsf{H}}
\newcommand{\GA}{\mathsf{A}}
\newcommand{\GN}{\mathsf{N}}
\newcommand{\GK}{\mathsf{K}}
\newcommand{\Grc}{{\mathsf{G}^*}}
\newcommand{\Gd}{{\mathsf{G}}}

 \newcommand{\DQ}{{{\mathscr D}_\Qm}}

  \newcommand{\Dk}{{{\mathscr D}_{\kk}}}
    \newcommand{\FQ}{F_{\Dk}}

\newcommand{\tcoev}{\stackrel{\longrightarrow}{\operatorname{coev}}}
\newcommand{\tev}{\stackrel{\longrightarrow}{\operatorname{ev}}}
\newcommand{\ev}{\stackrel{\longleftarrow}{\operatorname{ev}}}
\newcommand{\coev}{\stackrel{\longleftarrow}{\operatorname{coev}}}

\newcommand{\rcoev}{\stackrel{\longrightarrow}{\operatorname{coev}}}
\newcommand{\rev}{\stackrel{\longrightarrow}{\operatorname{ev}}}
\newcommand{\lev}{\stackrel{\longleftarrow}{\operatorname{ev}}}
\newcommand{\lcoev}{\stackrel{\longleftarrow}{\operatorname{coev}}}

\newcommand{\brk}[1]{{\left\langle{#1}\right\rangle}}

\newcommand{\ve}{\varepsilon}

\newcommand{\vp}{\varphi}

\newcommand{\e}{{\operatorname{e}}}
\newcommand{\slt}{{\mathfrak{sl}(2)}}
\newcommand{\Uq}{{U_q\slt}}
\newcommand{\unit}{\mathds{1}}
\newcommand{\cat}{\mathscr{C}}
\newcommand{\RR}{\mathscr{R}}

\newcommand{\Tang}{{{\operatorname{\mathcal{T}}}}}
\newcommand{\Diag}{{{\operatorname{\mathcal{D}}}}}
\newcommand{\Qm}{{{\mathcal{Q}}}}
\newcommand{\QX}{{{\operatorname{\mathsf{Q}}}}}
\newcommand{\Link}{{{\operatorname{\mathsf{L}}}}}
\newcommand{\D}{\mathscr{D}}
\newcommand{\Id}{\operatorname{Id}}
\newcommand{\bp}[1]{{\left(#1\right)}}
\newcommand{\qn}[1]{{\left\{#1\right\}}}

\newcommand{\qd}{{\mathsf d}}
\newcommand{\End}{\operatorname{End}}
\newcommand{\Hom}{\operatorname{Hom}}
\newcommand{\Aut}{\operatorname{Aut}}
\newcommand{\C}{\ensuremath{\mathbb{C}} }
\newcommand{\Z}{\ensuremath{\mathbb{Z}} }
\newcommand{\R}{\ensuremath{\mathbb{R}} }
\newcommand{\N}{\ensuremath{\mathbb{N}} }
\newcommand{\Q}{\ensuremath{\mathbb{Q}} }
\newcommand{\wt}{\widetilde}
\newcommand{\wh}{\widehat}
\newcommand{\wb}{\overline}
\newcommand{\ms}[1]{\mbox{\tiny$#1$}}
\newcommand{\oo}{\infty}
\newcommand{\et}{{\quad\text{and}\quad}}
\newcommand{\ets}{{\,\text{,}\quad}}

\newcommand{\gen}{\mathcal{G}}
\newcommand{\ideal}{I}
\newcommand{\Proj}{{{\operatorname{\mathsf{Proj}}}}}
\newcommand{\kt}{$\Bbbk$\nobreakdash-\hspace{0pt}}
\newcommand{\kk}{\Bbbk}

\renewcommand{\s}{\sigma}

\newcommand{\gp}{{\rightharpoonup}}
\newcommand{\gm}{{\rightharpoondown}}

\newcommand{\mt}{\mathsf t}

\newcommand{\mat}[2]{{\small
    \left(\begin{array}{cc}
      1&#2\\0&#1
    \end{array}\right)}}

\newcommand{\epsh}[2]
         {\begin{array}{c} \hspace{-1.3mm}
        \raisebox{-4pt}{\epsfig{figure=#1,height=#2}}
        \hspace{-1.9mm}\end{array}}

\newtheorem{theorem}{Theorem}[section]
\newtheorem{proposition}[theorem]{Proposition}
\newtheorem{lemma}[theorem]{Lemma}
\newtheorem{corollary}[theorem]{Corollary}

\newcounter{IntroCounter}
\stepcounter{IntroCounter}

\theoremstyle{definition}
\newtheorem{example}[theorem]{Example}
\newtheorem{conclusion}[theorem]{Conclusion} 
\newtheorem{definition}[theorem]{Definition}

\theoremstyle{remark}
\newtheorem{remark}[theorem]{Remark}

\newcounter{exo} \newcounter{numexercice}
\renewcommand{\theexo}{\arabic{exo}}

\begin{document}
\title[Holonomy braidings, biquandles and quantum invariants]{Holonomy braidings, biquandles and quantum invariants of links with $\SL_2(\C)$ flat connections}
\author[C. Blanchet]{Christian Blanchet}
\address{Univ Paris Diderot,  IMJ-PRG, UMR 7586 CNRS,  F-75013, Paris, France} 
\email{christian.blanchet@imj-prg.fr}

\author[N. Geer]{Nathan Geer}
\address{Mathematics \& Statistics\\
  Utah State University \\
  Logan, Utah 84322, USA}
\thanks{This work is supported by the NSF FRG Collaborative Research Grant DMS-1664387. The research of Blanchet was supported by 
 French ANR project ModGroup ANR-11-BS01-0020.  
Research of Geer was partially supported by 
NSF grant  DMS-1452093.  Geer and Reshetikhin would like to thank 
Institut Math\'ematique de Jussieu, Paris, France for its generous hospitality. 
  All the authors would like to thank the Erwin Schr\"odinger Institute for Mathematical Physics in Vienna for support during a stay in the Spring of 2014.
 }\
\email{nathan.geer@gmail.com}

\author[B. Patureau-Mirand]{Bertrand Patureau-Mirand}
\address{UMR 6205, LMBA, universit\'e de Bretagne-Sud, universit\'e
  europ\'eenne de Bretagne, BP 573, 56017 Vannes, France }
\email{bertrand.patureau@univ-ubs.fr}
  \author{Nicolai Reshetikhin}
\address{Department of Mathematics\\
University of California, Berkeley\\
970 Evans Hall \#3840\\
Berkeley, CA 94720-3840, USA\\
and\\
KDV Institute for Mathematics\\
Universiteit van Amsterdam, \\
Plantage Muidergracht 24\\
1018 TV, Amsterdam, The Netherlands}
\email{reshetik@math.berkeley.edu}

\begin{abstract}
R.\ Kashaev and N.\ Reshetikhin introduced the notion of \emph{holonomy braiding} extending V. Turaev's homotopy braiding to describe the behavior of cyclic representations of the unrestricted quantum group $\Uq$ at root of unity.  In this paper, using quandles and biquandles we develop a general theory for Reshetikhin-Turaev ribbon type functor for tangles with quandle representations.  This theory applies to the unrestricted quantum group $\Uq$ and produces an invariant of links with a gauge class of quandle representations.
\end{abstract}

\maketitle
\setcounter{tocdepth}{3}
\tableofcontents

\section{Introduction}

Following Jones and Reshetikhin-Turaev, it has been known for three decades that invariants of links and tangles can be obtained from quantum groups.
The relevant structure is that of a ribbon category and categories of modules over 
the restricted or unrolled  
 quantum groups are endowed with such structure, see for example \cite{Tu}.  This paper focus on defining tangle invariants from categories which are not ribbon but have a more general structure based on the so called \emph{holonomy braiding.}

  An important and well-studied
example of a category that is not ribbon comes from the non-restricted quantum group associated to a simple Lie algebra. 
 Such categories have been considered by a number of people including De~Concini, Kac,  Procesi, Reshetikhin, Rosso and others (see for example \cite{DK, DKP, DKP2, DPRR} and references within).  These categories have modules with vanishing quantum dimensions.  To define suitable invariants of tangles coming from such categories one needs additional algebraic and topological structures.   
  
In \cite{KR}, Kashaev and Reshetikhin generalize the Reshetikhin-Turaev link invariant construction to $G$-tangles:  tangles with a flat connection in a principal $G$-bundle over the complement of the tangle.  
     Their construction is based on the notion of a holonomy braiding: for certain pairs of objects $(V,W)$ in  a category $\cat$ there exists a \emph{holonomy braiding} $V\otimes W\to W^g \otimes V^g$ where $V^g$ and $ W^g$ are objects of $\cat$ (determined by $V, W$ and $G$) that may not be isomorphic to $V$ and $W$, respectively.  This map is meant to replace the braiding and can be represented in a diagram by: \; $ \put(-8,-6){{\footnotesize $V$}} \put(14,-6){{\footnotesize $W$}} \put(-9,6){{\footnotesize $W^g$}} \put(14,6){{\footnotesize $V^g$}}   \epsh{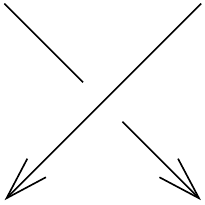}{12pt}$\;\;\;\;.  Thus, they consider diagrams where the coloring of the upper and lower strand of a crossing can both change.   Such transformations are accounted for by connections in the $G$-bundle and can be interpreted geometrically in terms of holonomies of the $G$-connection along certain paths (see \cite{KR}).  

The theory of \cite{KR} addresses the absence of the braiding but does not lead to a $G$-link invariant in the case of the non-restricted quantum group because of the following observations: (1) vanishing quantum dimensions force the original construction of \cite{KR} to be zero for all links and so only leads to invariants of open tangles and (2) the braiding is not defined on the whole category but instead on generic pairs of simple modules.   The theory of modified traces given in \cite{GPV, GP4} can be used to overcome the first observation.  As we will see, the second observation can be tackled with the use of gauge transformations.   

To describe coloring of tangles and their diagrams we use (bi)quandles.  
A (bi)quandle 
is an algebraic structure whose definition follows the Reidemeister
moves between classical knot diagrams.  The fundamental quandle of a
knot is an invariant which is stronger that the fundamental
group. Indeed, it was shown independently by Joyce \cite{Joyce} and
Matveev \cite{Ma} that the fundamental quandle classifies knots, while
the fundamental group does only for prime knots. The fundamental
quandle has both intrinsic or diagram based definition.  A biquandle is
a more general structure which can be used to study knots and links
via their diagrams.  It was shown by Lebed and Vendramin \cite{LV} 
that a biquandle produce a quandle structure in such a way 
that representations of the fundamental quandle can be described
combinatorially with the biquandle. In this situation the biquandle
will be called a factorization of its associated quandle.

 We now describe the main results of this paper. Let $\QX$
be a quandle, then a tangle together with a representation of its
fundamental quandle to $\QX$ will be called a $\QX$-tangle. We define
the notion of representation of a biquandle, which includes a
holonomic braiding.  Our first main theorem \ref{T:FunctorF} states the
existence of a ribbon functor on the category of $\QX$-tangle,
provided we have a representation of a biquandle factorization of
$\QX$.  This can be applied to the {\em semi-cyclic} version of
quantum $\slt$. For non restricted quantum $\slt$, formulas for the
biquandle involve rational functions on central characters which are
only generically defined.  We extend our theory to this setting; with
an appropriate notion of a representation of a generically defined
biquandle and using gauge action, we obtain similar ribbon functor.
Our second main theorem \ref{T:DefF'Generic} is about invariants of
links.  The above functor vanishes on $\QX$-links. We show that a
normalization using {\em modified dimension} produces an invariant
of gauge classes of $\QX$-links.  We apply the generically defined
theory to {\em cyclic representations} of non restricted quantum 
$\slt$: we compute the biquandle,
describe the associated quandle $\QX$ in relation with the conjugation
quandle of $\SL_2(\C)$ and show that we obtain an invariant of gauge
classes of $\QX$-links which is a complex number modulo a root of
unity.
 
 The existence of such an invariant is the first step in constructing a type of Homotopy Quantum Field Theory \cite{Tu2010, TuV2014} in a new context modeled on the representation theory of the non-restricted quantum group.  In particular, we hope the invariants given in the main example of this paper extend to invariants of 3-manifolds endowed with a flat $\SL(2,\C)$ connection.  Such invariants  ---and more generally invariants with flat $G_\C$ connection, where $G_\C$ is a complex simple Lie group ---are highly amenable to a physical formulation.  Indeed, quantum invariants of this type are prime candidates for the analytic continuation of Chern-Simons theory with compact gauge group $G$ \cite{Gu,W-anal}. 
 
The paper is organized as follows.  In Section \ref{S:Quandles} we
describe the intrinsic topological notion of our invariants in terms
of the fundamental quandle.  In Sections \ref{S:Biquandles} and \ref{S:BiquandleBraidings} we
formulate the algebraic setting underlying of our paper by considering
biquandles and their representations in a pivotal category.  In
Section \ref{S:GenericBiqu} we generalize the notions of Sections
\ref{S:Quandles}--\ref{S:BiquandleBraidings} to the setting of generically
defined biquandles in pivotal categories where the holonomic braiding
is not necessarily defined everywhere.  Section \ref{S:MainExCycle}
contains our main example coming from the cyclic modules of the
unrestricted quantum group of $\slt$.  The Appendices contain some
proofs.
\section{Quandles}\label{S:Quandles}
In this section we will recall some basic definitions and facts about
quandle \cite{Joyce,EN_book}.   
Quandles turns out to be the correct topological setting for this paper; they allow us to define intrinsic topological objects without working with diagrams. 
We consider links and tangles
with a representation of their fundamental quandle.
The motivating example is the conjugacy quandle associated to
$\SL_2(\C)$.  Using representation of the fundamental quandle over
this quandle, we will recover tangles with a flat connection in a
principal $\SL_2(\C)$-bundle over the complement of the tangle.

\subsection{Basic definitions}
\begin{definition} A \emph{quandle} is a set $\QX$ with a binary operation $(a,b)\to a \rhd b$  such that
\begin{enumerate}
\item \label{I:Qdist} for all $a,b,c\in \QX$, $a\rhd (b\rhd c)=(a\rhd b)\rhd(a\rhd c)$,
\item \label{I:Qmap!} for all $a,b\in \QX$ there is a unique $c\in \QX$ such that  $a=b\rhd c$,
 \item for any $a\in \QX$, $a\rhd a=a$.
\end{enumerate}
\end{definition}
A function $f:\QX \to \QX'$ between quandles is a \emph{homomorphism}
if $ f(a \rhd b)=f(a) \rhd f(b)$ for all $a,b\in \QX$.  For each
$a\in \QX$, Axioms \eqref{I:Qdist} and \eqref{I:Qmap!} imply that the
map $ (a\rhd \any):\QX\to \QX, $ given by $x\mapsto a\rhd x$ is
bijective homomorphism, where the symbol $\any$ is used to denote a
variable.  A \emph{subquandle} of a quandle $(\QX, \rhd)$ is a subset
$\mathsf{P}$ of $\QX$ such that the restriction of $\rhd$ to
$\mathsf{P}$ defines a quandle structure.

\begin{example}[The conjugacy quandle of a group]\label{E:ConjQuandle}
  Let $\Gd$ be a group.  The binary operation $g \rhd h=g^{-1}hg$
  defines a quandle structure on $\Gd$,  denote this quandle as
  $\Conj(\Gd)$.  Any subset of $\Gd$ that is closed under such conjugation
  is subquandle of this quandle.

  We are particularly interested in the case when $\Gd$ is equal to
  $\SL_2(\C)$: the set of $2\times 2$ matrices over $\C$ with
  determinant 1.  Here the quandle structure of $\Conj(\SL_2(\C))$ is
  given by conjugation of matrices:
  $$M\rhd M'=M^{-1}M'M.$$
  The quandle $\SL_2(\C)$ has a self-evident subquandle: 
  Its subset of non parabolic elements, i.e.\ matrices with traces not
  equal to $\pm2$.
\end{example}

\subsection{The fundamental quandle}
The fundamental quandle of a knot was first defined by Joyce \cite{Jo}
and Matveev \cite{Ma}.  In this subsection we give the definition of
the fundamental quandle of a tangle.

Here we use the contravariant concatenation of paths in which
$\gamma.\delta$ is defined when $\gamma(1)=\delta(0)$.  

A \emph{standard tangle} is an oriented framed tangle
${\Gf}\subset(0,+\oo)\times\R\times[0,1]$
 such that
$\partial{\Gf}={\Gf}\cap\R^2\times\{0,1\}=
\bp{\{1,\ldots,p\}\times\{0\}\times\{0\}}\cup\bp{\{1,\ldots,q\}\times\{0\}\times\{1\}}$
and ${\Gf}$ intersect $\R^2\times\{0,1\}$ transversally.  Let ${\Gf}$ be a
standard tangle, fix a base point $*$ in $\{0\}\times\R\times[0,1]$
and let $M_{\Gf}=\left(\R^2\times[0,1]\right)\setminus
{\Gf}$.  
Let $Q({\Gf},*)$ be the set of homotopy classes of continuous paths
$\gamma:[0,1)\to M_{\Gf}$ such that $\gamma(0)=*$ and $\lim_{t\to1}\gamma(t)$
exist and is equal to some point of the tangle $ {\Gf}$.  Loosely speaking,
$Q({\Gf},*)$ is the set of homotopy classes of paths from $*$ to points close
to ${\Gf}$.

To define the quandle structure we consider the \emph{augmentation} map
\begin{equation}
  \label{eq:hat}
  \wh{{\;\;}} :Q({\Gf},*)\to\pi_1(M_{\Gf},*) \text{ given by } 
a=[\gamma]\mapsto \wh a=[\gamma_\ve .m.\gamma_\ve^{-1}]
\end{equation}
where $m$ is a positive meridian of ${\Gf}$ around
$\lim_{t\to1}\gamma(t)$ and $\gamma_\ve=\gamma_{|[0,1-\ve]}$ for some
small $\ve$.  There is a left action of $\pi_1(M_{\Gf},*)$ on
$Q({\Gf},*)$ given by concatenation of the paths.
  \begin{lemma}
  The set $Q({\Gf},*)$ has a quandle structure defined by  
  the augmentation map and concatenation of paths: $$a\rhd b=\wh a^{-1}.b\,.$$
 \end{lemma}
 \begin{proof}
   A direct calculation using the definition shows
   $\wh {a\rhd b}=\wh a^{-1}.\wh b.\wh a$.  Using this it is easy to see that
   $\rhd$ defines a quandle structure.
 \end{proof}
 The pair $(Q({\Gf},*),\rhd)$ is called the \emph{fundamental
   quandle} of ${\Gf}$.
\begin{figure}
 \centering
 {\color{red}
   \[\begin{array}[t]{c} \epsh{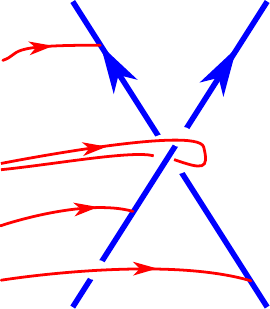}{12ex}
     \put(-67,17){$a\rhd b$}\put(-52,0){$\wh
       a$} \put(-52,-10){$a$}\put(-52,-20){$b$}
     \end{array}\]
   \hspace{5ex}\ }
\caption{Quandle structure of the fundamental quandle.  There are many
  choices of convention made in this representation: base point on the
  left, ``contravariant'' composition in the fundamental groupoid,
  positive/negative meridian, upward orientation and the choice of
  what is a positive crossing.}
  \label{fig:FundQuandleDiag}
\end{figure}
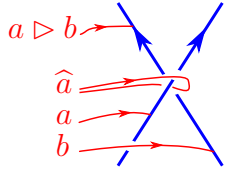

 \subsection{$\QX$-tangles}
\begin{definition}
  Let $(\QX,\rhd)$ be a quandle.  A representation of the fundamental quandle, or \emph{$\QX$-tangle} for short, is a  standard 
 tangle ${\Gf}$ together with a quandle morphism
  $\rho:Q({\Gf},*)\to (\QX,\rhd)$. 
\end{definition}

To gain some feeling about quandles we have the following remark:

\begin{remark} \label{R:QtanGtan}
  (A) \textbf{(The fundamental quandle inside the fundamental group)}
  Let $\Gf$ be a link then we can consider the conjugacy quandle of
  the fundamental group $\Conj( \pi_1(M_{\Gf},*)) $.  The augmentation
  map $a\mapsto \wh a$ is a quandle morphism which is injective if and
  only if  a $\Gf$ is prime link (the two parallel arcs realizing a
  connected sum produce quandle elements with the same image in the
  fundamental group, see \cite[Theorem 3.5]{Ry}).  In this situation we
  can identify $Q({\Gf},*)\subset \Conj(\pi_1(M_{\Gf},*))$.

(B) \textbf{ ($ \QX$-tangles generalize $\Gd$-tangles)} The associated functor
  $\Conj$ from groups to quandles has a left adjoint $\Env$ that
  associates to a quandle $(\QX,\rhd)$ its envelope: the group freely
  generated by $\QX$ modulo the relations $a\rhd b=a^{-1}ba$. It is
  well known that the envelope of the fundamental quandle is the
  fundamental group, so when $\QX=\Conj(\Gd)$, we have
  $$\rho\in\Hom_{\text{quandle}}(Q({\Gf},*),\Conj(\Gd))
  \cong\Hom_{\text{group}}(\pi_1(M_{\Gf},*),\Gd)$$
  and a $ \QX$-tangle is just a $\Gd$-tangle.  Indeed, if
  $\hat\rho\in\Hom_{\text{group}}(\pi_1(M_{\Gf},*),\Gd)$ then its
  pre-composition with the augmentation map of Equation \eqref{eq:hat}
  is a quandle structure on $\Gamma$.
  In particular, if the
  quandle $\QX$ is a normal sub-quandle of $\Conj(\Gd)$ (that is
  $(\Gd\rhd \QX)\subset \QX$) then a $\QX$-tangle is the same as a
  $\Gd$-tangle $(\Gf,\rho)$ (see \cite{KR}) whose representation
  $\hat\rho:\pi_1(M_{\Gf},*)\to \Gd$ sends the meridians of $\Gf$ into
  $\QX$.
 \end{remark}

\begin{example}
\newcommand{\Hol}{\operatorname{Hol}}
Let $\QX=\Conj(\SL_2(\C))$.  If the complement $M_\Gf$ of a tangle
$\Gf$ is equipped with a
flat $\SL_2(\C)$-bundle trivialized at the base point, the holonomy
map $\Hol:\pi_1(M_\Gf,*)\to\SL_2(\C)$ is equivalent to a quandle
structure $\rho:Q({\Gf},*)\to\QX$.  Here the trivialization at
$*\in M_\Gf$ gives for any loop $\gamma$ in $M$ from $*$ to $*$ a
canonical lift $\wt \gamma$ in the bundle from
$\wt\gamma(0)=I_2\in\SL_2(\C)$ to $\Hol([\gamma]):=\wt\gamma(1)$.
This gives a one to one correspondence between $\QX$-tangles and
diffeomorphism class of such bundles.  Furthermore, changing the
trivialization of the bundle at the base point by a translation by
$x\in\SL_2(\C)$ correspond to changing the quandle map $\rho$ with
$x\rhd\rho$.  Hence gauge classes of $\QX$-tangle (defined in next
section) are in one to one correspondence with diffeomorphism classes
of locally flat $\SL_2(\C)$-bundles over $(M_\Gf,*)$.
\end{example}

Given a set $X$ let $W_X$ be the free monoid generated by pairs $(x,\epsilon)$
where $x\in X$ and $\epsilon\in \{+,-\}$.
If $w=(x_1,\ve_1)\cdots(x_p,\ve_p)\in W_X$, set
$-w=(x_1,-\ve_1)\cdots(x_p,-\ve_p)\in W_X.$ Remark that the set $X$ can be
identified with $\{(x,+):x\in X\}\subset W_X$.

A standard $\QX$-tangle $(\Gf,\rho)$ determines a word in $W_\QX$ which depends
on its bottom boundary: We first fix paths $\gamma_i$ which are above $\Gf$ going from the base point to
the $i$\textsuperscript{th} bottom edge $e_i$, i.e. 
$\gamma_i\subset\{x=0\}\cup\{z=0\text{ and }y<0\}$ and
$\lim_{t\to1}\gamma_i(t)=(0,i,0)$. Then
$$\partial_-\Gf=(\rho(\gamma_1),\ve_1)\cdots(\rho(\gamma_p),\ve_p)\in W_\QX$$
where
$\ve_i$ is a sign which is $+$ if
the orientation of $e_i$ is incoming and $-$ if it is outgoing.  Similarly,
one can define a word $\partial_+\Gf$ associated to the top boundary.

\begin{definition}
  The category $\Tang_\QX$ of \emph{$\QX$-tangles} is defined
  as follows.  An object in $\Tang_\QX$ is an element of $W_\QX$.  A morphism
  $\Gf:w\to w'$ is an isotopy class of a $\QX$-tangle $\Gf$ such that
  $\partial_+ \Gf=w'$ and $\partial_- \Gf=w$.

\end{definition}
As usual, the composition is given by gluing standard tangles along their top
or bottom boundaries.  An argument similar to Van Kampen's theorem ensures that
the structure of two $\QX$-tangles can be glued (we joint the base points by a path in
the plane $x=0$.  Similarly the disjoint union of $\Gf_1$ and $\Gf_2$ gives
the monoidal structure (we joint the base point of $\Gf_2$ to the base point
of $\Gf_1$ by a path above $\Gf_1$ so that
$\partial_\pm(\Gf_1\sqcup \Gf_2)=\partial_\pm(\Gf_1).\partial_\pm(\Gf_2)\in
W_\QX$).

The regular planar projection in the $(Oy)$ direction of a $\QX$-tangle $\Gf$
is a standard planar diagram $D$ (a planar oriented graph embedded in
$(0,+\oo)\times[0,1]$ with univalent vertices on the boundaries
$(0,+\oo)\times\{0,1\}$ and four-valent vertices with the overcrossing
information).  The
quandle morphism $\rho:Q({\Gf},*)\to (\QX,\rhd)$
gives a coloring of each
edge $e$ of $D$ by the element $\rho(\gamma_e)$ where $\gamma_e$ is a path
above $\Gf$ from the base point to a point of $\Gf$ which project on the edge.
At a crossing, the four pathes $\gamma_e$ are related in $Q(\Gf,x)$ by the
relation of Figure \ref{fig:FundQuandleDiag}.  Hence the $\QX$-colors of the
edges of $D$ form a \emph{quandle coloring}, 
that is the colors in a neighborhood of a crossing satisfy the
relationship given in Figure \ref{fig:Q-col}.  Remark that the two
edge forming the overcrossing strand have the same color.
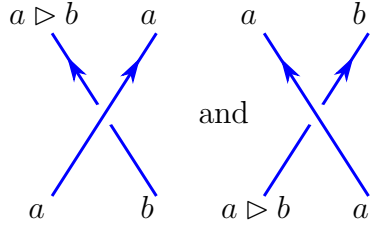
\begin{figure}
  \centering
  $$
  \epsh{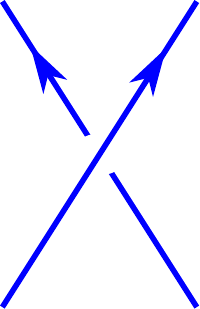}{12ex}\put(-40,-30){$a$}\put(-5,-30){$b$}
  \put(-46,31){$a\rhd b$}\put(-5,31){$a$}\quad\text{ and
  }\epsh{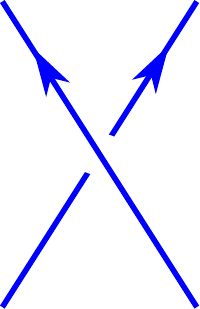}{12ex}\put(-46,-30){$a\rhd b$}\put(-5,-30){$a$}
  \put(-40,31){$a$}\put(-5,31){$b$}$$
  \caption{$\QX$-coloring around a crossing}
  \label{fig:Q-col}
\end{figure}
Reciprocally, the Wirtinger-like presentation of $Q(\Gf,*)$ implies
that a quandle morphism $\rho:Q({\Gf},*)\to (\QX,\rhd)$ is uniquely
determined by the quandle $\QX$-coloring of a regular projection of
$\Gf$.  One easily see that the elementary diagrams involved in
Reidemeister moves (see a generating set of these moves in Figure
\ref{fig:RM}) have their coloring uniquely determined by their
boundaries and this lead to the notion of $\QX$-colored Reidemeister
move for $\QX$-colored diagrams.  
\begin{figure}
  \centering
  $$
  \begin{array}{ccc}
    \epsh{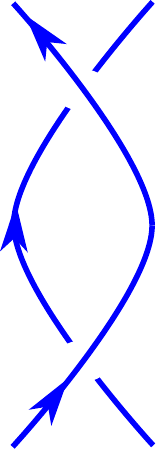}{12ex}\leftrightarrow\epsh{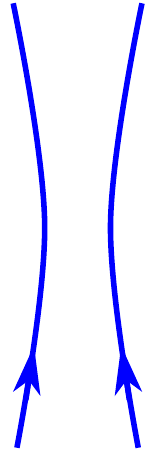}{12ex}
    \leftrightarrow\epsh{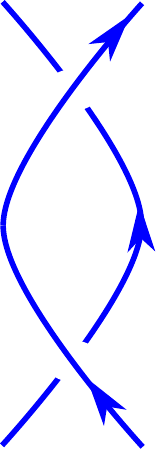}{12ex}&\hspace{3ex}
    \epsh{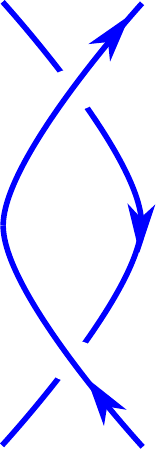}{12ex}\leftrightarrow\epsh{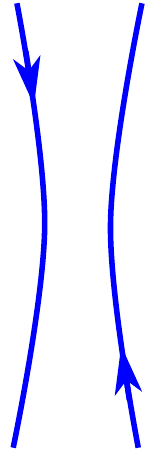}{12ex}&\hspace{3ex}
    \epsh{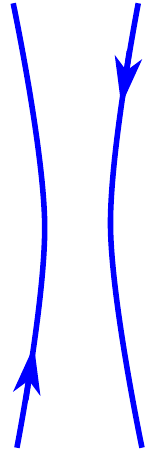}{12ex}\leftrightarrow\epsh{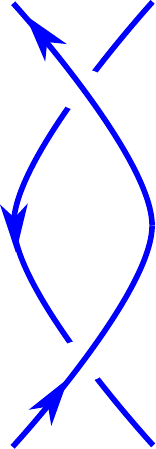}{12ex}\\
    RII_{++}&RII_{-+}&RII_{+-}
  \end{array}
  $$
  $$
  \begin{array}{cc}
    \epsh{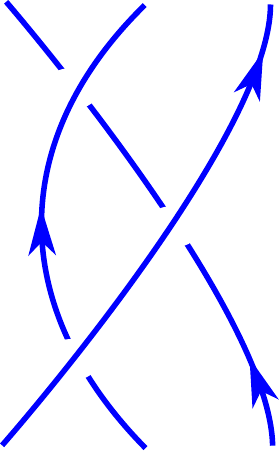}{12ex}\leftrightarrow\epsh{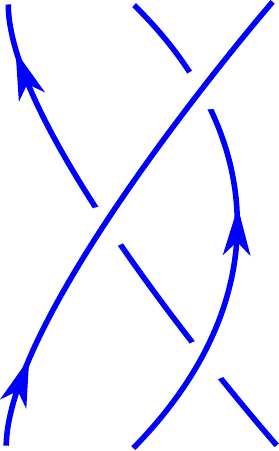}{12ex}&\hspace{3ex}
    \epsh{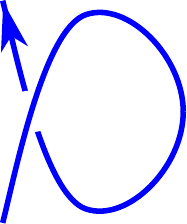}{8ex}\leftrightarrow\epsh{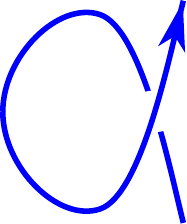}{8ex}\\
    RIII_{+++}&\hspace{3ex}RI^f
  \end{array}
  $$

\caption{{Oriented framed Reidemeister moves.  We call $RII_{++}$ and $RIII_{+++}$ \emph{positive} and $RII_{-+}$ and $RII_{+-}$ \emph{negative}.}}
  \label{fig:RM}
\end{figure}

We define a category $\Diag_{\QX}$ of planar diagrams up to planar isotopy:
The objects of $\Diag_{\QX}$ are the same as those of $\Tang_\QX$ and
morphisms of $\Diag_{\QX}$ are formed by $\QX$-colored diagram.
Both categories $\Diag_{\QX}$ and $\Tang_\QX$ are pivotal with duality given by
standard cup and cap like diagrams and tangle.  In particular, the dual of
$w=(x_1,\ve_1)\cdots(x_p,\ve_p)\in W_X$ is $w^*=(x_p,-\ve_p)\cdots(x_1,-\ve_1)$.
\begin{theorem}\label{T:QD}
  The natural surjective pivotal functor $\Diag_{\QX}\to\Tang_\QX$
  induces bijections
  $\Diag_{\QX}/\equiv\quad\stackrel\cong\longrightarrow\quad\Tang_\QX$
  where $\equiv$ is the equivalence relation on diagrams generated by
  colored Reidemeister moves.  
\end{theorem}
\begin{proof}
  This follows from the standard Reidemeister theorem which ensures that any
  isotopy of $\Gf$ translates into a finite sequence of framed Reidemeister
  moves that can be upgraded to $\QX$-colored framed Reidemeister moves
  between $\QX$-colored diagrams.
\end{proof}
\subsection{Gauge action}\label{SS:GaugeActionBiq}
Let $\vp:\QX\to\QX'$ be a quandle map (i.e.
$\vp(a\rhd b)=\vp(a)\rhd\vp(b)$).  Then post composing
a $\QX$-tangle $\rho: Q({\Gf},*)\to (\QX,\rhd)$
  with $\vp$ gives a
$\QX'$-tangle $\vp\circ\rho: Q({\Gf},*)\to (\QX',\rhd)$.
Similarly, changing all colors of a
$\QX$-diagram to their image by $\vp$ produces a $\QX'$-diagram.  These assignments induce functors  which we still denote by $\vp$:
\begin{equation}
  \label{eq:Q-functor}
  \vp:\Tang_{\QX}\to\Tang_{\QX'}\et \vp:\Diag_{\QX}\to\Diag_{\QX'}.
\end{equation}
These functors clearly commute with the functors of Theorem \ref{T:QD}.

In particular, the self distributivity of the quandle operation
implies that for any $a\in\QX$, $a\rhd\any$ is an automorphism of
$\QX$.  So if $\Gf$ is a $\QX$-tangle with
representation
$\rho:Q(\Gf,*)\to\QX$ and $b\in\QX$ then
$b\rhd\rho:\gamma\mapsto b\rhd\rho(\gamma)$ is also a
representation
on $\Gf$ and the action of $b\in\QX$ on a $\QX$-colored diagram $D$ is
just given by changing the color $c$ of every edge of $D$ with
$b\rhd c$.  We call this functorial action a \emph{gauge
  transformation} by $b$.

\section{Biquandles versus quandles}  \label{S:Biquandles}
Here we study a generalization of a quandle called a biquandle (for
more on biquandles see \cite{FJK, FRS93, CJKLS}).  Our motivation for studying
biquandles come from unrestricted quantum groups.  Kashaev and the
last author showed that even though the representation of the unrestricted
quantum group at root of unity is not braided, one can still define an 
outer automorphism $R$-matrix inducing
a holonomic braiding
 between some
simple modules.  We will see that these maps give a canonical
structure of a biquandle on the set of isomorphism classes of simple
projective modules.

As for quandles, biquandles were studied in classical knot theory.  In
particular, the axioms for a quandle are designed to allow colored
Reidemeister moves.  With biquandle coloring, all four colors
corresponding to the edges at a crossing can be different.  The algebraic
structure underlying the topological invariants we define later fit
into this setting.  Thus, biquandles allow us to consider colored
diagrams corresponding to our desired algebraic setting.  The main
point is that quandle colored tangles (topology) correspond to biquandle colored
diagrams of the tangle up to colored Reidemeister moves (algebra).

\subsection{Basic definitions}  
We will use the symbol $\any$ as a shortcut for denoting a variable
(``$f(\any)$'' means ``$x\mapsto f(x)$'') and if a map
$F:X\to Y\times Z$ has values in a cartesian product, for $i=1,2$,
$F_i$ is the $i$\textsuperscript{th} component of $F$.

\begin{definition}\label{D:biquandle} A \emph{biquandle} is a set $X$ with a bijective map
  $B=(B_1,B_2): X\times X\rightarrow X\times X$ which satisfies the following
  axioms: 
  \begin{enumerate}
  \item \label{DI1:biquandle} The map $B$ satisfies the set Yang-Baxter equation
    $$(\Id\times B)\circ(B\times \Id)  \circ(\Id\times B)=
    (B\times \Id) \circ(\Id\times B)\circ(B\times \Id).$$
  \item \label{DI2:biquandle} The map $B$ is sideways invertible: there exists a
    unique bijective map $S:X\times X\rightarrow X\times X$ such that
    $$S(B_1(x,y),x)=(B_2(x,y),y)$$
    for all $x,y\in X.$
  \item \label{DI3:biquandle} The map $S$ induces a bijection $\alpha:X\to X$
    on the diagonal:
    $$S(x,x)=(\alpha(x),\alpha(x))$$ for all $x\in X.$
     \end{enumerate} 
 \end{definition}
 \begin{remark}
   Axiom \ref{DI2:biquandle} is equivalently written as: for any
   $x\in X$, the maps $B_1(x,\any)$ and $B_2(\any,x)$ are bijections.
   Relaxing \ref{DI3:biquandle} gives a more general structure
   called a birack, for which a similar theory can be developed.
\end{remark}

In the rest of this section, let $(X,B)$ be a biquandle.  Let $D_\Gf$
be a regular projection of a standard tangle $\Gf$.  A
\emph{$X$-coloring} of $D_\Gf$ is a coloring of its edges by elements
of $X$ satisfying compatibility condition given by the biquandle $B$
for the four edges incident to any crossing as in Figure
\ref{fig:X-col}.
  \begin{figure}
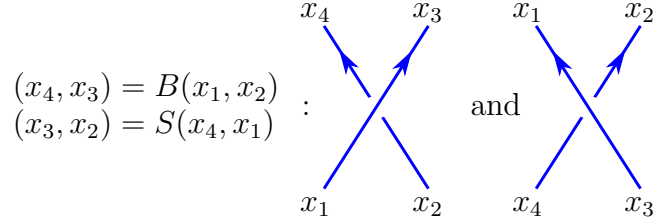

  \centering
  $$
  \begin{array}[c]{l}
    (x_4,x_3)=B(x_1,x_2)\\
    (x_3,x_2)=S(x_4,x_1)
  \end{array}
  :\epsh{fig1}{12ex}\put(-40,-30){$x_1$}\put(-5,-30){$x_2$}
  \put(-40,31){$x_4$}\put(-5,31){$x_3$}\quad\text{ and
  }\epsh{fig2}{12ex}\put(-40,-30){$x_4$}\put(-5,-30){$x_3$}
  \put(-40,31){$x_1$}\put(-5,31){$x_2$}$$
  \caption{$X$-coloring of a crossing}
  \label{fig:X-col}
\end{figure}

As in the case of quandle colored tangles, an $X$-coloring of $D_\Gf$
determines two words $\partial_-\Gf$ and $\partial_+\Gf$ in $W_X$.   On the other hand,  the axioms of a biquandle insure that any word associated to the top (resp.\ bottom) boundary of an oriented braid diagram can be extended uniquely to a $X$-coloring of the whole braid.  

\begin{definition}
  The category $\Diag_X$ of $X$-colored diagrams is defined
  as follows.  An object in $\Diag_X$ is an element of $W_X$.  A morphism
  $D:w\to w'$ is a planar isotopy class of a $X$-colored diagram $D$ such that
  $\partial_+ D=w'$ and $\partial_- D=w$.
\end{definition}
It is well known that the category of link diagrams colored by a ribbon category is generated by crossings, cups and caps (see Lemma 3.1.1 of \cite{Tu}).    The following lemma is the analogous lemma in the setting of $X$-colored diagrams.  
\begin{lemma}\label{L:cat_diag}
  As a tensor category,
  $\Diag_X$ is generated by the six families of elementary tangles:
  \begin{enumerate}
  \item positive crossing
    $\chi^+_{x_1,x_2}:(x_1,+)(x_2,+)\to
    (B_1(x_1,x_2),+)(B_2(x_1,x_2),+)$
    where $x_1,x_2\in X$ (see Figure \ref{fig:X-col}),
  \item negative crossing
    $\chi^-_{x_1,x_2}:(B_1(x_1,x_2),+)(B_2(x_1,x_2),+)\to
    (x_1,+)(x_2,+)$ where $x_1,x_2\in X$ (see Figure \ref{fig:X-col}),
  \item left evaluation: $\epsh{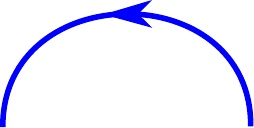}{2ex}=\lev_x:(x,-)(x,+)\to\emptyset$,
  \item left coevaluation: $\epsh{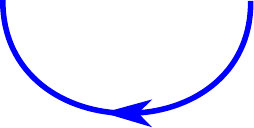}{2ex}=\lcoev_x:\emptyset\to(x,+)(x,-)$,
  \item right evaluation: $\epsh{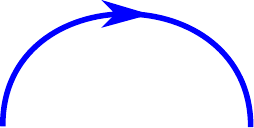}{2ex}=\rev_x:(x,+)(x,-)\to\emptyset$,
  \item right coevaluation: $\epsh{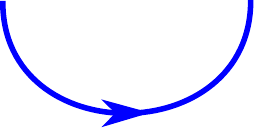}{2ex}=\rcoev_x:\emptyset\to(x,-)(x,+)$.
\end{enumerate}
\end{lemma}
\begin{proof}
Up to planar isotopy, one can put all crossing upward.
\end{proof}

\begin{definition}{\bf $X$-colored Reidemeister moves.} 
  Let $D$ and $D'$ be $X$-colored diagrams whose underlying diagrams
  are related by a framed Reidemeister move.  We say that they are
  related by an {\em $X$-colored framed Reidemeister move} if the color assigned to an unmodified edge of $D$ and $D'$ is the same.   
\end{definition}
\begin{remark}\label{r:notBiquandle}
  The axioms of a biquandle are not used in the definition of
  $\Diag_X$ and in the notion of  a $X$-colored Reidemeister move (only the existence of $B$).
  The map $B$ could even be replaced by an arbitrary subset
  $\{(x_1,x_2,x_3,x_4)\}$ of $X^4$ used to define coloring of
  crossings as in Figure \ref{fig:X-col}.
\end{remark}
The theory of biquandles insures that any framed Reidemeister move from a
$X$-colored diagram gives rise to a unique colored Reidemeister move:
\begin{lemma}\label{L:isotopy}\
  \begin{enumerate}
  \item Each framed Reidemeister move 
  determines
   a canonical
    bijection between $X$-colorings of the two diagrams in the move.
  \item Let $T$ be a standard tangle with a regular projection $D$.  Any generic isotopy $f_t$ of $T$     induces a sequence of framed Reidemeister moves 
  $$D=D_0\to
    D_1\to\cdots\to D_n=D'.$$
      The associated bijection between
     colorings of $D$ and $D'$ only depends of the homotopy
    class of $f_t$.
  \end{enumerate}
\end{lemma}
\begin{proof}
  The first statement follows directly from the axioms of a biquandle.  
For  the second statement, let $(f_{t,s})_{t,s\in[0,1]}$ be a generic homotopy of isotopies.   The values $(t,s)$ for which the projection of $f_{t,s}(T)$ is not a
  regular planar diagram form a finite graph in $[0,1]\times [0,1]$
  whose edges correspond to framed Reidemeister moves. 
   It is enough
  to show that the loop of Reidemeister moves around each vertex of this graph (i.e.\ a movie move, see 
  \cite{CS})  induces the identity on
  the set of coherent colorings of the initial diagram.  This follows
  from the fact that 
  the change in any diagram 
  in such a loop is contained in a flat closures of some braid diagram.   Hence, their
  coloring are uniquely (over-)determined by the constant coloring of
  their boundary.
\end{proof}
\subsection{Biquandle factorization of a quandle}

In \cite{LV}, Lebed and Vendramin show that each biquandle can be associated with a quandle structure, they actually prove something more general.  In our setting their ideas imply:
\begin{proposition} Let $(X,B)$ be a biquandle.  For $x,y\in X$ the operation $\rhd$ given by
\begin{equation}
  \label{eq:quandle}
  x\rhd y=B_1(x,S_1(x,y)):\qquad
  \epsh{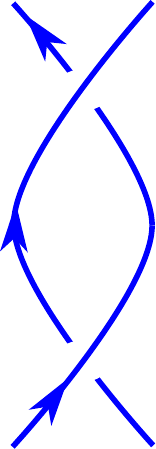}{16ex}\put(-27,-30){$y$}\put(-29,0){$x$}
  \put(-44,30){$x\rhd y$}\put(1,0){$S_1(x,y).$}
\end{equation}
defines a quandle structure on $X$.  
\end{proposition}
\begin{proof}
  This proposition is essentially a special case of Proposition 5.7 of
  \cite{LV}.  However, in \cite{LV} Lebed and Vendramin choose a
  different convention.  So in Proposition 5.7 they define a slightly
  different operation $x\lhd y=\wb B_2(\wb S_2(x,y),y)$.  We change
  this convention to have biquandle maps associated to positive
  crossings and base point on the left.  However, the proof is
  essentially the same: The axiom $x\rhd x=x$ follows from Axiom
  \eqref{DI3:biquandle} of the definition of a biquandle.  Given
  $x,y\in X$ then $z=B_1^{-1}(y,S_2(x,y))$ is the unique element satisfying
  $x=y\rhd z$.  Finally, the Axiom \eqref{I:Qdist} of a definition of
  a quandle follows from an analogous argument using three unknots as
  in Figure 5.4 of \cite{LV}.
\end{proof}
We say $\QX$ is the quandle \emph{associated} to $(X,B)$.  We will denote by $\QX$ the set $X$ (or an isomorphic copy
$\Qm:X\stackrel{\sim}\to \QX$) with the quandle structure $\rhd$.  We
say that the biquandle $(X,B)$ is a {\em biquandle factorization} of the quandle
$(\QX,\rhd)$. 

Let $\QX$ be a quandle with a biquandle factorization $(X,B)$.  Next
we will discuss how $\QX$-colorings are related to $X$-colorings.  Let
$T$ be a standard tangle and let $(D,c)$ be a regular projection of
$T$ with a $X$-coloring $c$.  Recall
$M_{T}=\R^2\times[0,1]\setminus T$ has a base point $*$ to the left of
$T$.  Consider the set $P$ of continuous embedded paths in $M_T$ from
$*$ to $T$, i.e. $\gamma:[0,1]\to M_{T}$ such that $\gamma(0)=*$ and
$\gamma(1) \in {T}$.

We will define a map $f_{(D,c)}: P \to X$.  Given $\gamma \in P$, we
can think of $T\cup \gamma$ as a graph in $\R^2 \times [0,1]$.  Let
$D \cup \gamma$ be a regular projection of this graph such that
$\gamma$ meets $T$ on the left, i.e.\ when looking in the direction of
the orientation of edge of $T$ you see $\gamma$ on the left.  By
``pulling'' a neighborhood of $\gamma$ back to the base point $*$ we
obtain a isotopy $i$ which ends with $\gamma$ being a short line
segment near $*$.  This isotopy can be represented by a series of
colored Reidemeister moves.  After apply this sequence of colored
Reidemeister moves, the path $\gamma$ becomes a short line segment
near $*$ and the edge of $D$ which containing $\gamma(1)$ is colored
by an element $a \in X$.  We set $f_{(D,c)}(\gamma)=a$.  For an
example of how to compute $f_{(D,c)}(\gamma)$ see Figure
\ref{F:ExHowCompute}.
\begin{figure}
  \centering
$$\begin{array}[t]{ccc}
  \epsh{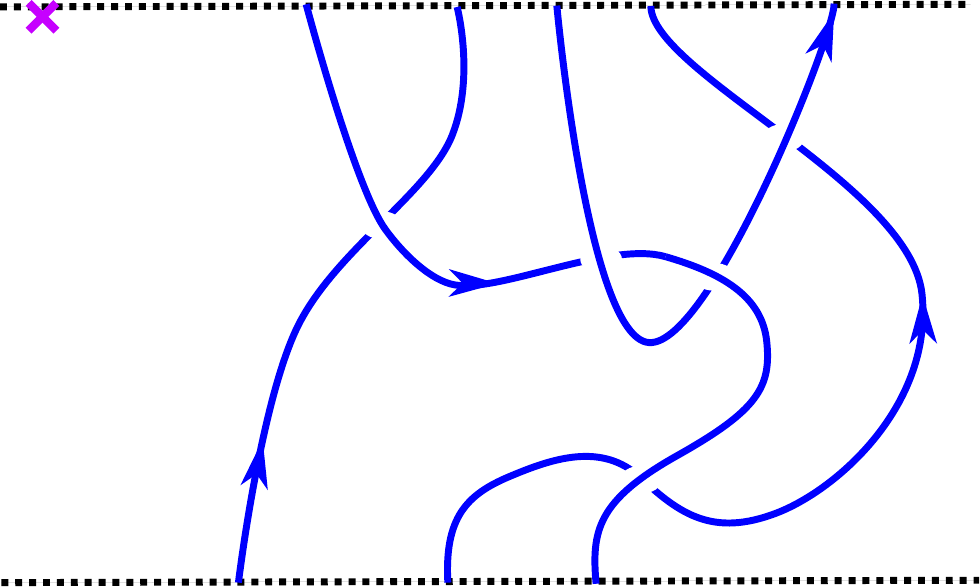}{14ex}\put(-4,0){\ms{x}}\put(-28,-9){\ms{y}}\put(-78,-12){\ms{z}}
&\put(12,12){\ms{\gamma}}\epsh{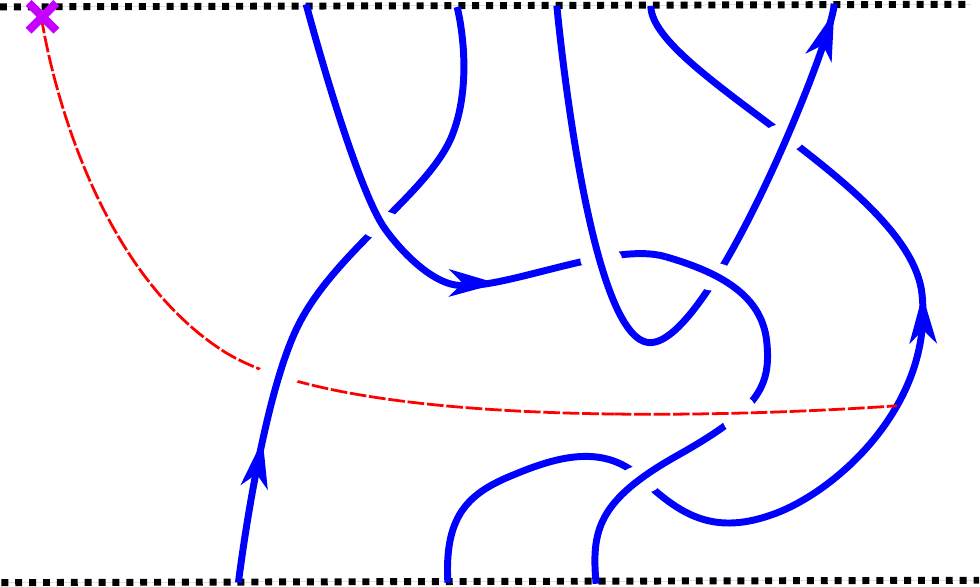}{14ex}
&\epsh{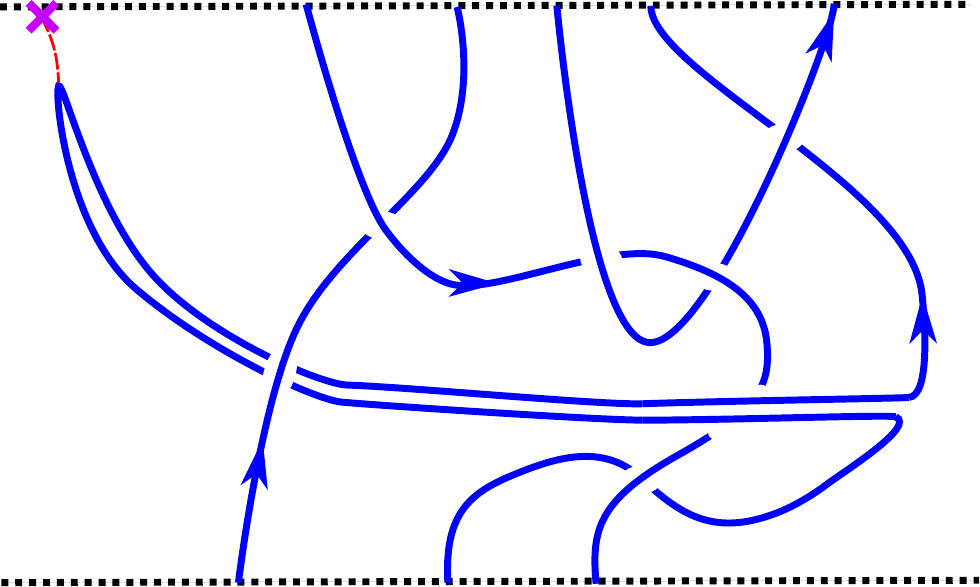}{14ex}\put(-90,17){\ms{x'}}\\&&x'=\ms{B_1(z,S_1(y,x))}
\end{array}$$\hspace{5ex}\ 
  \caption{The first image is an example of a $X$-colored diagram $(D,c)$.  The second image is the diagram with a path $\gamma$.  The third image is the result of the Reidemeister moves showing $f_{(D,c)}(\gamma)=B_1(z,S_1(y,x))$.  Here the $S_1(y,x)$ comes from the passing the edge labeled with $x$ over the edge labeled with $y$ and the $B_1$ appears after passing the edge under the edge colored with $z$.  }
  \label{F:ExHowCompute}
\end{figure}

In \cite{LV}, Lebed and Vendramin defined the notion of a guitar map which makes a
correspondence between colorings given by algebraic structures which are more general
than biquandles and quandles.  Applying this notion to our context and generalizing it 
to the categorical setting, we have the following theorem: 
\begin{theorem}\label{th:LV}
The map $f_{(D,c)}:P\to X$ induces a quandle morphism $f_{D,c}:Q(T,*) \to \QX$ from the fundamental quandle of $T$ to $\QX$.  Moreover, there exists a unique bijective functor
 $$\Qm: \Diag_X \to \Diag_\QX  $$
 defined by $(D,c) \mapsto (D,f_{(D,c)})$.  This functor induces a unique bijective functor
$$\wt \Qm: \Diag_X/\text{colored Reidemeister moves} \to \T_\QX  $$
defined by $(D,c) \mapsto (T,f_{(D,c)})$.
\end{theorem}
\begin{proof}
  First, we will show that for each $(D,c)$ in $ \Diag_X$ the map
  $f_{(D,c)}:P\to X$ extends to a quandle homomorphism
  $ Q(T,*)\to \QX$ which is invariant under colored Reidemeister
  moves.  Given $[\gamma]\in Q(T,*)$, let $\gamma:[0,1)\to M_{T}$ be a framed
  representative of $[\gamma]$ with $\gamma(0)=*$ and
  $\gamma(1)=\lim_{t\to1}\gamma(t)\in {D}$.  Using continuity we can
  view $\gamma$ as an element of $P$ and as explained above can assign
  $f_{(D,c)}(\gamma)\in X$.  We will show this assignment depends only
  on the homotopy class $[\gamma]$.

Suppose $\gamma'$ is another representative  of $[\gamma]$.  Doing the above process we obtain an isotopy $i'$ and a corresponding sequence of colored Reidemeister moves where the path $\gamma'$ becomes a short line segment near $*$ and the edge of $D$ which contains $\gamma'(1)$ is colored by an element $a' \in X$.   We will show $a=a'$.  

Assume first that $\gamma'$ and $\gamma$ are isotopic.  Let $i$ be the isotopy used to define $f_{(D,c)}(\gamma)$ as discussed
above the theorem.  Let $j$ be an isotopy of $\R^2\times [0,1]$ taking
$T\cup \gamma$ to $T\cup \gamma'$.  Now the concatenation of isotopies $i'. j .i^{-1}$ is an isotopy taking the short line segment near $*$ (which is isotopic to
$\gamma$) to the same short line segment (which is isotopic to
$\gamma'$).  Notice that the line segment near $*$ which is isotopic
to $\gamma$ and $\gamma'$ is in a 3-ball.  Thus, there exists a
isotopy $k$ which takes the initial embedding of
$i'.j.i^{-1}$ to its final embedding and does not modify
the 3-ball containing the line segment.  The isotopy  $k$ can be
represented by colored Reidemeister moves.  Since the isotopy  does not change the
3-ball the color of the edge near $*$ is the same before and after any
of these Reidemeister moves.  Moreover, when the points $\gamma(1)$
and $\gamma'(1)$ are pulled back along the path to $*$ one produces a
``doubling'' of the path with the same color but opposite orientations
on the two strands.  If the color of the double stand is $x$ and the
pull back is going to cross an edge labeled with $y$ then the crossing
becomes:
$$
\epsh{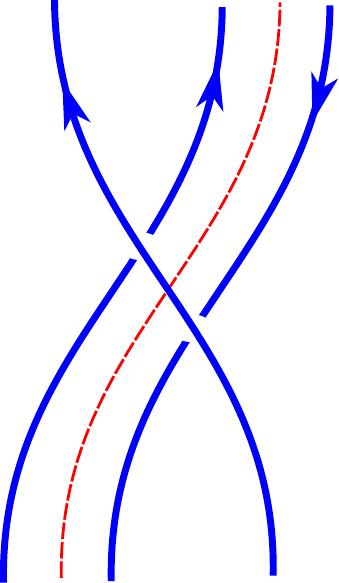}{16ex}\put(-44,-29){\ms{x}}\put(-25,-29){\ms{x}}\put(-12,-29){\ms{y}}
\put(-38,29){\ms{y}}\put(-20,29){\ms{z}}\put(1,29){\ms{z}}
$$
Since this is a double with opposite orientated edges colored with $x$, the rules of the biquandle imply that the colors of the edges are the same on both sides of crossing.  
Hence $a$ only depends of the isotopy class of $\gamma$.

Next remark that at a self crossing of $\gamma$, the pull back looks like:
$$
\epsh{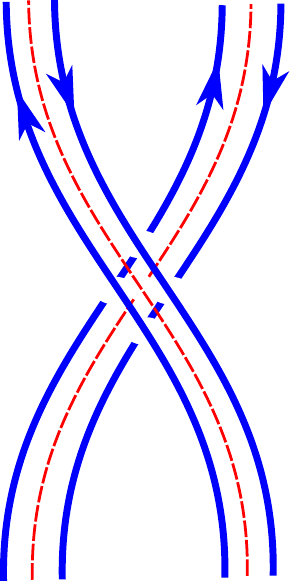}{16ex}\put(-40,-29){\ms{x}}\put(-25,-29){\ms{x}}\put(-13,-29){\ms{y}}\put(1,-29){\ms{y}}
\put(-38,29){\ms{y}}\put(-24,29){\ms{y}}\put(-14,29){\ms{x}}\put(1,29){\ms{x}}
$$
Thus, changing such a crossing we see that the coloring of the diagram does not change outside a neighborhood of the crossing, so $a$ does not change.  Hence $a$ only depends of the homotopy class of $\gamma$ and this proves $a=a'$.
 
 Next, we prove that $f_{(D,c)}$ gives a $\QX$-coloring of $D$.  If $e$ is an edge of $D$, define 
 $$\Qm(e)=f_{(D,c)}(\gamma)\in \QX$$
 where $\gamma$ is a path above $D$ going from $*$ to $e$
 from the left.
 We just showed this assignment is independent of the homotopy class
 determined by $\gamma$.  We need to show that the relation given in
 Figure \ref{fig:Q-col} holds for every crossing.  Let
 $\alpha,\beta, \gamma$ be representatives of the three paths in
 $ Q(T,*)$ going above $D$ to three edges of a crossing as in left
 hand side Figure \ref{F:QpreserQunadleSt}.
 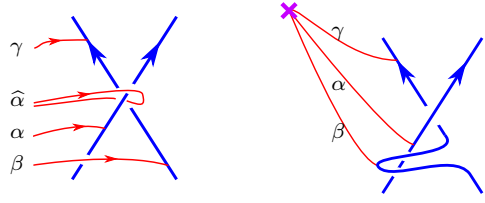
\begin{figure}
  \centering
$$\begin{array}[t]{cc}
\epsh{fig5}{12ex}\put(-52,17){\ms{\gamma}}\put(-52,0){\ms{\wh \alpha}}\put(-52,-10){\ms{\alpha}}\put(-52,-20){\ms{\beta}}\quad\,&\quad\,\epsh{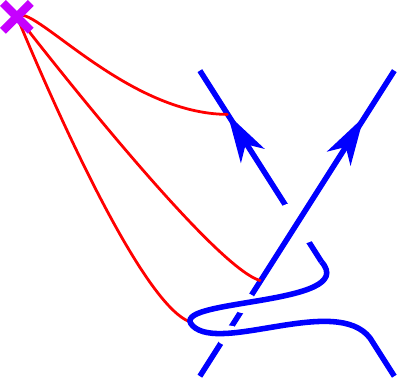}{14ex}\put(-47,22){\ms{\gamma}}\put(-47,5){\ms{\alpha}}\put(-47,-10){\ms{\beta}}
\end{array}$$\hspace{5ex}\ 
\caption{The 3 pathes defining the $\QX$-coloring near a crossing can
  be isotoped together near the base point.}
  \label{F:QpreserQunadleSt}
\end{figure}
  By definition 
of the fundamental quandle, we have 
 $$\alpha\rhd\beta=\wh\alpha^{-1}.\beta= (\alpha_\ve .m.\alpha_\ve^{-1})^{-1}.\beta 
 =\alpha_\ve .m^{-1}.\alpha_\ve^{-1}.\beta= \gamma$$
 Consider an isotopy moving the neighbor of the crossing to the base
 point above the diagram.  Then the 3 pathes are retracted to the
 position in the right of Figure \ref{F:QpreserQunadleSt}.  Now, by
 the definition of the biquandle coloring, we see that the path
 $\gamma$ in Figure \ref{F:QpreserQunadleSt} ends to an edge colored
 by
 $$B_1(f_{(D,c)}(\alpha),S_1(f_{(D,c)}(\alpha),f_{(D,c)}(\beta)))=f_{(D,c)}(\alpha)\rhd
 f_{(D,c)}(\beta)$$
 where $\rhd$ is the quandle structure on $X\cong \QX$ defined in
 Equation \eqref{eq:quandle} and we have showed $\Qm$ is a
 $\QX$-coloring of $D$.
Moreover, this shows $\Qm$ is uniquely defined on objects.

So we have proved that $f_{(D,c)}:P\to X$ factors as a map  $ Q(T,*)\to \QX$.   By construction, if $(D',c')$ is a $X$-colored regular projection of $T$ related to ${(D,c)}$ by colored Reidemeister moves then $f_{(D,c)}=f_{(D',c')}$.  So we can assign 
$$\wt \Qm: \Diag_X/\text{colored Reidemeister moves} \to \T_\QX \text{ given by }(D,c) \mapsto (T,f_{(D,c)}). $$
Finally,  notice that $ \Qm$ is a functor because $f_{(D,c)}$ does not change when a tangle is glued on top of $D$.
\end{proof}

\subsection{More gauge actions} \label{SS:GaugeActionQuandle} In
Subsection \ref{SS:GaugeActionBiq} we defined a gauge action on
quandle diagrams and $\QX$-tangles.  In this subsection, we define
gauge actions on
biquandle colored diagrams.  We show that the functor $\Qm$
preserves these transformations.  To do this we extend the gauge
actions to functors.
The idea behind the
functors is easy: one use colored Reidemeister moves to slide a
component colored by element of $X$ over a diagram, the result is a
gauge transformed diagram.  To define the functor we make this idea
precise and extend it to words in $X$.

In this subsection, let $\QX$ be a quandle with a biquandle factorization $(X,B)$.   To define the desired functors we first consider the gauge action of words in $W_X$ as follows.    
Given words $w=((x_1,\epsilon_1),...,(x_{m},\epsilon_m))$ and $w'=((x'_1,\epsilon'_1),...,(x'_{n},,\epsilon'_n))$ in $ W_X$, let 
 $\chi^+_{w,w'}$ be the $(m,n)$-cable of the positive crossing.   We color and orient the bottom boundary of $\chi^+_{w,w'}$ using the corresponding color and sign in  $(w, w')$.  
 The biquandle structure allows us to extend the  coloring of the bottom boundary to a $X$-coloring of the whole diagram.  Thus, the  first $n$ components of the top boundary are associated  with a word  which we denote by $\wt B^+_1(w,w')$.  Similarly, we denote $B^+_2(w,w')$ by the word associated to the next $n$ components, see Figure \ref{fig:DiagramWords}.   
\begin{figure}
 $$\chi^+_{w,w'}=\epsh{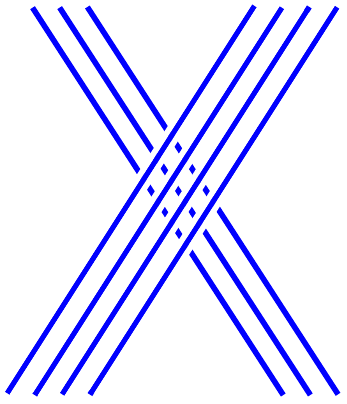}{16ex}
 \put(-74,40) {{\tiny $\wt B^+_1(w,w')$}} \put(-19,40) {{\tiny $\wt B^+_2(w,w')$}}
\put(-86,-36) {{\tiny $(x_1,\epsilon_1)...(x_m,\epsilon_m)$}}\put(-21,-36){{\tiny $(x'_1,\epsilon'_1)...(x'_n,\epsilon'_n)$}}\,\hspace{12ex}\,\chi^-_{w,w'}=\epsh{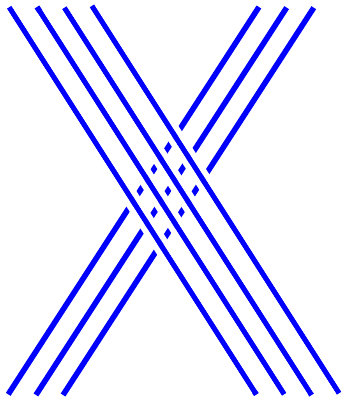}{16ex}
 \put(-74,40) {{\tiny $\wt B^-_1(w,w')$}} \put(-19,40) {{\tiny $\wt B^-_2(w,w')$}}
\put(-86,-36) {{\tiny $(x_1,\epsilon_1)...(x_m,\epsilon_m)$}}\put(-21,-36){{\tiny $(x'_1,\epsilon'_1)...(x'_n,\epsilon'_n)$}}
 $$
 \caption{The diagram $\chi^\pm_{w,w'}$ and words $\wt B_1^\pm $ and $\wt B_1^\pm $.  The orientation is induced from the boundary. }  
  \label{fig:DiagramWords}
\end{figure}
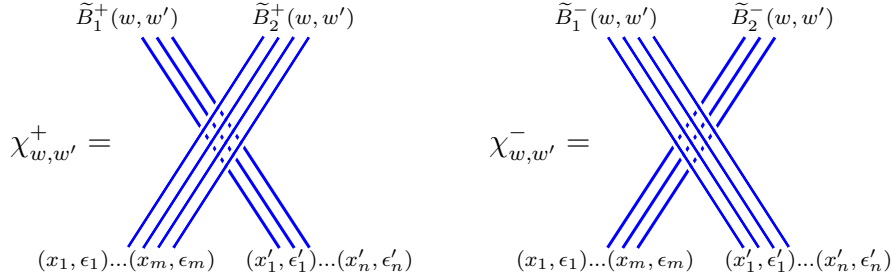
 We obtain a map $\wt B^+=(\wt B^+_1,\wt B^+_2):W_X\times W_X\to W_X\times W_X$.

These maps can be extended to diagrams:
$$\wt
B_1^+:W_X\times\Diag_X\to\Diag_X \text{ and } \wt B_2^+:\Diag_X\times
W_X\to\Diag_X$$
 such that $\wt B_1^+(w,\any)$ and $\wt B_2^+(\any,w)$ are
functors where the symbol $\any$ is used to denote a
variable.  The value of the  functor $\wt B_1^\pm$ on a diagram is defined by using colored Reidemeister moves to slide the $w$ colored cable of a strand above (resp.\ under) the diagram (see Figure \ref{fig:FB}).  Lemma
\ref{L:isotopy} insures that this definition is well defined.
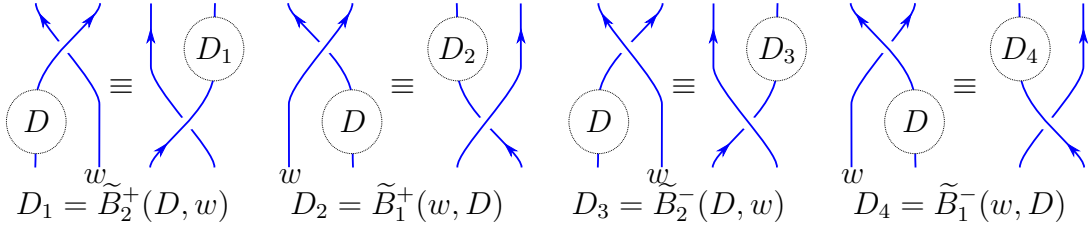
\begin{figure}
  \centering
  $$
  \begin{array}{cccc}
    \epsh{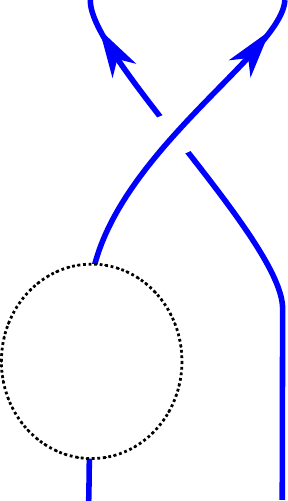}{12ex}\put(-24,-12){$D$}\put(-5,-28){$w$}
    \equiv\epsh{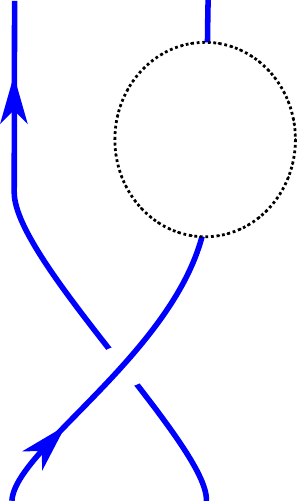}{12ex}\put(-15,11){$D_1$}&\hspace{1ex}
    \epsh{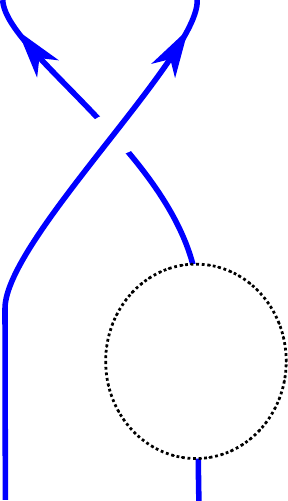}{12ex}\put(-13,-12){$D$}\put(-32,-28){$w$}
    \equiv\epsh{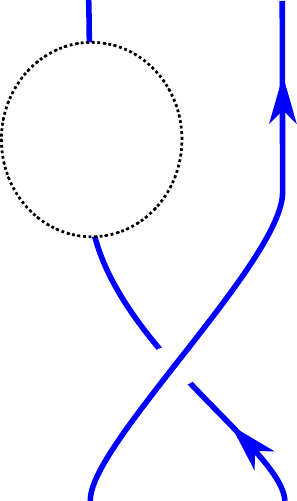}{12ex}\put(-27,11){$D_2$}&\hspace{1ex}
    \epsh{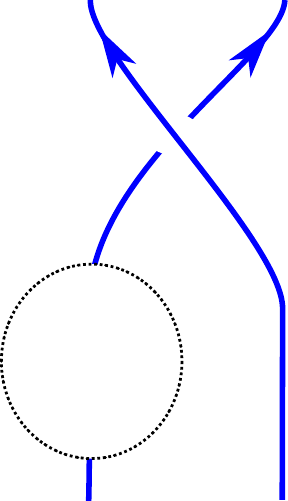}{12ex}\put(-24,-12){$D$}\put(-5,-28){$w$}
    \equiv\epsh{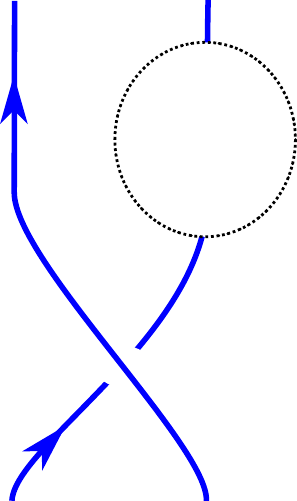}{12ex}\put(-15,11){$D_3$}&\hspace{1ex}
    \epsh{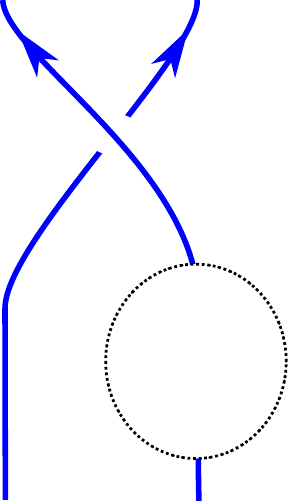}{12ex}\put(-13,-12){$D$}\put(-32,-28){$w$}
    \equiv\epsh{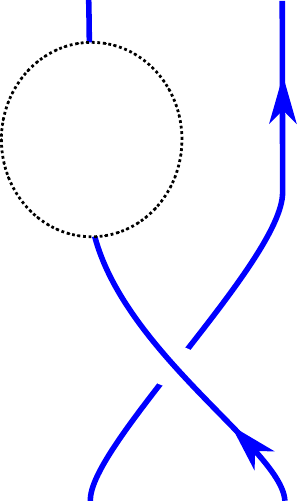}{12ex}\put(-27,11){$D_4$}\\
    D_1=\wt B^+_2(D,w)&D_2=\wt B^+_1(w,D)&
    D_3=\wt B_2^-(D,w)&D_4=\wt B_1^-(w,D)
  \end{array}
  $$
\caption{{Functors $\wt B_i^\pm$.}  Here the blue strands represent the parallel strands made from the cabling; 
 their color and orientation is induced from the boundary.   The equivalence $\equiv$ is generated by colored Reidemeister moves.}
  \label{fig:FB}
\end{figure}

These functors are not
monoidal; instead, if $f:w_1\to w_1'$ and $g:w_2\to w_2'$ are morphism in $\Diag_X$ then $$\wt
B_1^+(w,f\otimes g)=\wt B_1^+(w,f)\otimes \wt B_1^+\left(\wt
B_2^+(w,w_1),g\right)$$
and 
$$\wt B_2^+(f\otimes g,w)=\wt B_2^+\left(f,\wt
B_1^+(w_2,w)\right)\otimes \wt B_2^+(g,w).$$
 Also
$$\wt B_1^+(w^*,\any)=\wt B_1^+(w,\any)^{-1}\et
\wt B_2^+(\any,w^*)=\wt B_2^+(\any,w)^{-1}.$$ 

Analogously, we define the functors $\wt B^-_i$ by using the $(m,n)$-cable of the negative crossing $\chi^-_{w,w'}$ instead of
$\chi^+_{w,w'}$, see Figures \ref{fig:DiagramWords} and \ref{fig:FB}.  All the functors
$\wt B_i^\pm$ leave unchanged the underlying uncolored diagram and
consequently they are compatible with colored Reidemeister moves.

These functors generate an equivalence relation: we say all the diagrams $D,D_1,D_2,D_3,D_4$ of Figure \ref{fig:FB} are \emph{$B$-gauge equivalent}; we call the functors representing these equivalences \emph{$B$-gauge transformations.}   
We also say that two $\QX$-tangles $T,T'$ (resp.\ $\QX$-colored diagrams $D,D'$) are \emph{$B$-gauge equivalent} if there exists $B$-gauge equivalent $ X$-diagrams $E,E'$  such that $\wt\Qm(E)=T$ and $\wt\Qm(E')=T'$ (resp.\ $\Qm(E)=D$ and $\Qm(E')=D'$).

As sets $\QX=X$ so $B$-gauge transformations give  
bijections of $\QX$
which we see as actions of $X$ on $\QX$ and denote with up and down harpoons: if $x\in X$ and $b\in \QX$, let
\begin{equation}
  \label{eq:Baction}
  x\gp b=B_1(x,b)\et x\gm b=B_1^{-1}(x,b).
\end{equation}
We extend this bijective action to words:  for  $w\in W_X$ set
$$w\gp b= \wt B_1^{+}(w,b) \et 
w\gm b= \wt B_1^{-}(w,b).$$ 
\begin{proposition}\label{P:HarpoonAuto}
  For any $w\in W_X$, the bijections of $\QX$ given by $w\gp\any$ and
  $w\gm\any$ are quandle automorphisms of $\QX$.  
In other words, for all $b,b'\in\QX$ we
  have
  $$w\gp(b\rhd b')=(w\gp b)\rhd (w\gp b')
  \et w\gm(b\rhd b')=(w\gm b)\rhd (w\gm b').$$
  Moreover, 
$$b\rhd b'=x\gp\bp{x'\gm b'}\quad\text{ where }x=(b,+)\text{ and } x'=(b,-).$$
\end{proposition}
\begin{proof}
  We prove the relation for down harpoon, the up harpoon is similar.
  Let $x\in X$ and $b,b'\in \QX$.  By applying a sequence of colored
  Reidemeister moves we have the following equivalence of $X$-colored
  diagrams:
$$
\epsh{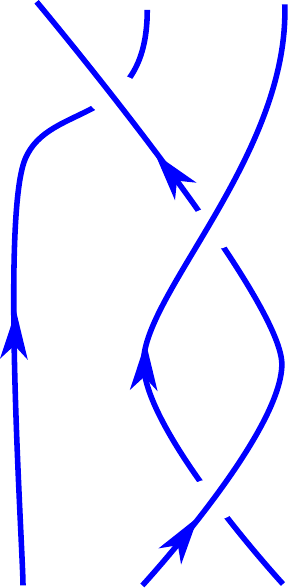}{18ex}{\put(-43,-6){\ms{x}}\put(-25,-9){\ms{b}}
  \put(-25,-35){\ms{b'}}\put(-32,15){\ms{b\!\rhd\!b'}}
  \put(-62,33){\ms{x\gm(b\!\rhd\!b')}}}
\quad\equiv\quad
\epsh{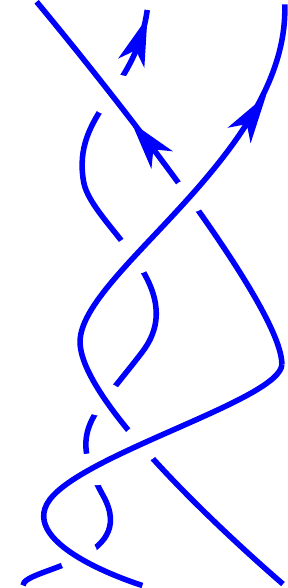}{18ex}{\put(-32,16){\ms{x}}%
  \put(-49,-29){\ms{x\gm b'}}\put(-43,-6){\ms{x\gm b}}}
\quad\equiv\quad\hspace*{8ex}
\epsh{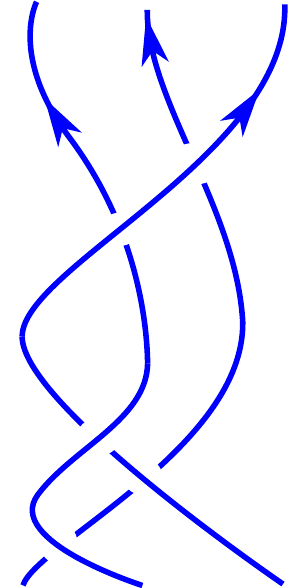}{18ex}{\put(-80,23){\ms{(x\gm b)\!\rhd\!(x\gm b')}}
  \put(-50,-28){\ms{x\gm b'}}\put(-51,-5){\ms{x\gm b}}}.
$$
Since the color on the upper left corner does not change under these
moves we have $x\gm(b\rhd b')=(x\gm b)\rhd (x\gm b')$.  Applying this
relation recursively we have that the desired relations holds for
words.

The last statement of the proposition follows from the fact that for $b,b'\in\QX=X$,
the quandle operation $b\rhd b'$ was defined by the $X$-coloring of
the following diagram:
\begin{equation}
  \label{eq:rhd2}
  \epsh{fig3}{16ex}\put(-27,-30){$b'$}\put(-29,0){$b$}
  \put(-46,30){$b\rhd
    b'$}\put(1,0){$z$}\quad\cong\quad
  \epsh{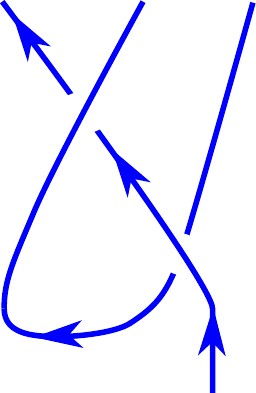}{14ex}\put(-33,-28){$b$}\put(-3,-24){$b'$}\put(-18,9){$z$}
  \put(-30,26){$t$}\text{ where }\left\{
  \begin{array}{l}
    z=(b,-)\gm b'\\t=(b,+)\gp z
  \end{array}
\right.
\end{equation}
Thus $b\rhd b'=(b,+)\gp\bp{(b,-)\gm b'}$. 

\end{proof}
  Note the following reformulation of the previous identity:
  \begin{equation}
    \label{eq:gp}
    \forall x\in X,\,\forall b\in\QX,\,x\gp b=x\rhd(x\gm b).
  \end{equation}

As explained in Subsection \ref{SS:GaugeActionBiq}, the quandle automorphisms of Proposition~\ref{P:HarpoonAuto} induce bijective endofunctors $w\gp\any$ and
  $w\gm\any$ on both $\Tang_\QX$ and $\Diag_\QX$,
  see Equation \eqref{eq:Q-functor}.

\begin{proposition}\label{P:B-gauge=harpoons}
  The $B$-gauge equivalence of $\QX$-tangles are generated by the
  harpoon automorphisms.
  In particular, gauge equivalent $\QX$-tangles (for the quandle gauge
  equivalence) are $B$-gauge equivalent.
\end{proposition}
\begin{proof}
  To prove the first statement, it is enough to show the images under
  $\Qm$ of equivalences represented in Figure \ref{fig:FB} can be
  written in terms of the harpoon automorphisms.  Considering the
  second and fourth equivalence will be sufficient because
  bijections from $\wt B_2^\pm$ are inverse of those from $\wt B_1^\pm$: 
  for any $D:w_1\to w_2$ and $w\in W_X$, an obvious isotopy implies that
  \begin{equation}
    \label{eq:B2=B1}
    \wt B_2^\pm(\wt B_1^\pm(w^\pm,D),w)=D
    \text{ where }w^\pm=B_1^\pm(w_2,w)=B_1^\pm(w_1,w).
  \end{equation}
  We will show the fourth equivalence, the second analogously
  follows.
  Let $T$ and $T_4$ be standard tangles whose regular projections are
  the $X$-diagrams $D$ and $D_4$, respectively, in the fourth 
  equivalence of Figure \ref{fig:FB}.  We will show
  \begin{equation}
    \label{eq:B-gauge}
    w\gm\Qm(D)=\Qm(\wt B_1^-(w,D))
  \end{equation}
that is $w \gm \Qm(D)=\Qm(D_4)$ where
  $w$ is the word in the equivalence.

To do this, notice the edges of $D$ and $D_4$ are the same.  Fix one of these edges $e$.  Choose a path $\gamma$ in $M_{T}=\R^2\times[0,1]\setminus T$ from the base point to $e$ such that $\gamma$ goes above $T$.  Pulling this edge back to the base point as in Figure \ref{F:ExHowCompute} we obtain an $X$-colored diagram whose color near the base point is by definition $\Qm(D)(\gamma)\in X=\QX$.  Similarly, by choosing a path $\gamma_4$ from $*$ to $e$ above $M_{T_4}$ we obtain $\Qm(D)(\gamma_4)$.  

Next, we compare these values.   Let $T'$ be the standard tangle obtained from $T$ by adding strands  determined by $w$ having the two regular projections in the fourth equivalence of Figure \ref{fig:FB}.\begin{figure}
$$
\epsh{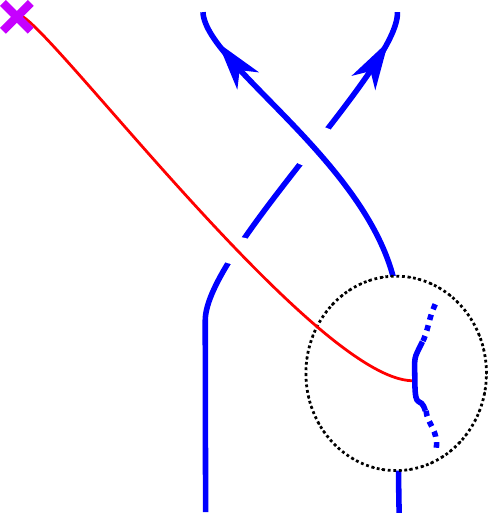}{16ex}\put(-20,-23){$D$}\put(-55,11){$\gamma'$}\qquad
\epsh{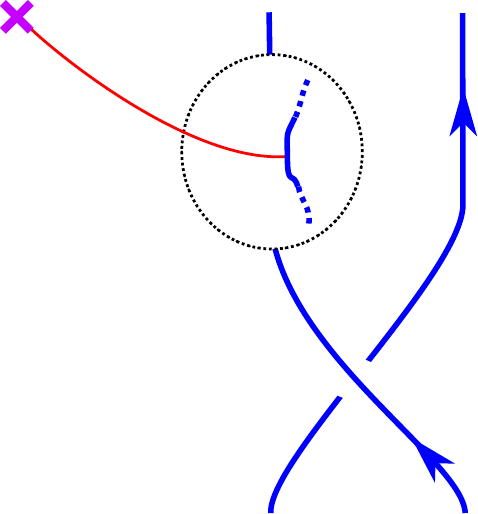}{16ex}\put(-35,8){$D_4$}\put(-55,18){$\gamma_4'$}\qquad
\epsh{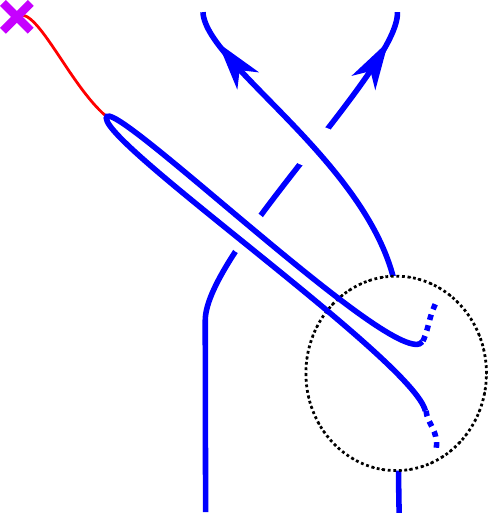}{16ex}\put(-20,-23){$D$}\qquad
\epsh{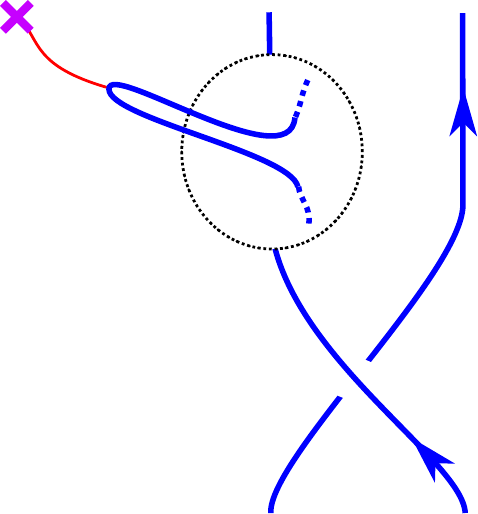}{16ex}\put(-35,8){$D_4$}\qquad
$$
\caption{}
  \label{fig:BgaugeEqHar}
\end{figure}
 Let $\gamma'$ be the path in $T'$ which is the extension of $\gamma$ by going above the strands colored by $w$ then continuing by $\gamma$, see Figure \ref{fig:BgaugeEqHar}.  Let $\gamma'_4$ be the inclusion of $\gamma_4$ also depicted in Figure \ref{fig:BgaugeEqHar}.  Pulling the paths $\gamma'$ and $\gamma'_4$ back to the bases point we obtain $X$-colored diagrams whose colors near the base point are $w \gm \Qm(D)(\gamma)$ and $\Qm(D_4)(\gamma_4)$, respectively (see Figure \ref{fig:BgaugeEqHar}).  But the paths $\gamma'$ and $\gamma'_4$ are isotopic in $T'$ and so $w \gm \Qm(D)(\gamma)=\Qm(D_4)(\gamma_4)$.  Since $e$ is a general edge we have proven $w \gm \Qm(D)=\Qm(D_4)$.

The last part of the proposition follows from the last part of Proposition \ref{P:HarpoonAuto} which implies that for any $\QX$-tangle $T$,  
$$b\rhd T=(b,+)\gp\bp{(b,-)\gm T}.$$

\end{proof}
Remark that the action $\gm$ is the defect of monoidality of $\Qm$ and $\wt\Qm$:
\begin{proposition}\label{P:twisttensor} For all  $D,D'\in\Diag_X$ we have
  $$\Qm(D\otimes D')=\Qm(D)\otimes \bp{w\gm\Qm(D')} \text{ and } 
  \wt\Qm(D\otimes D')=\wt\Qm(D)\otimes \bp{w\gm\wt\Qm(D')}$$
 where $w=\partial_\pm D$.
\end{proposition}
\begin{proof}
Let $e$ be an edge of $D'$.   To compute the value assigned to $e$ in $\Qm(D\otimes D')$ one needs to pass over the edges corresponding to $\partial_\pm D$ then extend to a path above $D'$ as done in the proof of Proposition \ref{P:B-gauge=harpoons}.  Thus, the proof of Equation \eqref{eq:B-gauge} implies the proposition.  
\end{proof}

\subsection{Examples}
\begin{example}[Group factorization]\label{s:G+G-}\ 
A \emph{group factorization} is a pair of groups $(\Gd, \Grc)$ with morphisms $\vp_+,\vp_-:\Grc\to\Gd$ such that the map $$\psi:\Grc \to \Gd \; \text{ given by } \; x \mapsto
  \vp_+(x)\vp_-(x)^{-1}$$ is bijective.
We generalize this notion as follows:  a \emph{generalized
group  factorization} is a tuple $(\Gd, \wb \Gd, \Grc, \vp_+,\vp_-)$
where:
\begin{enumerate}
\item $\Gd,\wb \Gd$ and $\Grc$ are groups such that $\Gd$ is a normal subgroup
  of $\wb \Gd$,
\item $\vp_+,\vp_-:\Grc\to\wb\Gd$ are group morphisms such that the
  map
  $$
  \psi:\Grc \to \wb\Gd \; \text{ given by } \; x \mapsto
  \vp_+(x)\vp_-(x)^{-1}
  $$
  restricts to a bijection between $\Grc$ and $\Gd$.
\end{enumerate}

Then we can associate a biquandle $B:\Grc\times\Grc\to\Grc\times\Grc$ 
to a generalized group factorization.  For $x_1,x_2\in \Grc$, the elements
$(x_4,x_3)=B(x_1,x_2)$ are the unique solutions of the system
  \begin{equation}
    \label{eq:cross}
    \left\{
      \begin{array}{ccc}
        x_4\,x_3&\!\!=\!\!&x_1\,x_2\in\Grc\\\vp_+(x_4)\vp_-(x_3)&\!\!=\!\!&
        \vp_-(x_1)\vp_+(x_2)\in\Gd.
      \end{array}
    \right.
  \end{equation}
  They are given by
  $$x_4=\psi^{-1}\bp{\vp_-(x_1)\psi(x_2)\vp_-(x_1)^{-1}}\et
  x_3=\psi^{-1}\bp{\vp_+(x_4)^{-1}\psi(x_1)\vp_+(x_4)}.$$ 
  Then $B$ satisfies the set braid
  relation (see \cite{KR,GP} for details).  Further, one can see that the sideways map $S$ is given by
  \begin{equation}\label{E:DefS}
  S=(\Id\times i)\circ B^{-1}\circ (i\times \Id)
  \end{equation}
  where $i:x\mapsto x^{-1}$ and $x^{-1}$ is the inverse of $x$ in
  $\Grc$.  Thus $B$ is sideways invertible.  Moreover, for any
  $x\in\Grc$, let $\alpha(x)=\psi^{-1}(\vp_-(x)^{-1}\vp_+(x))=\psi^{-1}(\psi(x^{-1})^{-1})$ then
  $B(x,\alpha(x))=(x,\alpha(x))$.  Thus the last condition of the
  definition of a biquandle holds.

  Let us show the quandle structure associated to this biquandle under the image of $\psi$ is the conjugacy quandle of $ \Gd$.  To do this consider   $B(x_1,x_2)=(x_4,x_3)$ and $B(x_4,x_3)=(x_6,x_5)$ for $x_1,x_2\in \Grc$.  Then from Equation \eqref{eq:quandle} we have $x_6=x_4 \rhd x_1$.  So we need to compute the image of $x_6$ under $\psi$:
  \begin{align}\label{E:AssBConj}
 \psi(x_4 \rhd x_1)= \psi(x_6)&=\vp_-(x_4)\psi(x_3)\vp_-(x_4)^{-1}\notag\\
&  =\vp_-(x_4)\vp_+(x_4)^{-1}\psi(x_1)\vp_+(x_4)\vp_-(x_4)^{-1} \notag\\
 & =\psi(x_4)^{-1}\psi(x_1)\psi(x_4).
  \end{align}

  Thus, the quandle structure on $\Grc$ associated to $B$ is the pull
  back via $\psi$ of $\Conj(\Gd)$.  Finally,
  $$\psi\circ\bp{B^{\pm}_1}\left(x,\any\right)=\vp_\mp(x)\psi(\any)\vp_\mp(x)^{-1}$$ 
  so $B$-gauge transformations are the pull back of conjugations in $\wb \Gd$.

  Note that factorizable Poisson Lie groups provide examples of such
  factorization $(\Gd, \Grc, \vp_+,\vp_-)$.  In this case $\Gd$ is a
  factorizable Poisson Lie group and $\Grc$ is its Poisson Lie dual.
\end{example}
\begin{example}
  Let $\Gd$ be a complex simple Lie group. Fix a Borel subgroup
  $\GB\subset \Gd$.  According to the Iwasawa decomposition, every element
  $g\in \Gd$ can be written as
$$g=ank=k'a'n'.$$
Here $a,a'\in {\mathsf A}\subset \GH \subset \GB$ are totally positive elements of
the Cartan subgroup of $\GB$ (for $\SL_n(\C)$ they are diagonal matrices
with positive entries).  Elements $n,n'\in \GN\subset \GB$ are unipotent
(complex) and form the maximal unipotent subgroup of $\GB$, and
$k,k'\in \GK\subset \Gd$ are elements of the compact real form of $\Gd$.  Let $\GK'=\GK$ with opposite multiplication,
then we have an example of factorization $(\Gd, \Grc, \vp_+,\vp_-)$ with 
$$\Grc=\GA\GN\times \GK',\quad\vp_+:(an,k)\mapsto an\et\vp_-:(an,k)\mapsto k^{-1}.$$

\end{example}

\begin{example}\label{Ex:SC-G}
Here we consider a particular example of a exact group factorization which we will see later is related to the $\Gd$-link invariants associated to the semi cyclic $U_q\slt$-modules of \cite{GP}, see Subsection \ref{SS:ExampleSemicyclic}.  

Let $\SL_2(\C)$ be the group of $2\times2$ matrices with determinate $1$.    Let $\wb\Gd$ be the group of $2\times2$ upper triangular invertible
  matrices, $\Gd=\wb\Gd\cap\SL_2(\C)$ and
 $$\Grc=\qn{g_{\kappa,\ve}=\bp{{\small\left(\begin{array}{cc}
     \kappa&0\\0&1
    \end{array}\right) },\mat\kappa\ve}:\kappa\in\C^*,\ve\in\C}\subset \wb\Gd\times \wb\Gd.$$
Let $\vp_+,\vp_-$ the projections on the two factors.

Then $(\Gd, \wb\Gd, \Grc, \vp_+,\vp_-)$ is a group factorization where $\psi: \Grc\to \Gd$ is given by $$\psi\bp{g_{\kappa,-\ve}}={\small \left(\begin{array}{cc}     \kappa&0\\0&1    \end{array}\right)}\mat\kappa{-\ve}^{-1}={\small \left(\begin{array}{cc}      \kappa&\ve\\0&\kappa^{-1}  \end{array}\right)}.$$

As above this induces a structure of biquandle
$B:\Grc\times\Grc\to\Grc\times\Grc$.  In particular, given
$x_1=g_{\kappa,{\ve}}$ and $x_2=g_{\kappa',\ve'}$ we can use the
formulas right after Equation \eqref{eq:cross} to compute
$B(x_1,x_2)=(x_4,x_3)$:
  $$x_4= g_{{\kappa'},{(\ve \kappa' + \ve' - \ve (\kappa')^{-1})\kappa^{-1}}}
  \et
  x_3= g_{{\kappa},{\ve(\kappa')^{-1}}}.
  $$ 
  Equation \eqref{E:AssBConj} implies the quandle associated to
  $(B,\Grc)$ is $\Conj(\Gd)$ where $\Gd$ is the upper Borel of
  $\SL_2(\C)$.

\end{example}

\begin{example}[Fibered product]\label{ss:fp}\ 
Let $(X,B)$ be a biquandle and $f:X\to Y$ be a map to a set $Y$. 
We say that $f$ is {\em invariant}
on $X$ if whenever $(x_4,x_3)=B(x_1,x_2)$, one has $f(x_4)=f(x_2)$ and
$f(x_3)=f(x_1)$.    
Given an invariant map $f:X\to Y$ and a surjective map $g:Z\to Y$,
the fiber product $X\times_{(f,g)}Z=\{(x,z):f(x)=g(z)\}$ is naturally
equipped with the biquandle structure given by
$$B((x_1,z),(x_2,z'))=((x_4,z'),(x_3,z))$$
where $(x_4,x_3)=B(x_1,x_2)$.  Furthermore, the projection on the
factors are biquandle maps.
\end{example}

\section{Biquandle braidings in pivotal categories}\label{S:BiquandleBraidings}
In this section  we define an
invariant of $\QX$-tangles analog to the
Reshetikhin-Turaev ribbon functor.  Here $\QX$ is a quandle and the
algebraic data involved is a biquandle factorization $(X,B)$ of $\QX$ and
a representation of $(X,B)$ in a pivotal category $\cat$.  The
construction can be summarized in a composition of functors:
$$\T_\QX\ \stackrel{\wt\Qm^{-1}}\longrightarrow\ \Diag_X/\!\equiv\ \stackrel F\longrightarrow\ \cat$$
where as above $\T_\QX$ is the category of $\QX$-tangles (topological intrinsic object), 
$\Diag_X$ is the category of $X$-colored 
planar diagrams (algebraic and computational object) and the equivalence ``$\equiv$'' is generated by colored
Reidemeister moves.

\subsection{Pivotal categories}\label{SS:PivCat}
Recall that a \emph{pivotal}  category is a (strict) tensor category $\cat$, with unit object~$\unit$, such that to each
object $X\in \cat$ there is associated a \emph{dual object}~$X^*\in \cat$ and four morphisms
\begin{align*}
& \ev_X \colon X^*\otimes X \to\unit,  \quad \coev_X\colon \unit  \to X \otimes X^*,\\
&   \tev_X \colon X\otimes X^* \to\unit, \quad   \tcoev_X\colon \unit  \to X^* \otimes X,
\end{align*}
such that $(\ev_X,\coev_X)$ is a left duality for $X$,
$( \tev_X,\tcoev_X)$ is a right duality for~$X$, and the induced left
and right dual functors coincide as monoidal functors (see for example
\cite{BW}).  Let $\kk=\End_\cat(\unit)$.  The tensor product gives
$\kk$ the structure of a commutative monoid which acts on left and
right on homomorphism set in $\cat$.  We assume that these left and
right actions are equal.  If $f$ is a morphism in $\cat$ and $k\in\kk$,
we just write $k.f$ for 
\begin{equation}
  \label{eq:kk}
  k.f=k\otimes f=f\otimes k.
\end{equation}
For $X\in \cat$, the \emph{right quantum dimension} $\dim_\cat(X)$ is
the element
in $\kk$ determined by
$ \dim_\cat(X)=\tev_X \circ \coev_X  $

An object $V$ of
$\cat$ is \emph{absolutely simple}
if the map
$\kk \to\End_\cat(X)$, $k \mapsto k.\Id_X$ is a bijection.
An object $V$
of $\cat$
is \emph{regular} if $V\otimes\any$
is a faithful endofunctor of $\cat$,
i.e.\ $\Id_V\otimes\any$
defines a family of injective maps on $\Hom$-sets
of $\cat$
(this is true for example when $\kk$
is a commutative ring, for $V$ a non zero object, if objects of $\cat$
are free $\kk$-modules).

\subsection{Biquandle  representations}\label{SS:BiquandleRep}
Let $(X,B)$ be a biquandle and let $\cat$ be a strict pivotal category.   Given a family of
  objects $\{V_x\}_{x\in X}$ of $\cat$ the map $B:X\times X\to X\times X$ induces a map (which we still denoted by $B$):
  $$B=(B_1, B_2): \{V_x\}_{x\in
    X}\times\{V_x\}_{x\in X}\to \{V_x\}_{x\in X}\times\{V_x\}_{x\in X}$$
 given by $B((V_x , V_y))=  (V_{B_1(x,y)}, V_{B_2(x,y)})$.

 A \emph{Yang-Baxter model}
  $(V_\any,c_{\any,\any})$ of a biquandle $(X,B)$ in $\cat$ is a family of
  objects $\{V_x\}_{x\in X}$ of $\cat$  and a family of isomorphisms
  $$\{c_{x,y}:V_x\otimes V_y\to B_1(V_x,V_y)\otimes B_2(V_x,V_y)\}_{(x,y)\in X^2}$$
  which satisfy the colored braid relation: for any elements $x,y,z$ in $X$,
  we have an equality of isomorphisms
    \begin{equation}
      \label{eq:YB}
      (c_{\any,\any}\otimes \Id_\any)\circ 
      (\Id_\any\otimes c_{\any,\any})\circ (c_{x,y}\otimes \Id_{V_z})=
      (\Id_\any\otimes c_{\any,\any})\circ
      (c_{\any,\any}\otimes \Id_\any)\circ (\Id_{V_x}\otimes c_{y,z})
    \end{equation}
    where the $\any$ objects are completed with the biquandle
    structure $B$.

  Let $(V_\any,c_{\any,\any})$ be a Yang-Baxter model of $(X,B)$ in $\cat$.  For any
  objects $V_1$, $V_2$, in $\{V_x\}_{x\in X}$, with
  $B(V_1,V_2)=(V_4,V_3)$, the isomorphisms $c_{\any,\any}$ and the pivotal
  structure give rise to sideways morphisms:
  \begin{align*}
    s^{+}_L(V_4,V_1)&=\bp{\ev_{V_4}\otimes\Id_{V_3\otimes V_2^*}}\circ
    \bp{\Id_{V_4^*}\otimes c_{V_1,V_2}\otimes \Id_{V_2^*}}\circ
    \bp{\Id_{V_4^*\otimes V_1}\otimes\coev_{V_2}},\\
    s^{-}_R(V_4,V_1)&=\bp{\Id_{V_4^*\otimes V_1}\otimes\tev_{V_2}}\circ
    \bp{\Id_{V_4^*}\otimes c_{V_1,V_2}^{-1}\otimes \Id_{V_2^*}}\circ
    \bp{\tcoev_{V_4}\otimes\Id_{V_3\otimes V_2^*}}.
  \end{align*}
  We say the model $(V_\any,c_{\any,\any})$ is \emph{sideways invertible} if $s^{+}_L$ and
  $s^{-}_R$ are inverse isomorphisms.  In other words, for any objects
  $V_1$, $V_2$, $V_3$, $V_4$ as above, we have
  \begin{align}\label{eq:R2-a}
    s^{-}_R(V_4,V_1)\circ s^{+}_L(V_4,V_1)&=\Id_{V_4^*\otimes V_1},&
    s^{+}_L(V_4,V_1)\circ s^{-}_R(V_4,V_1)&=\Id_{V_3\otimes V_2^*}.
  \end{align}
  
For $x\in X$, let $\theta_x:V_x\to V_x$ be the endomorphism 
  defined by
  $$\theta_x=\bp{\Id_{V_x}\otimes\tev_{V_{\alpha(x)}}}
  \bp{c_{V_x,V_{\alpha(x)}}\otimes\Id_{(V_{\alpha(x)})^*}}
  \bp{\Id_{V_x}\otimes\coev_{V_{\alpha(x)}}}$$
  where $\alpha: X\to X$ is the bijection given in Definition \ref{D:biquandle}.
  We say that the model $(V_\any,c_{\any,\any})$ induces a {\em twist} $\theta$ if 
  $$\theta_x=\bp{\ev_{V_{\alpha^{-1}(x)}}\otimes\Id_{V_x}}
  \bp{\Id_{(V_{\alpha^{-1}(x)})^*}\otimes c_{V_{\alpha^{-1}(x)},V_x}}
  \bp{\tcoev_{V_{\alpha^{-1}(x)}}\otimes\Id_{V_x}}.$$

\begin{definition}
  A \emph{representation} $(V_\any,c_{\any,\any})$ of a biquandle $(X,B)$ in a pivotal
  category $\cat$ is a Yang-Baxter model of $(X,B)$ in $\cat$
  which  is sideways invertible and induces a twist.
\end{definition}

\subsection{The biquandle ribbon functor}
Here we define the analog of the Reshetikhin-Turaev ribbon functor for $\QX$-tangles.  First we define a functor on $X$-diagrams.   

\begin{theorem}\label{T:FunctorF}
  Let $(V_\any,c_{\any,\any})$ be a representation of a biquandle $(X,B)$ in a
  pivotal category $\cat$.  Then there exists a unique tensor functor 
  $$
  F:\Diag_X/\!\equiv\,\to\cat$$ 
  such that for any
  $x,y$ in $X$ we have $F((x,+))=V_x$,
  $F((x,-))=V_x^*$ and the morphisms are determined by it value on the elementary tangle diagrams:
  $$F({\chi^+_{x,y}})=c_{x,y}, \;\;\; F({\chi^-_{x,y}})=c_{x,y}^{-1},$$
  $$F(\lev_x)=\lev_{V_x}, \; F(\lcoev_x)=\lcoev_{V_x},\;  F(\rev_x)=\rev_{V_x} \text{ and } F(\rcoev_x)=\rcoev_{V_x}.$$
\end{theorem}
\begin{proof}
  The pivotal category $\cat$ gives rise to a RT-pivotal functor from
  the category of $\cat$-colored planar graphs with coupons to $\cat$.
  A diagram $D\in\Diag_X$ is made up of the elementary tangles:
  strands, cap, cup and braidings.  By replacing all the
  $\pm$-crossings of $D\in\Diag_X$ with coupons filled with the
  morphisms $c^{\pm1}$ we can obtain a $\cat$-colored planar graph
  with coupons which we denote by $D'$.  Define $F(D)$ as the
  RT-pivotal functor evaluated on $D'$.  From the properties of the
  RT-pivotal functor, $F(D)$ only depends on the planar isotopy class
  of $D'$.  To see $F$ is invariant under the equivalence generated by
  the color Reidemeister moves it is sufficient to check that it
  satisfy the generating set of six colored Reidemeister moves:
  $$\{\text{the two }RII_{++},\; RIII_{+++},\; RII_{+-}, \; RII_{-+}, \; RI^f\},$$ 
  see Figure \ref{fig:RM}.  But the axioms of the biquandle 
representation
imply
  these moves are satisfied.
Thus, we obtain a tensor functor $F:\Diag_X/\!\equiv\,\to\cat$.  By definition this functor satisfies the values on the elementary tangles given in the statement of the theorem.  By Lemma \ref{L:cat_diag} these elementary tangles generate the tensor category $\Diag_X$ and so determine the functor.  
\end{proof}
 
\begin{corollary} \label{C:RTFunctorTangle}
  Let $\QX$ be a quandle with a biquandle factorization $(X,B)$ and
  $(V_\any,c_{\any,\any})$ be a representation of $(X,B)$ in a pivotal
  category $\cat$.  Then the formula $\wt F=F\circ \wt\Qm^{-1}$
  defines a functor
  $$\wt F:\Tang_\QX\rightarrow \cat.$$
  This functor is uniquely determined by $\wt F((x,+))=V_x$,
  $\wt F((x,-))=V_x^*$,  the image of the elementary tangles
  cap, cup and braidings and by the property: for any
  $T,T'\in\Tang_\QX$,
  \begin{equation}\label{E:wtFotimes}
  \wt F(T\otimes T')=\wt F(T)\otimes\wt F(w\gm T'),
  \end{equation}
  where $w=\wt\Qm^{-1}(\partial_\pm T)^*$.
\end{corollary}
\begin{proof}
Clearly,  $\wt F$ is a functor.  Let us prove that Equation \eqref{E:wtFotimes} holds.
By applying $\wt \Qm^{-1}$ to both sides of the equality in Proposition \ref{P:twisttensor} we have
$$
D\otimes D'=\wt \Qm^{-1}\left( \wt\Qm(D)\otimes \bp{w\gm\wt\Qm(D')}\right)
$$
where $w=\partial_\pm D$.  
Substituting the tangles $T=\wt\Qm(D)$ and $T'=w\gm\wt\Qm(D')$  into this equation we have
$$
\wt\Qm^{-1}(T)\otimes \wt\Qm^{-1}(w^*\gm T') =\wt \Qm^{-1}( T\otimes T').
$$
where $D'=\wt\Qm^{-1}(w^*\gm T')$ because $(w \gm\any)^{-1}=(w^*\gm\any)$ which follows from an argument using a Reidemeister II move.  Since $F$ is a tensor functor the desired equation holds.  

Finally, we need to prove the functor is uniquely determined by the above properties.  Each tangle in $\Tang_\QX$ can be written as the composition of the tensor product (i.e.\ disjoint union) of the elementary tangles cap, cup and braidings.  But Equation \eqref{E:wtFotimes} says that such a tensor product is determined by the values of the elementary tangles.   
\end{proof}

\subsection{Gauge invariance of the functor $F$}\label{SS:GaugeInvF}
Here we consider gauge invariance of the restriction of
$F$ to $\QX$-links.
Recall that we assume that $\cat$ is a pivotal category for which
Equation \eqref{eq:kk} is satisfied.
\begin{definition}\label{D:reg}
  A representation $(V_\any,c_{\any,\any})$ of a biquandle $(X,B)$ in
  $\cat$ is {\em simple} (resp.\ \emph{regular}) if each object $V_x$
  is absolutely simple (resp.\ regular), see Subsection
  \ref{SS:PivCat} for definitions of absolutely simple and regular.
\end{definition}
If a regular object $V\in\cat$, has an endomorphism $f\in\kk.\Id_V$, we
denote by $\brk f$ the unique element of $\kk$ such that
$f=\brk{f}\Id_V$.
\begin{theorem}\label{T:FGaugeInv} Let $(V_\any,c_{\any,\any})$ be a simple regular representation
  of a biquandle $(X,B)$ in a pivotal category $\cat$ and $\QX$ its
  associated quandle.  For any $x,y\in X$ and any diagram
  $D_x\in\End_{\Diag_X}((x,+))$ one has
  \begin{equation}\label{E:FD=FB1D}
  \brk{F(D_x)}=\brk{F(\wt B_1^\pm(y,D_x))}.
  \end{equation}
  Moreover, if $T$ and $T'$ are $B$-gauge equivalent 1-1 $\QX$-tangle
 then $$\brk{\wt F(T)}=\brk{\wt F(T')}.$$
\end{theorem}
\begin{proof}
 We prove the positive version of Equation \eqref{E:FD=FB1D}, the negative is
  similar.  By definition, the diagrams $\Id_{(y,+)}\otimes D_x$ and
  $\chi^-_{B(y,x)}\circ\left(\wt B_1(y,D_x)\otimes \Id \right)\circ\chi^+_{y,x}$ are related
  by a sequence of colored Reidemeister moves so their images under $F$
  are equal.  Using this observation we have
  $$\brk{F(D_x)}\Id_{(y,+)}\otimes \Id_{(x,+)}=F( \Id_{(y,+)}\otimes
  D_x)$$$$=F\bp{\chi^-_{B(y,x)}\circ
    \left(\wt B_1(y,D_x)\otimes\Id_{(B_2(y,x),+)}\right)\circ
    \chi^+_{y,x}}=\brk{F(\wt B_1(y,D_x))}c^{-1}_{y,x}\circ
  c_{y,x}$$ and the result follows from the regularity of $V_x\otimes
  V_y$.
  
To prove the last statement, recall that the bijections  $\wt B_2^\pm$ are inverses of $\wt B_1^\pm$, see Equation \eqref{eq:B2=B1}.  Thus, Equation \eqref{E:FD=FB1D} implies a generating set of equivalences are satisfied and the result follows.  
\end{proof}

\begin{lemma}\label{L:qdimGauge}
 Let $(V_\any,c_{\any,\any})$
  be a regular representation of a biquandle $(X,B)$
  in a pivotal category $\cat$.  
  Then the right quantum dimension $\dim_\cat(V_\any): X \to \kk$ is gauge invariant, i.e.\ $\dim_\cat(V_x)=\dim_\cat(V_{B_1^\pm(y,x)}) $ for all $x,y\in X$.  
\end{lemma}
\begin{proof}
  The proof is similar to the previous one: Let $u_x$
  denote a clockwise oriented planar unknot colored with $x\in
  X$, let $y\in X$ and $x'= B_1(x,y)$.  Then $\Id_{(y,+)} \otimes
  u_x$ and $u_{x'}\otimes
  \Id_{(y,+)}$ are related by colored Reidemeister moves thus their
  image by $F$
  are equal.  Then $F(u_x)\Id_{V_y}=F(u_{x'})\Id_{V_y}$
  and $F(u_x)=F(u_{x'})$ follows because $V_y$ is regular.
\end{proof}
  
  Let $\Link_\QX$ be the set of
  $\QX$-links, i.e.\ closed $\QX$-tangles. 
By closing 1-1 $\QX$-tangles Lemma \ref{L:qdimGauge} immediately implies a version of Theorem \ref{T:FGaugeInv} for $\QX$-links:
\begin{corollary}
  The invariant $\wt F$ restricted to the set $\Link_\QX$ of
  $\QX$-links is gauge invariant, i.e. $\wt F(L)=\wt F(w\gm L)$ and $\wt F(L)=\wt F(w\gp L)$ for all $L\in \Link_\QX$ and $w\in W_X$.
\end{corollary}
\begin{proof}
  Any $\QX$-link $L$ is the closure of a 1-1 $\QX$-tangle $T_x$ and
  its invariant is the product of the corresponding
  element of $\kk$ with the quantum dimension
  $\dim_\cat(V_x)\brk{\wt F(T_x)}$.  By Theorem \ref{T:FGaugeInv} and
  Lemma \ref{L:qdimGauge} both of these functions are gauge invariant,
  implying is $F(L)$.
\end{proof}

\subsection{Modified $\QX$-link invariant}\label{SS:ModInv} Here we use the modified
dimension to re-normalize the functor $\wt F$.

Let $\{V_x\}_{x\in X}$ be a family of simple modules in the pivotal
category $\cat$.  Let $\qd:X\to\kk$ be a function called the modified
dimension.  Let $\D_\cat$ be the set of $\cat$-colored
spherical 
ribbon graph with coupons whose edges are colored with the
objects $\{V_x\}_{x\in X}$.  Let $D\in\D_\cat$ and let $e$ be an edge of $D$.
By cutting $e$ we can obtain a 1-1 $\cat$-colored diagram $D_x$ where
$x\in X$ is the color of both the top and bottom boundary with
positive orientation: $ \partial_\pm D_x=(V_x,+)$.  We call $D_x $ a
cutting presentation of $D$.  

Following \cite{GPT2} we say that $(\{V_x\}_{x\in X},\qd)$ is an
\emph{ambidextrous pair}, or \emph{ambi pair} for short, if
the map $\D_\cat\to \kk=\End_\cat(\unit)$ given by
\begin{equation}
  \label{eq:md}
  D\mapsto F'(D):=\qd(x)\brk{F(D_x)}
\end{equation}
is
independent of the choice of the cutting presentation
$D_x$ of $D$.  In such a situation we have the following renormalized invariant:

\begin{theorem}\label{T:LinkF'}
  Let $\QX$ be a quandle with a biquandle factorization $(X,B)$ and
  $(V_\any,c_{\any,\any})$ be a simple regular representation of
  $(X,B)$ in $\cat$.  Assume $(\{V_x\}_{x\in X},\qd)$ is an ambi
  pair, then the function
  $$\begin{array}{rcl}
    \wt{F}':\Link_\QX&\to& \kk\\
    T&\mapsto& F'(\wt\Qm^{-1}(D))
  \end{array}$$
  is a well defined invariant of the $\QX$-link $T$ where $D$ is any
  $\QX$-diagram representing $T$.
Moreover, the function   $\wt F'$ is gauge invariant if $\qd$ is an invariant map (see
  Example \ref{ss:fp}).
\end{theorem}

We end the subsection by recalling how one can obtain an ambi pair.
Suppose that $\kk=\End_\cat(\unit)$ is a commutative ring and $\cat$
is a \emph{tensor \kt category}: $\cat$ is a tensor category such that
its hom-sets are left \kt modules, the composition and tensor product
of morphisms are \kt bilinear.

By an \emph{ideal} of  $\cat$ we mean a full subcategory $\ideal$ of $\cat$ such that:  
\begin{enumerate}
\item  If $V\in \ideal$ and $W\in \cat$, then $V\otimes W,\; W\otimes V  \in \ideal$;
\item If $V\in \ideal$ and if $W\in \cat$ is a retract of $V$,
  then $W\in \ideal$.  
\end{enumerate}
 A \emph{trace} on an ideal  $\ideal$ is a family of linear functions
$$\{\mt_V:\End_\cat(V)\rightarrow \kk \}_{V\in \ideal}$$
such that:
\begin{enumerate}
\item  If $U,V\in \ideal$ then for any morphisms $f:V\rightarrow U $ and $g:U\rightarrow V$  in $\cat$ we have 
\begin{equation*}
\mt_V(g f)=\mt_U(f  g).
\end{equation*} 
\item \label{I:Prop2Trace} If $U\in \ideal$ and $W\in \cat$ then we have
\begin{equation*}
\mt_{U\otimes W}\bp{f}=\mt_U \bp{(\Id_U \otimes \tev_W)(f\otimes \Id_{W^*})(\Id_U \otimes \coev_W)},
\end{equation*}
\begin{equation*}
\mt_{W\otimes U}\bp{g}=\mt_U \bp{( \ev_W \otimes \Id_U)( \Id_{W^*}\otimes g)(\tcoev_W \otimes \Id_U)}
\end{equation*}
for any $f\in \End_\cat(U\otimes W)$ and $g\in \End_\cat(W\otimes U)$.  
\end{enumerate}

Let $\{V_x\}_{x\in X}$ be any family of simple objects of $\cat$.  
Let $\{\mt_V \}_{V\in \ideal}$ be a trace on an ideal $\ideal$ in
$\cat$ such that $V_x\in \ideal$ for each $x\in X$.  Then Theorem 5 of \cite{GPV} implies 
$(\{V_x\}_{x\in X},\qd)$ is an {ambi pair} where the modified
dimension $\qd:X\to \kk$ is defined by $\qd(x)=\mt_{V_x}(\Id_{V_x})$.

\subsection{Example: Semi cyclic $U_q\slt$-modules}\label{SS:ExampleSemicyclic}
In \cite{GP}, $\Gd$-links invariants associated to semi cyclic
$U_q\slt$-modules are introduced.  In this subsection we 
discuss a biquandle which has a representation formed from the semi cyclic modules of quantum $\slt$.
Using a modified trace, these representation gives rise to a renormalized invariant which we show is equal to the ADO type invariants of
  $\C$-colored link in $S^3$ defined in \cite{CGP1}, which is a generalization of \cite{ADO}.

Fix a positive integer $r\ge2$ and let $\xi=e^{\frac{i\pi}{r}}$ be a
$2r^{th}$-root of unity. If $x\in \C$ then let
$\xi^x=\e^{\frac{2ix\pi}{N}}$, $\qn x=\xi^x-\xi^{-x}$ and let
$\s=\xi^{-\frac{r(r-1)}2}=i^{1-r}$.

Let $(\Gd, \wb\Gd, \Grc, \vp_+,\vp_-)$ be the generalized group factorization given
in Example \ref{Ex:SC-G} with $\Gd$ the upper Borel of
$\SL_2(\C)$.  Let $(X,B)$ be the sub biquandle of the
fibered product $\Grc\times_{\wt\kappa,f}\C$ where
$\wt\kappa:g_{\kappa,\ve}\mapsto\kappa$ and $f(\alpha)=\xi^{r\alpha}$,
see Example \ref{ss:fp}.  That is
$$X=\qn{\bp{g_{\kappa,\ve},\alpha}:\alpha\in\C\setminus\Z,\kappa=\xi^{r\alpha},
  \ve\in\C}\cup\qn{\left(g_{1,0},0\right)}$$ where $g_{\kappa,\ve}=\bp{{\small\left(\begin{array}{cc}
      \kappa&0\\0&1
    \end{array}\right) },\mat\kappa\ve}\in\wb\Gd \times \wb\Gd$.  Its associated quandle is $\QX=\Conj(\Gd)\times_{\wt\kappa,f}\C$
where $\Gd$ is the upper Borel of $\SL_2(\C)$.

In \cite{GP}, it is shown that the pivotal $\C$-category $\cat$ of semi-cyclic
$U_\xi\slt$-modules gives a representation of the braid group which supports a Markov trace.  
In Appendix \ref{Ap:ProofSemiCyclic} we discuss how these semi-cyclic
$U_\xi\slt$-modules lead to a regular simple representation $(V_\any,c_{\any,\any})$ of the biquandle 
  $(X,B)$ which we will now use in  this subsection.  Let $\wt F$ be the functor associated to this biquandle representation given in Corollary \ref{C:RTFunctorTangle}.  

 In \cite{GP4} a modified trace on the projective modules of unrestricted quantum group is defined.  This trace restricts to a trace $\mt$ on the ideal of projective modules $\Proj$ of $\cat$.  The modules in the representation  $(V_\any,c_{\any,\any})$ are all projective.  Thus, Theorem \ref{T:LinkF'} implies there exists a $\Q$-link invariant $\wt{F}':\Link_{\QX}\to \C$.

\begin{theorem}\label{T:F'=M'} The invariant $\wt{F}'$ is equal to the semi-cyclic link invariant $M':\Link_{\QX}\to \C$
  of \cite{GP}. \end{theorem}
\begin{proof}
  In \cite{GP}, the use of a Markov trace on the colored braid group
 was used because at that time the existence of a modified right  trace was only known (from \cite{GP4} we have a modified left and right trace).  The above construction is more
  general because colored braids are $\QX$-tangles and if $\sigma$ is
  such a braid whose braid closure is a $\QX$-link $L$ then
  $\wt F'(L)=\mt(\wt F(\sigma))$ by the properties of a modified trace.
  But this was the definition of $M'$ given in \cite{GP}.
\end{proof}
This new approach and the following lemma imply that $M'=\wt F'$ is
gauge invariant:
\begin{lemma}\label{L:SCunknotGauge}
The modified dimension $\qd$ and the invariant $ \wt F'$ are both gauge invariant.
\end{lemma}
\begin{proof}
 From Theorem \ref{T:LinkF'} it is enough to show $\qd$ gauge invariant.
  By definition  $\mt(\Id_{V_x})$ is the modified dimension of the
  semi-cyclic module $V_x$.  In  \cite{GP4}  it is shown that $\mt(\Id_{V_x})$  only
  depends on the value $\Omega_x$ of the casimir action on $V_x$.  But
  now the map $x\mapsto\Omega_x$ is invariant (as in Subsection
  \ref{ss:fp}) on $X$ so the modified dimension is gauge invariant.
\end{proof}

 Let $\pi:\QX\to \C$ be the group morphism given by
$$\bp{{g_{\kappa,\ve}
}
,\alpha}\mapsto\alpha.$$
    The following theorem answers a question of \cite{GP}:
\begin{theorem}
  By applying $\pi$ to the colors of the strands of a $\QX$-link $L$, one
  gets a $\C$-colored link which we denote by $\pi_*L$.  Let $M'$ be the semi-cyclic
  invariant of $\QX$-links of \cite{GP} and let $ADO$ be the nilpotent invariant of
  $\C$-colored links in $S^3$ given in \cite{CGP1}.   Then for any
  $\QX$-link $L$, we have

 $$M'(L)=ADO(\pi_*L).$$
\end{theorem}
\begin{proof}
 The invariant $M'$ is designed as an extension of ADO:  
  let $\QX_d$ be the sub-quandle of $\QX$ formed by elements with diagonal matrices; if $L$ is a $\QX_d$-link then by definition we have $M'(L)=ADO(\pi_*(L))$. 
  The restriction
  $\pi_{|\QX_d}:\QX_d\to\C$
   is a bijection.  
Theorem \ref{T:F'=M'} and  Lemma \ref{L:SCunknotGauge} imply that $M'$ is
  gauge invariant.  Fixing a regular projection of a $\QX$-link $L$,
  we get a diagram $D$.  From the algebraic formulas for the holonomy
  R-matrix of \cite{GP}, one can see that the invariant $M'$ is a
  continuous function of the $X$-colors of the edges of $\Qm^{-1}(D)$
  and so of the finite set of $\QX$-colors of the edges of $D$.  Let
  $x_\kappa\in\QX$ be an element whose matrix part is 
  ${\small \left(\begin{array}{cc}
          \kappa^{-1}&0\\0&\kappa \end{array}\right)}$.  Then the
    colors of $x_\kappa\rhd D$ are the same as those of $D$ except
    that their upper diagonal coefficient is multiplied by $\kappa^2$.
    Hence we have that $\pi_*L=\lim_{\kappa\to 0}x_\kappa\rhd L$
    (after identifying $\QX_d\cong\C$).  Thus, using the continuity of $M'$, we get
    $$M'(L)=M'(\lim_{\kappa\to 0}x_\kappa\rhd L)=ADO(\pi_*L).$$
\end{proof}

\section{Generically defined biquandles in pivotal categories}\label{S:GenericBiqu}
Recall that in Section \ref{S:Quandles}, we used quandles to give the
needed topological notion of this paper.  Then in Sections
\ref{S:Biquandles} and \ref{S:BiquandleBraidings} we used biquandles
as the algebraic objects associated to this topological notion.  These
settings are related by the functor $ \Qm$ given in Theorem
\ref{th:LV}, which sends a biquandle coloring to a quandle coloring
and is equivariant
for gauge transformations.

The rest of the paper is motivated by our main example which is the
conjugacy quandle of the group $ \SL_2(\C)$, see Example
\ref{E:ConjQuandle}.  The topological setting for this example is a
locally flat $ \SL_2(\C)$-bundle over the complement of a tangle.  In
this example the underlying algebraic object is not a biquandle but
instead what we call a generic biquandle factorization.  This is
almost a biquandle in the sense that there is a ``generically''
defined binary operation $B$ which (generically) satisfies the
Yang-Baxter equation and the other axioms of a biquandle.  What is
interesting here is that the topological notion is defined everywhere
but the associated algebraic object is not.  In other words, quandle
colorings always exist but not all biquandle colorings
exist.
Thus, in the setting of this section,
we do not have a bijective functor $\Qm$ which relates biquandle and
quandle
colorings.
Instead, we require there exists a map $\Qm$ from biquandle colorings
to quandle colorings which is injective and up to gauge transformation
onto (loosely speaking, one could say the map $\Qm$ is generically
bijective).

Let us explain how this is useful in our example.  We start with a locally flat $ \SL_2(\C)$-bundle over the complement of a tangle.   This is interrupted in terms of  a quandle closely related to $\Conj( \SL_2(\C))$.  This quandle has a generic biquandle factorization which is related to the representation theory of the unrestricted quantum group.  Up to gauge, the bundle gives a generic biquandle coloring of a diagram of the tangle.  Then representation theory of unrestricted quantum group can be used to get an invariant of the tangle  which only depends on the bundle.

\subsection{Generically defined biquandle}\ 
Let $Y$ be a set and $\gen \subset \mathcal{P}(Y)$ a set of subset of $Y$ such that 
\begin{enumerate}
\item $Y \in \gen$, $\emptyset \notin \gen$, 
\item if $Z_1, Z_2\in \gen$ then $Z_1 \cap Z_2\in \gen$.
\end{enumerate}
A natural example comes from taking $Y$ as a non empty topological space then the set of open
dense subsets $\gen$ of $Y$ satisfies the above axioms (other examples
are the set of full measure subspaces of a measured space or the set
of comeagre sets of a Baire space).

We use the following terminology:  
\begin{description}
\item[Partial map] By a \emph{partial map} $f:A\to B$ we mean a map whose domain is some subset of $A$ and range is~$B$.
\item[Generic bijection] By a {\em generic bijection} of $Y$  we mean a bijection $f:Z_1\to Z_2$ where $Z_1,Z_2\in\gen$ such that for any $Z\in\gen$, the sets $f(Z\cap Z_1)$ and $f^{-1}(Z\cap Z_2)$ are elements of $\gen$.
\item[Generic $x$ in $Y$] If $\mathscr P(x)$ is a statement partially defined  for $x\in Y$ then ``$\mathscr P(x)$ is true for generic $x\in Y$'' means ``there exists $ Z\in\gen$ such that $\mathscr P(x)$ is true for all $x\in Z$.''
\end{description}

Next we give the notion of a generically defined biquandle associated to a quandle.  This definition requires the existence of a functor $\Qm$ relating biquandle colorings and quandle colorings.  This functor is modeled after the bijective functor $\Qm$ given in Theorem \ref{th:LV} (we use the same notation for the map and functor).   In Section \ref{S:Biquandles}, we showed the bijective functor $\Qm$ was compatible with $B$-gauge transformations, this is done using the action $w\gm\any$, see both Proposition~\ref{P:twisttensor} and the proof of Proposition \ref{P:B-gauge=harpoons}.   In this section, this compatibility is expressed using a map, which we also denote by $w\gm\any$.

\begin{definition}\label{d:genbirack}
  Let $(\QX,\rhd)$ be a quandle.  A \emph{generic biquandle
    factorization} of $\QX$ is a tuple $(Y,\gen,B,\Qm,\gm)$ equipped
  with four partial maps
  $$B, S, B^{-1},S^{-1}: Y\times Y \rightarrow Y\times Y$$ 
  satisfying the following axioms:
 \begin{enumerate}
\item \label{I:BandSrelateGenDef}
 For any $(x_1,x_2,x_3,x_4)\in  Y^4$,
 $$\begin{array}[t]{rl}  (x_4,x_3)=B(x_1,x_2)&\iff(x_1,x_2)=B^{-1}(x_4,x_3)\\  
& \iff (x_3,x_2)=S(x_4,x_1)\\
 &\iff(x_4,x_1)=S^{-1}(x_3,x_2).\end{array}$$
\item \label{I:BBSSGenDef} For all $x\in Y$, the eight maps
  $B_1^\pm(x,\any),\, B_2^\pm(\any,x),\, S_1^\pm(x,\any)$ and 
  $S_2^\pm(\any,x)$ are 
  generic bijections of $Y$.
\item There exists a function $Y\to\Aut(\QX,\rhd)$,
  $x\mapsto (x\gm\any)$.  Since $x\gm\any$ is a quandle morphism it induces a functor 
    $$x\gm\any :\Diag_\QX\to\Diag_\QX$$
    see Subsection \ref{SS:GaugeActionBiq}  and Equation \eqref{eq:Q-functor}.
\item As in the case of a biquandle, we use $B$ to define $Y$-colored diagrams that form a
  category $\Diag_Y$ (see Remark \ref{r:notBiquandle}).  Then there exists a functor $\Qm :\Diag_Y\to\Diag_\QX$
  inducing the identity on the underlying uncolored diagrams, which is
  injective on $Y$-colored diagram and satisfies:
 \begin{equation}\label{E:Qmxw}
 \Qm\big((x,\pm)\otimes w\big)  =\Qm\big((x,\pm)\big)\otimes  (x\gm\any)^{\pm1}\big(\Qm(w)\big)
 \end{equation}
 for all $x\in Y$ and $w\in W_Y$.  
\item \label{I:QinvxD}  
For any $D\in\Diag_\QX$,  $\Qm^{-1}(x\gm D)$ exists for generic $x\in Y$.
\end{enumerate}

\end{definition}

 Let $(Y,\gen,B,\Qm, \gm)$ be a generic biquandle
    factorization of a quandle $(\QX,\rhd)$.  
Let $D$ and $D'$ be two diagrams of a tangle which are related by a
Reidemeister move.  
If $D$ is $Y$-colored then 
it may happen that the coloring of $D$ does not induce a coloring on
$D'$ via the Reidemeister move.
In the case when a Reidemeister move is colored we use the following notation: if two colored diagrams $D$ and $D'$ in $\Diag_Y$ or in $\Diag_\QX$ are related by a single colored Reidemeister move we write $D\stackrel1\equiv D'$.  
\begin{proposition}\label{P:guitarRM}
   The functor $\Qm$ is compatible with colored Reidemeister moves:
   For any $D,D'\in\Diag_Y$ we have 
   $D\stackrel1\equiv D'\iff \Qm(D)\stackrel1\equiv\Qm(D').$
\end{proposition}
The proof of the proposition is given in Appendix \ref{A:ProofPropTheorGat}.
The proposition implies the functor $\Qm$ induces a functor 
$$\wt \Qm: \Diag_Y/\text{colored Reidemeister moves} \to \T_\QX.  $$

\begin{remark}
  Proposition \ref{P:guitarRM} implies that whenever both side are defined,
  the set Yang-Baxter equation for $B$ holds.  Also, if
  $(x_4,x_3)=B(x_1,x_2)$ then $x_1=x_4\iff x_2=x_3$.  Indeed, if one
  of the two equalities holds, we can form the $Y$-colored diagram of
  a twist with colors $\{x_1,x_2,x_3,x_4\}$.  Then its image by $\Qm$
  is an endomorphism which implies that the two other colors are
  equal.
\end{remark}

As mentioned above $Y$-colored Reidemeister moves may not exist; however, the following theorem says they generically exist after tensoring with the identity. 
\begin{theorem}\label{T:stableRM}
  Let $D$ and $D'$ be two $Y$-colored diagrams such that $\Qm(D)$ and
  $\Qm(D')$ represent isotopic $\QX$-tangles.  Then 
 for generic $x\in Y$ the diagrams 
  $\Id_{(x,+)}\otimes D$ and $\Id_{(x,+)}\otimes D'$ are related by a
  sequence of $Y$-colored Reidemeister moves.
\end{theorem}
The proof of Theorem \ref{T:stableRM} is given in Appendix \ref{A:ProofPropTheorGat}.  

\subsection{Example: generic Lie group factorization of $\SL_2(\C)$}\label{ss:glgf}
We will give an example of a generic biquandle factorization of the conjugacy quandle of $\SL_2(\C)$.  To do this we first consider a generic Lie group factorization which is closely related to the
generalized group factorization of Example \ref{s:G+G-} in the case of $\SL_2(\C)$.

Let $\Gd=\SL_2(\C)$, $\wb \Gd=\GL_2(\C)$ and
$$\Grc=\qn{ 
  \bp{
    \bp{\begin{array}{cc} \kappa&0\\\vp&1  \end{array}}
    ,
    \bp{\begin{array}{cc} 1&\ve\\0&\kappa \end{array}}
    }
  :\ve,\vp\in\C,\kappa\in\C^*}
\subset\wb \Gd\times\wb \Gd.$$
Let $\vp_+, \vp_-:\Grc\to\wb\Gd$  be the maps defined by 
$$\vp_+((M_1, M_2))=M_1, \;\;\; \vp_-((M_1, M_2))=M_2$$
where $(M_1, M_2)\in \Grc$.  
Then $\psi:\Grc\to\Gd$ is given by 
\begin{equation}\label{E:DefpsiSL2}
\psi \bp{\bp{\begin{array}{cc}
      \kappa&0\\\vp&1
    \end{array}},\bp{\begin{array}{cc}
      1&\ve\\0&\kappa
    \end{array}}}=\bp{\begin{array}{cc}
      \kappa&0\\\vp&1
    \end{array}}\bp{\begin{array}{cc}
      1&\ve\\0&\kappa
    \end{array}}^{-1}=\bp{\begin{array}{cc}
      \kappa&-\ve\\\vp&\dfrac1\kappa-\dfrac{\ve\vp}\kappa
    \end{array}}.
    \end{equation}
  The map $\psi$ is a bijection from $\Grc$ to the Zariski open dense
  subset $\Gd'$ of $\Gd$, where $\Gd'$ is the set of determinant $1$ matrices $M=(m_{ij})$ such that $m_{11}\neq0$.  Remark that each conjugacy class 
  $C$ 
  in $\Gd$ is
  non-empty, connected and contains a trigonal matrix then since $\Gd'$ contains all
  invertible trigonal matrices we have 
  $C\cap \Gd'$
   is a Zariski open
  dense subset of 
  $C$. 

Next we explain how $Y=\Grc$ and $\gen$ extends to a generic biquandle factorization of the quandle $\Conj(\SL_2(\C))$.  First, 
analogously to the case of a generalized group factorization the tuple $(\Gd, \wb\Gd, \Grc, \vp_+,\vp_-)$ gives rise to a partial map $B:\Grc\times\Grc\to\Grc\times\Grc$ coming from solving Equations \eqref{eq:cross}.   
In particular, for $y_1,y_2\in
\Grc$, the elements $(y_4,y_3)=B(y_1,y_2)$ are
given by 
\begin{equation}\label{E:Formulasy4y3}
y_4=\psi^{-1}\bp{\vp_-(y_1)\psi(y_2)\vp_-(y_1)^{-1}} \text{ and }
y_3=\psi^{-1}\bp{\vp_+(y_4)^{-1}\psi(y_1)\vp_+(y_4)}.
\end{equation}  This
definition make sense if ${\vp_-(y_1)\psi(y_2)\vp_-(y_1)^{-1}}$ and
$\vp_+(y_4)^{-1}\psi(y_1)\vp_+(y_4)$ are in $\Gd'$.  
These formulas also imply that $B$ has a partially defined inverse $B^{-1}$.   As in the group factorization,  Equation~\eqref{E:DefS} defines a partially defined invertible sideways map $S$. 

Let  $\QX=(\Conj(\SL_2(\C)),\rhd)$ be the conjugacy quandle structure on $\Gd$, see Example~\ref{E:ConjQuandle}.
Define the harpoon automorphisms as follows: for any $x\in Y$ it is easy to show
$$x\gm \any\,=\,\vp_+(x)^{-1}\rhd\any\,=\,\vp_+(x)\,\any\,\vp_+(x)^{-1}: \QX\to \QX$$ is a quandle automorphism.  
Finally, the definition of the functor $\Qm:\Diag_\Grc\to \Diag_\Gd$ was given in \cite{KR}, also see \cite{GP}.   In Appendix \ref{Ap:ProofSL2CGenericFact} we recall this definition and prove the following theorem:
\begin{theorem}\label{T:YisGenericBiquConjSL2C}
The tuple 
$$(Y=\Grc,\gen=\{\text{Zariski open subsets of $Y$}\},B,\Qm,\gm)$$
 is a generic biquandle factorization of the quandle $\Conj(\SL_2(\C))$.  
\end{theorem}

As we will now explain, the generic biquandle factorization of Theorem~\ref{T:YisGenericBiquConjSL2C} can be extended to a generic biquandle factorization  $(Y', \gen',B',\Qm',\gm)$ of the quandle $\Conj(\SL_2(\C))\times_{\Ch_r} \C$ defined using a fibered product:  Following
Example \ref{ss:fp} let $f:Y=\Grc\to \C$ be given by
$f(y)=\operatorname{trace}(\psi(y))$.   If $(y_4,y_3)=B(y_1,y_2)$ then using
the formulas for $y_3$ and $y_4$ given in Equation \eqref{E:Formulasy4y3} it
is easy to see $f(y_1)=f(y_3)$ and $f(y_2)=f(y_4)$.  
Let $\ell\geq 3$ and set $r=\ell/2$ if $\ell$ is even and $r=\ell$ else.
Let $\Ch_r:\C\to\C$ be the renormalized $r^{th}$
Chebyshev polynomial (determined by $\Ch_r(2\cos\theta)=2\cos(r\theta)$).
Then 
\begin{equation}\label{E:DefOfY'}
Y'=Y\times_{(f,\Ch_r)}\C=\{(y,z)\in Y\times \C :f(y)=\Ch_r(z)\}
\end{equation} is naturally
equipped with a partial map given by
\begin{equation}\label{E:DefB'}
B'((y_1,z),(y_2,z'))=((y_4,z'),(y_3,z))
\end{equation}
which is defined whenever $B(y_1,y_2)=(y_4,y_3)$.  This assignment naturally
gives maps $(B')^{-1}$ and $(S')^{\pm}$ which satisfy Axiom
\eqref{I:BandSrelateGenDef} of Definition \ref{d:genbirack}.  Setting 
$\gen'=\{Z\times \C \cap Y': Z\in \gen\}$ 
then these maps satisfy  Axiom \eqref{I:BBSSGenDef}.

The quandle $\QX'=\Conj(\SL_2(\C))\times_{\Ch_r} \C$ is defined as follows.   The elements of the quandle are pairs $(y,z)$ such that $y\in \SL_2(\C)$ and $\operatorname{trace}(y)=\Ch_r(z)$.  The quandle structure is given by 
$(y,z) \rhd (y',z') =(y^{-1}y'y, z')$.  
There is a function $Y'\to \Aut(\QX',\rhd)$ given by $(y,z)\gm\any=(\vp_+(y)^{-1},z)\rhd\any $.   Finally, the functor $\Qm:\Diag_Y\to\Diag_{\QX}$ extends to a functor $\Qm':\Diag_{Y'}\to\Diag_{\QX'}$ which satisfies the last two axiom of Definition \ref{d:genbirack}.    Thus, the  tuple $(Y', \gen',B',\Qm',\gm)$ is a generic biquandle factorization of the quandle $\Conj(\SL_2(\C))\times_{\Ch_r} \C$.  

In our main example, given in Section \ref{S:MainExCycle}, we need to consider a sub-quandle of $\Conj(\SL_2(\C))\times_{\Ch_r} \C$ and its corresponding factorization.  Let us discuss this now.  Let $\QX_{\rel}$  be the sub-quandle of $\QX'$ defined by $$
 \QX_{\rel}=\{(x,z) \in \QX' : \operatorname{trace}(x)=\Ch_r(z)\neq \pm 2 \}\cup \left\{\left((-1)^{r-1}\Id_{2\times 2}, 2(-1)^{\ell-1} \right) \right\}.
 $$
Let 
$$Y_{\rel}=\{(y,z)\in Y' : \Ch_r(z)\neq \pm 2 \}\cup
\left\{\left(\psi^{-1}\left((-1)^{r-1}\Id_{2\times 2}\right), 2(-1)^{\ell-1} \right)\right\}.
 $$
Then the maps $(B')^{\pm}$ and  $(S')^{\pm}$ restrict to maps $B^{\pm}_{\rel}$ and  $S^{\pm}_{\rel}$ which satisfy Axiom 
\eqref{I:BandSrelateGenDef} of Definition \ref{d:genbirack}.  Setting 
$\gen_{\rel}=\{Z\times \C \cap Y_{\rel}: Z\in \gen\}$ 
then these maps satisfy  Axiom \eqref{I:BBSSGenDef}.   Finally, the functor $\Qm':\Diag_{Y'}\to\Diag_{\QX'}$ restricts to a functor $\Qm_{\rel}:\Diag_{Y_{\rel}}\to\Diag_{\QX_{\rel}}$ which satisfies the last two axiom of Definition~\ref{d:genbirack}.   
Thus, the results of this subsection can be summarized by the following theorem.
\begin{theorem}\label{T:ExistY'genB'}
  The tuple $(Y_{\rel}, \gen_{\rel},B_{\rel},\Qm_\ell,\gm)$ is a
  generic biquandle factorization of the quandle $\QX_{\rel}$ which is
  a sub-quandle of $\Conj(\SL_2(\C))\times_{\Ch_r} \C$.
\end{theorem}

\subsection{Representation of a generically defined birack}
Here we give the notion of a representation of a generic biquandle factorization of a quandle.  This is almost exactly the same as the biquandle representations defined in Subsection~\ref{SS:BiquandleRep} except that we only require morphisms to be defined when the generic biquandle is defined.    We use the notation and terminology of Subsection \ref{SS:BiquandleRep}.  

Let $(Y,\gen,B,\Qm, \gm)$ be a generic biquandle factorization of
a quandle $(\QX,\rhd)$ and $\cat$ be a pivotal
category.  Let $A$
be the set of $(x,y)\in Y\times Y$ such that $B(x,y)$ is defined.  A
\emph{representation} of $Y$ in $\cat$ is a family of objects 
$\{V_y\}_{y\in Y}$ and a family of isomorphisms of $\cat$
$$\{c_{x,y}:V_x\otimes V_y\to
B_1(V_x,V_y)\otimes B_2(V_x,V_y)\}_{(x,y)\in A}$$
which satisfy the colored braid relation, is sideways invertible and induces a twist 
whenever relevant elements of the biquandle $B$ are defined.

As in the case of  biquandles (see Theorem \ref{T:FunctorF}), there is a well
defined functor at level of $Y$-diagrams:
\begin{theorem}
  Let $(V_\any,c_{\any,\any})$ be a representation of $Y$ in a
  pivotal
  category $\cat$.  Then there exists a unique tensor
  functor $$F:\Diag_Y\to\cat,$$ such that for any $x,y$ in $Y$ one has
  $F((x,+))=V_x$, $F((x,-))=V_x^*$ and the morphisms are determined by its
  value on the elementary tangle diagrams:
  $$F({\chi^+_{x,y}})=c_{x,y}, \; F({\chi^-_{x,y}})=c_{x,y}^{-1} \;\text{ if $(x,y)$ is in the domain of $B^\pm$}, $$
  $$F(\lev_x)=\lev_{V_x}, F(\lcoev_x)=\lcoev_{V_x}, F(\rev_x)=\rev_{V_x} \text{ and }  F(\rcoev_x)=\rcoev_{V_x}.$$  
  Furthermore, the functor $F$ is invariant by $Y$-colored
  Reidemeister moves.
\end{theorem}
\begin{proof}
The proof is the same as the proof of Theorem \ref{T:FunctorF}. 
\end{proof}
Two $Y$-colored diagrams representing the same $\QX$-tangle are not
necessarily related by a sequence of $Y$-colored Reidemeister moves.
Nevertheless, we have:
\begin{theorem} Assume that the representation of $Y$ is regular as in
  Definition~\ref{D:reg}.  If $D$ and $D'$ are two $Y$-diagrams such
  that $\Qm(D)$ and $\Qm(D')$ represent isotopic $\QX$-tangles then
  $F(D)=F(D')$.
\end{theorem}
\begin{proof}
  By Theorem \ref{T:stableRM}, for generic $x\in Y$, $\Id_x\otimes
  D\equiv\Id_x\otimes D'\in\Diag_Y$.  As $F$ is invariant by
  $Y$-colored Reidemeister moves, we have that 
  $$\Id_{V_x}\otimes
  F(D)=F(\Id_x\otimes D)=F(\Id_x\otimes D')=\Id_{V_x}\otimes F(D').$$
  Now the theorem follows from the fact that $V_x$ is regular.
\end{proof}

There is a natural equivalence on $\QX$-tangles, which we call \emph{$B$-gauge equivalence}, generated by the automorphisms $(a\rhd\any)_{a\in\QX}$ and
$(x\gm\any)_{x\in Y}$.
The terminology is justified by Equation
\eqref{eq:gp}.  As in the case of biquandle representations, $F$
restricts to a $\kk$-valued $B$-gauge invariant function for links or
1-1-tangles:
\begin{theorem}\label{c:ginv} 
Assume that the representation of $Y$ is regular and simple as in
  Definition~\ref{D:reg}.
  Let $D,D'$ be two $Y$-colored diagrams such that $\Qm(D)$ and $\Qm(D')$ are $B$-gauge equivalent 1-1 $\QX$-tangles (resp.\ $\QX$-links) then $\brk{F(D)}=\brk{F(D')}$ (resp.\  $F(D)=F(D')$).
\end{theorem}
The proof of Theorem \ref{c:ginv} is given in Appendix \ref{Ap:ProofThmFBgauge}.

\subsection{$Q$-link invariants}
Here  we define two invariants of $\QX$-links, using the following hypothesis:
\begin{enumerate}
\item $\QX$ is a quandle and $L_\QX$ is the set of isotopy class of $\QX$-links,
\item $(Y, \gen, B, \Qm, \gm)$ is a generic biquandle factorization of $\QX$,
\item $(V_\any,c_{\any,\any})$ is a regular simple representation of
  $Y$ in a pivotal category $\cat$ and $\kk=\End_\cat(\unit)$.
\end{enumerate}
\begin{theorem} Under these hypothesis,
 the partial map $F\circ\wt \Qm^{-1} :L_\QX\to\kk$
   extends uniquely to a gauge invariant globally defined map 
   $$
   \wt F: L_\QX\to\kk \text{ given by } L\mapsto F\circ\wt \Qm^{-1}(x \gm L)
   $$
 where $x$ is any element of $ Y$ such that $\Qm^{-1}(x \gm L)$ is defined.  
\end{theorem}
\begin{proof}
  This is an immediate consequence of Theorem \ref{c:ginv}.
\end{proof}
In the main examples of this paper this map is trivial.  However, 
one can renormalize $\wt F$ using modified dimension: let us assume
\begin{enumerate}
\item[(4)] There is a gauge invariant modified dimension function
  $\qd$ on $\{V_x\}_{x\in Y}$ such that $(V_\any,\qd)$ is an ambi
  pair.
\end{enumerate}
Then recall the map $F'$ is defined on closed $Y$-colored diagrams, see 
Equation~\eqref{eq:md}.
\begin{theorem} \label{T:DefF'Generic} Under the four hypothesis listed in this subsection, the partial map
 $F'\circ\wt \Qm^{-1}:\Link_\QX\to\kk$
 extends uniquely to a gauge invariant globally defined map 
   $$
   \wt F': L_\QX\to\kk \text{ given by } L\mapsto F'\circ\wt \Qm^{-1}(x \gm L)
   $$
 where $x$ is any element of $ Y$ such that $\Qm^{-1}(x \gm L)$ is defined.  
\end{theorem}
\begin{proof}
  Clearly, $\qd$ induces a map on gauge equivalence class of elements
  of $\QX$. Theorem \ref{c:ginv} implies that $\brk{\wt F(\any)}$
  uniquely extends to a globally defined gauge invariant map on 1-1
  $\QX$-tangles.
  If $T_a$ is a 1-1 $\QX$-tangle whose open strand is colored by $a$ then the map $\wt F'$ is given on the braid closure $\wh T_a$ of $T_a$ by $\qd(a)\brk{\wt F(T_a)}$.  
  This
  clearly extends uniquely to a global gauge invariant
  $Q$-link invariant.
\end{proof}
\begin{conclusion}\label{SS:conclusionGenericInv}
The following steps can be used to compute $\wt F'(L)$:
\begin{enumerate}
\item Choose a diagram $D\in\Diag_\QX$ representing $L$.
\item If $\Qm^{-1}(D)$ exists then let $D'=D$ else, replace $D$ by a gauge
  equivalent diagram $D'=y\gm D$ (for some $y\in Y$) which is in the
  image of $\Qm$.
\item Consider any diagram $D'_x$ which is a cutting presentation of
$\Qm^{-1}(D')$
  and compute its image by $F$:
  $F(D'_x)=\brk{F(D'_x)}.\Id_{V_x}$.
\item Finally, multiply this bracket with the modified dimension
  $\qd(V_x)$:
$$\wt F'(L)=\qd(V_x)\brk{F(D'_x)}.$$
\end{enumerate}
\end{conclusion}

\section{A biquandle representation from cyclic quantum $\slt$ modules}\label{S:MainExCycle}
Recall the generic biquandle factorization $(Y_{\rel}, \gen_{\rel},B_{\rel},\Qm_\ell,\gm)$ of the quandle $\QX_{\rel}$ given in Theorem \ref{T:ExistY'genB'}.  In this section we define a representation of $Y_{\rel}$ from cyclic quantum $\slt$ modules.  There are two main objects needed to do this:  a holonomic braiding and a modified trace.  In the context of quantum $\slt$ both of these objects have been studied in 
\cite{KR1} and \cite{GP4}, respectively.   In Subsections \ref{SS:UslRRmatR} and  \ref{SS:ModTraceQuSl} we review the material we need from these papers.  Then in Subsection \ref{SS:NegReidIIMove} we prove the holonomic braiding leads to a generically defined Yang-Baxter model which is sideways invertible.  Finally in Subsection \ref{SS:Twist} we show the representation induces a twist.

\subsection{The algebra $U_q=\Uq$ and cyclic $U_\xi$-modules}\label{SS:CyclicMod}
Let $U_q=\Uq$ be the $\C[q]$-algebra given by generators $E, F, K, K^{-1}$ 
and relations:
\begin{align}\label{E:RelDCUqsl}
  KEK^{-1}&=q^2E, & KFK^{-1}&=q^{-2}F, &
  [E,F]&=\frac{K-K^{-1}}{q-q^{-1}}.
\end{align}
The algebra $\Uq$ is a Hopf algebra where the coproduct, counit and
antipode are defined by
\begin{align}\label{E:HopfAlgDCUqsl}
  \Delta(E)&= 1\otimes E + E\otimes K, 
  &\varepsilon(E)&= 0, 
  &S(E)&=-EK^{-1}, 
  \\
  \Delta(F)&=K^{-1} \otimes F + F\otimes 1,  
  &\varepsilon(F)&=0,& S(F)&=-KF,
    \\
  \Delta(K)&=K\otimes K
  &\varepsilon(K)&=1,
  & S(K)&=K^{-1}.
\end{align}

Let $\xi=\e^{2i\pi/\ell}$ be a $\ell$ root of unity.   Set $r=\ell/2$ if $\ell$ is even and $r=\ell$ else.
Let $U_\xi=U_\xi\slt$ be the specialization of $\Uq$ at $q=\xi$.  
 Let $Z_0$ be the subalgebra of $U_\xi$ generated by the central elements $K^{\pm r},E^r$ and $F^r$.  
The center $Z$ of $U_\xi$ is generated by $Z_0$ and
the Casimir element 
$$\Omega=\qn1^2EF+K\xi^{-1}+K^{-1}\xi=\qn1^2FE+K\xi+K^{-1}\xi^{-1}$$
where $\qn x=\xi^x-\xi^{-x}$.
The Casimir satisfies the polynomial equation:
\begin{equation}\label{E:ChOmega}
\Ch_r\bp{\Omega}=\qn1^{2r}E^rF^r-(-1)^\ell (K^r+K^{-r})\in Z_0
\end{equation} where $\Ch_r$
is the renormalized $r^{th}$ Chebyshev polynomial
(determined by $\Ch_r(2\cos\theta)=2\cos(r\theta)$). 

By a \emph{$U_\xi$-weight module} we mean a finite-dimensional module over $U_\xi$ which restrict to a semi-simple module over $Z_0$.  Let $\cat$ be the tensor category of $U_\xi$-weight modules.  
The category $\cat$ is a pivotal $\C$-category where for any 
object $V$ in $\cat$, the dual object and the duality
morphisms are defined as follows: $V^* =\Hom_\C(V,\C)$ and
\begin{align}\label{E:DualityForCat}
  \coev_{V} :\, & \C \rightarrow V\otimes V^{*} \text{ is given by } 1 \mapsto
  \sum
  v_j\otimes v_j^*,\notag\\
  \ev_{V}:\, & V^*\otimes V\rightarrow \C \text{ is given by }
  f\otimes w \mapsto f(w),\notag\\
  \tev_{V}:\, & V\otimes V^{*}\rightarrow \C \text{ is given by } v\otimes f
  \mapsto f(K^{1-{r}}v),\notag
  \\
  \tcoev_V:\, & \C \rightarrow V^*\otimes V \text{ is given by } 1 \mapsto \sum
  v_j^*\otimes K^{{r}-1}v_j,
\end{align}
where $\{v_j\}$ is a basis of $V$ and $\{v_j^*\}$ is the dual basis of $V^*$.

The set of characters on $Z_0$ (resp.\ $Z$) is $\Hom_{alg}(Z_0, \C)$
(resp. $\Hom_{alg}(Z, \C)$).
The set $\Hom_{alg}(Z_0, \C)$ is a group where the multiplication is given by
$\chi_1\chi_2 =(\chi_1\otimes \chi_2) \Delta$.  For each $\chi$
character on $Z_0$ let $\cat_{\chi}$ be the full subcategory of $\cat$
whose objects are modules where each $z \in Z_0$ acts by $\chi(z)$.
At the end of the next subsection we will see that $\cat_{\chi}$ is
semi-simple if $\chi(\Ch_r(\Omega))\neq \pm2$.  For such a character
$\chi$ the simple modules in $\cat_{\chi}$ are called \emph{cyclic};
the name comes from the fact that $E^r$ and $F^r$ can act by non-zero scalars creating a
circular diagram depicting the action on the weight vectors.

\newcommand{\rr}{{r^2}}
\newcommand{\catro}{\cat_{/\rr}}

\subsection{The algebra isomorphism $\RR$ and the matrix $R$} \label{SS:UslRRmatR} 
Let $U_h=U_h(\slt)$ be h-adic completion version of the $\Uq$.
  Following \cite{KR1}, there exists an algebra isomorphism
$$\RR: U_h\otimes U_h\to U_h\otimes U_h$$
given by conjugation of the $R$-matrix in the h-adic completion.    
As in \cite{KR1}, this isomorphism induces an outer automorphism of the division ring $\overline{U_\xi^{\otimes 2}}$ of $U_\xi^{\otimes 2}$ which we also denote by the same letter $\RR$.  In particular, it induces a map
\begin{equation}
  \label{eq:curlyR}
  \RR:U_\xi\otimes U_\xi\to U_\xi\otimes U_\xi[W^{-1}]
\end{equation}
where $$\W=\bp{1-q^r\qn1^{2r}(K^{-1}E)^r\otimes (FK)^r}
=\bp{1+(-1)^\ell\qn1^{2r}K^{-r}E^r\otimes F^rK^r},$$
which satisfies
\begin{equation}\label{E:DeltaRR}
(\Delta \otimes 1)\RR (u\otimes v)= \RR_{13}\RR_{23}(\Delta (u)\otimes v) \text{ and } (1 \otimes \Delta)\RR (u\otimes v)= \RR_{13}\RR_{12}(u\otimes \Delta (v))
\end{equation}
and 
\begin{equation}\label{E:epsilonRR}
(\epsilon \otimes 1)\RR (u\otimes v)= \epsilon(u)v \text{ and } (1 \otimes \epsilon)\RR (u\otimes v)= \epsilon(v)u.
\end{equation}

The map $\RR$ is given on $Z_0\otimes Z_0$ by
\begin{align}\label{E:RRGenZ0}
\RR(K^r\otimes 1)=(K^r\otimes 1)\W, \;\;\; \RR(1\otimes K^r)=(1\otimes K^r)\W^{-1},\notag \\
\RR(E^r\otimes 1)=E^r\otimes K^r, \;\;\; \RR(1\otimes F^r)=K^{-r}\otimes F^r,\notag \\
\RR(1\otimes E^r)=K^r\otimes E^r+E^r\otimes1\bp{1-(1\otimes K^{2r})\W^{-1}}, \notag\\
\RR(F^r\otimes 1)=F^r\otimes K^{-r}+1\otimes F^r\bp{1-(K^{-2r}\otimes 1)\W^{-1}}.
\end{align}

Recall the generic Lie group factorization associated to $\SL_2(\C)$ given in Subsection \ref{ss:glgf}:  $\Gd=\SL_2(\C)$, $\wb \Gd=\GL_2(\C)$ and 
\begin{equation}\label{E:Mkappavevp}
\Grc=\qn{ 
  M\bp{\kappa,\ve,\vp}=\bp{
    \bp{\begin{array}{cc} \kappa&0\\\vp&1  \end{array}}
    ,
    \bp{\begin{array}{cc} 1&\ve\\0&\kappa \end{array}}
    }
  :\ve,\vp\in\C,\kappa\in\C^*}
\end{equation}
This factorization was used to define a generic biquandle factorization of $\QX=\Conj(\SL_2(\C))$ with a map $B: \Grc \times \Grc \to   \Grc \times \Grc$.  Next we show that this map also arises from $\RR$ and characters on $Z_0$.

Given a character $\chi$ on $Z_0$, let 
$$\vp_+(\chi)=\bp{\begin{array}{cc}
      \kappa&0\\\vp&1
      \end{array}} \;\text{  and } \; \vp_-(\chi)=\bp{\begin{array}{cc}
      1&\ve\\0&\kappa
    \end{array}}$$
   where $\kappa=\chi(K^r),\,\ve=\qn1^{r}\chi(E^r)$, and $\vp=(-1)^\ell\qn1^{r}\chi(K^rF^r)$.
   Also, let $\psi(\chi)=\vp_+(\chi)\vp_-(\chi)^{-1}$.   
We identify $\Grc$ with characters on $Z_0$  by 
$$M\bp{\kappa,\ve,\vp}(K^r)=\kappa,\quad M\bp{\kappa,\ve,\vp}(E^r)=\qn1^{-r}\ve\et
$$
$$
M\bp{\kappa,\ve,\vp}(F^r)=(-1)^\ell\qn1^{-r}\vp\kappa^{-1}.
$$

Let $\chi_1$ and $ \chi_2$ be two characters on $Z_0$ which we identify with $M\bp{\kappa_1,\ve_1,\vp_1}$ and $M\bp{\kappa_2,\ve_2,\vp_2}$, respectively.   One easily checks that $\chi_1\chi_2=\chi_1\otimes \chi_2\circ\Delta$. 

The equalities in Equation \eqref{E:RRGenZ0} imply $\RR$ acts invariantly on $Z_0\otimes Z_0$.  So the transpose of the map $\RR^{-1}\circ\tau:Z_0\otimes Z_0\to Z_0\otimes
Z_0$  (after identification) is a partial map $B:\Grc\times\Grc\to \Grc\times\Grc$.  
\begin{lemma}
This partial map is equal to the partial map $B$ given in Example~\ref{ss:glgf}.  Thus, it extends to the generic biquandle factorization given in Theorem~\ref{T:YisGenericBiquConjSL2C}. 
\end{lemma}
\begin{proof}
To show that $B$ is given by the generic Lie group factorization of $\SL_2(\C)$ it
is enough to see that $B$ satisfy the system \eqref{eq:cross}.  

Let
$(\chi_4,\chi_3)=B(\chi_1,\chi_2)$ then
\begin{equation}\label{E:DefiningRelRChar}
\chi_3\otimes \chi_4\circ\RR=\chi_1\otimes \chi_2\in\Hom_{\Alg}(Z_0\otimes Z_0,\C).
\end{equation}  The
first equation is easily checked:
$$\chi_4\chi_3=\chi_4\otimes \chi_3\circ\Delta=\chi_3\otimes
\chi_4\circ\RR\circ\Delta=\chi_1\otimes
\chi_2\circ\Delta=\chi_1\chi_2.$$
To check the second equation of \eqref{eq:cross},
using Equations \eqref{E:RRGenZ0}
and \eqref{E:DefiningRelRChar} we compute  $\vp_-(\chi_1)$ and $\vp_+(\chi_2)$.  From  these  computations we have
\begin{align*}\vp_-(\chi_1)\vp_+(\chi_2)
&=\bp{\begin{array}{cc}
      1&\kappa_4\ve_3\\0&\kappa_3\omega
    \end{array}}
\bp{\begin{array}{cc}
      \kappa_4\omega^{-1}&0\\\vp_4\kappa_3^{-1}\omega^{-1}&1
    \end{array}}\\
 & =\bp{\begin{array}{cc}
      \kappa_4&\kappa_4\ve_3\\\vp_4&\kappa_3+\vp_4\ve_3
    \end{array}}=\vp_+(\chi_4)\vp_-(\chi_3)
\end{align*}
where $\omega=\chi_3\otimes \chi_4(W)=1+\frac{\ve_3\vp_4}{\kappa_3}$.  
Thus, the partial maps are equal and the lemma follows.  
\end{proof}
 
Recall the generic biquandle factorization $(Y_{\rel}, \gen_{\rel},B_{\rel},\Qm_\ell,\gm)$ of the quandle $\QX_{\rel}$ given in Theorem \ref{T:ExistY'genB'}.  The definition of $Y_{\rel}$ is motivated by the following fact: Equation \eqref{E:ChOmega} implies the characters $\chi$  on $Z$ with $\operatorname{trace}(\psi(\chi_{|Z_0}))\neq \pm2$ are in one to one correspondence with elements of $Y_{\rel}$ determined  by the assignment   $\chi\mapsto (\chi_{|Z_0},\chi(\Omega))\in Y_{\rel}$ where  $\Grc$ is  identified with characters on $Z_0$.  
Since
$\Omega\otimes1$ and $1\otimes\Omega$ are fixed by $\RR$ then we have that $B$ extends to a partial map
$B_{\rel}: Y_{\rel}\otimes Y_{\rel} \to Y_{\rel} \otimes Y_{\rel}$ defined by Equation
\eqref{E:DefB'}.
 
For each character $\chi$ on $Z$ that
corresponds to an element of $Y_{\rel}$, let $I_\chi\subset U_\xi$ be
the corresponding ideal and let $V_{\chi}$ be an irreducible
$U_\xi$-module with this character. Then $U_\xi/I_\chi$ is canonically
isomorphic to $\End_\C(V_{\chi})$ and $I_\chi$ is the two side ideal
of $U_\xi$ generated by the kernel of $\chi$ inside $Z$.  We fix an
isomorphism $\phi_\chi: V_{\chi}\to \C^r$.  Thus, we get an
isomorphism $U_\xi/I_\chi\to \Mat_r(\C)$.

 Let $A$ be the set of pairs $(\chi_1, \chi_2)\in Y_{\rel}^2$ such that
 $B_\rel(\chi_1, \chi_2)$ 
 is defined
 (i.e.\ $\chi_1\otimes\chi_2(W)\neq0$).
For $(\chi_1, \chi_2)\in A$,
 let $(\chi_4, \chi_3)=B_\rel(\chi_1, \chi_2)$ then,
 for $z_i\in\ker\chi_i$, one has
 $(\chi_3\otimes\chi_4)\RR(z_1\otimes
 1)=0=(\chi_3\otimes\chi_4)\RR(1\otimes z_2)$.
 Then the map
 $\RR:U_\xi\otimes U_\xi\to U_\xi/I_{\chi_3}\otimes U_\xi/I_{\chi_4}$
 vanishes on $I_{\chi_1}\otimes U_\xi+U_\xi\otimes I_{\chi_2}$ and
 induces an algebra isomorphism
 $$\RR: U_\xi/I_{\chi_1}\otimes U_\xi/I_{\chi_2}\to 
 U_\xi/I_{\chi_3}\otimes U_\xi/I_{\chi_4}.$$

Using $\phi_{\chi_i}$, for $i=1\cdots 4$, 
we can see this isomorphism as an
automorphism $\bar{\RR}$ of $\Mat_r(\C)\otimes \Mat_r(\C)$.  Then
linear algebra implies there exists
$\bar{R}\in \Mat_r(\C)^{\otimes 2}$ such that $\bar{\RR}=\bar{R} \cdot \bar{R}^{-1}$ where $\bar{R}$ is determined up to a scalar.  By requiring that $\det(\bar{R})=1$ then this scalar becomes a $\rr$-root of unity.
For each pair $({\chi_1},{\chi_2})\in A$, choose such a matrix $\bar{R}$ with $\det(\bar{R})=1$ and define 
\begin{equation}\label{E:braidingR}
R: V_{\chi_1}\otimes V_{\chi_2}\to
V_{\chi_3}\otimes V_{\chi_4}
\end{equation}
by
$v_1\otimes v_2 \mapsto (\phi_{\chi_3}^{-1}\otimes \phi_{\chi_4}^{-1})
\bar R (\phi_{\chi_1}\otimes \phi_{\chi_2}) (v_1\otimes v_2)$.  Then
by definition, for any $u\in U_\xi^{\otimes2}$ and $v\in V_{\chi_1}\otimes V_{\chi_2}$ we have 
\begin{equation}\label{E:RRRvu}
R(u.v)=\RR(u).R(v).
\end{equation}

\begin{theorem}\label{L:BraidingForCyclic}
  Let
  $\tau:V_{\chi_3}\otimes V_{\chi_4}\to V_{\chi_4}\otimes V_{\chi_3}$
  be the flip map.  The family of isomorphisms
\begin{equation*}
\{c_{{\chi_1},{\chi_2}}=\tau\circ R: V_{\chi_1}\otimes V_{\chi_2}\to
V_{\chi_4}\otimes V_{\chi_3}\}_{({\chi_1},{\chi_2})\in A }
\end{equation*}
satisfies the colored braid relation up to a $\rr$-root of unity.  In
particular, it satisfies $Y_{\rel}$-colored positive Reidemeister
moves $RII_{++}$ and $RIII_{+++}$ up to a $\rr$-root of unity.
\end{theorem}
The proof of Theorem \ref{L:BraidingForCyclic} is given in  Appendix \ref{Ap:Proofr2root}.
The following lemma implies that if $\chi\in Y_{\rel}$ then the $U_\xi$-module $V_\chi$ is simple.
\begin{lemma} \label{L:chinot2SS}
Let $\chi$ be a $Z_0$-character, then the following are equivalent 
\begin{enumerate}
\item $\cat_\chi$ is semi simple,
\item all the simple of $\cat_\chi$ are $r$-dimensional projective modules,
\item $\cat_\chi$ has $r$ isomorphism classes of simple modules on which
  $\Omega$ takes distinct values,
\item $\left\{
    \begin{array}{l}
      \operatorname{trace}(\psi(\chi))\neq\pm2 \text{ if $\ell\neq4$,}\\
      \operatorname{trace}(\psi(\chi))\neq2\text{ if $\ell=4$.}
    \end{array}
\right.$
\end{enumerate}
\end{lemma}
\begin{proof}
 Lemma 3 of \cite{GP4} says that if  $\operatorname{trace}(\psi(\chi))\neq\pm2$ then $\cat_\chi$ is semi simple.  So (4) implies (1).  In the proof of Lemma 3 of \cite{GP4} it is shown that (1) implies (2) and (2) implies (3).  To see (3) implies (4): Let $\xi^{\pm ra}$ be the eigenvalues of
  $(-1)^{\ell+1}\psi(\chi)$.  Then the minimum polynomial of $\Omega$ in
  $U_\xi\otimes_g\C$ is $\Ch_r(\Omega)=\xi^{ra}+\xi^{-ra}$ whose roots
  are $\xi^{a+2k}+\xi^{-a-2k}$ for $k=0\cdots r-1$.  
  Now
  \begin{align*}\xi^{a+2k}+\xi^{-a-2k}=\xi^{a+2j}+\xi^{-a-2j}&\iff 
  (\xi^{2k}-\xi^{2j})(\xi^{2a+2j+2k}-1)=0\\ & \iff\xi^{ra}=\pm1
 \end{align*}
 which implies the lemma.
\end{proof}
\subsection{The modified trace on the algebra $U_\xi$}\label{SS:ModTraceQuSl} 
Let $\Proj$ be the ideal of projective $U_\xi$-modules in $\cat$.  Lemma \ref{L:chinot2SS} implies that each module $V_\chi$ in $\{V_\chi\}_{\chi\in Y_{\rel}}$ is simple and projective.  
From Corollaries 4 and 5 of \cite{GP4} it admits a unique (up to
global scalar) nontrivial modified trace $\{\mt_{V }\}_{V\in \Proj}$
such that for any $\chi\in Y_{\rel}$,
then there exists $\alpha \in(\C\setminus \Z) \cup r\Z$ such that
$\chi(\Omega)=(-1)^r\bp{q^\alpha+q^{-\alpha}}$ and
 \begin{align}\label{E:modNilCyc}
\qd(\chi)&=\mt_{V_\chi}(\Id_{V_\chi})= (-1)^{r-1}\prod_{j=1}^{r-1}
\frac{\qn{j}}{\qn{\alpha+r-j}}\notag\\
&=(-1)^{r-1}\frac{r\,\qn{\alpha}}{\qn{r\alpha}}
=\frac{(-1)^{r-1}r}{q^{(1-r)\alpha}  +\cdots+q^{(r-3)\alpha}+q^{(r-1)\alpha}}.
\end{align}
 Moreover, $\mt_{V_\chi}(\Id_{V_\chi})= \mt_{V_\chi^*}(\Id_{V_\chi^*})$.   
  Remark that $\qd(\chi)^{-1}$ is the degree $r-1$ renormalized {\em
    second kind} Chebyshev polynomial in $\chi(\Omega)$.

\begin{lemma}\label{L:dGaugeInvInYrel}
The modified dimension $\qd$ is gauge invariant, i.e.\ if $B_1^\pm(\chi',\chi)$ is defined for  $\chi,\chi'\in Y_{\rel} $ then $\qd(\chi)=\qd({B_1^\pm(\chi',\chi)})$.
\end{lemma}
\begin{proof}
In the notation of Subsection \ref{ss:glgf}, if $(y,z)\in Y_{\rel}\subset Y\times_{(f,\Ch_r)}\C$ then $z=\chi(\Omega)$ where $\chi$ is the corresponding $Z$-character.  Thus, the lemma follows from the definition of the partial map $B_{\rel}$ given in Equation \eqref{E:DefB'} and the above formula which implies the value of $\Omega$ determines $\qd$.   
\end{proof}

 \subsection{The negative Reidemeister II move} \label{SS:NegReidIIMove}

In this subsection we show that the morphism $c$ generically satisfies the negative Reidemeister moves $RII_{-+}$ and $RII_{+-}$, this implies that $\{V_\chi, c\}$ is sideways invertible.  We do this in two major steps.  First, we show that the negative Reidemeister II moves holds up to a scalar.  Then in the second step we use positive Reidemeister moves and the properties about the modified dimension to show that this scalar is a root of unity.  

\begin{lemma}\label{L:PictPropV}  
 For $i=1$ (resp.\ $i=2$), suppose $V_i, W_i, U_i, V_i'$ and $ U_i'  $ are simple $U_\xi$-modules corresponding to elements of $Y_{\rel}$ which give a coherent coloring of the following diagrams where the coupons are colored with isomorphisms.  Then the diagrams are proportional under $F$, in other words: 
 $$
\epsh{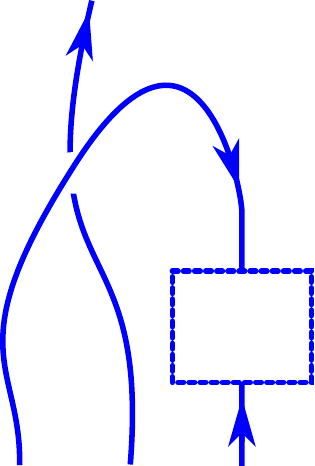}{12ex}\put(-36,-30){\ms{V_1}} \put(-24,-30){\ms{W_1}} \put(-10,-30){\ms{U_1}} \put(-8,13){\ms{U_1'}}
\stackrel\lambda=\epsh{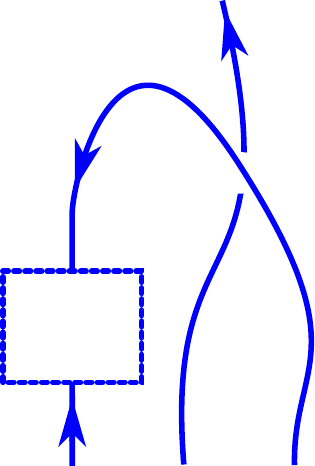}{12ex}\put(-30,-30){\ms{V_1}} \put(-18,-30){\ms{W_1}} \put(-4,-30){\ms{U_1}}\put(-34,13){\ms{V_1'}}
\;\;\;\; \text{ resp.\ }\;\; \;\;
\epsh{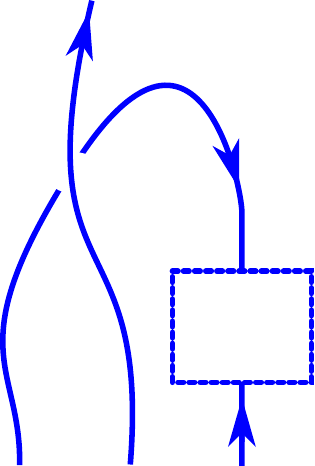}{12ex}\put(-36,-30){\ms{V_2}} \put(-24,-30){\ms{W_2}} \put(-10,-30){\ms{U_2}}  \put(-8,13){\ms{U'_2}}
\stackrel\lambda=\epsh{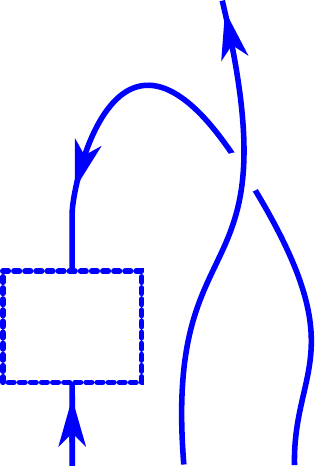}{12ex}\put(-30,-30){\ms{V_2}} \put(-18,-30){\ms{W_2}} \put(-4,-30){\ms{U_2}}\put(-34,13){\ms{V_2'}}
$$
where $\stackrel\lambda=$ means the two sides are proportional under $F$ and
the coupons are filled with any isomorphism. 
\end{lemma}
Lemma  \ref{L:PictPropV} is proved in Appendix \ref{A:ProofPictPropV} and used in the next proposition.  

\begin{proposition}\label{P:RII-scalar}
The two negative Reidemeister moves $RII_{-+}$ and $RII_{+-}$ holds up to a scalar.  
\end{proposition}
\begin{proof}
From Lemma \ref{L:PictPropV} and
isotopies of the plane we have:
$$
\epsh{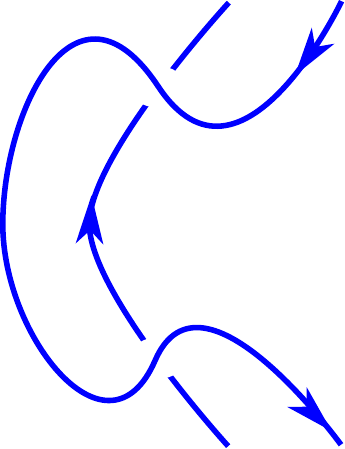}{12ex}\stackrel\lambda=\epsh{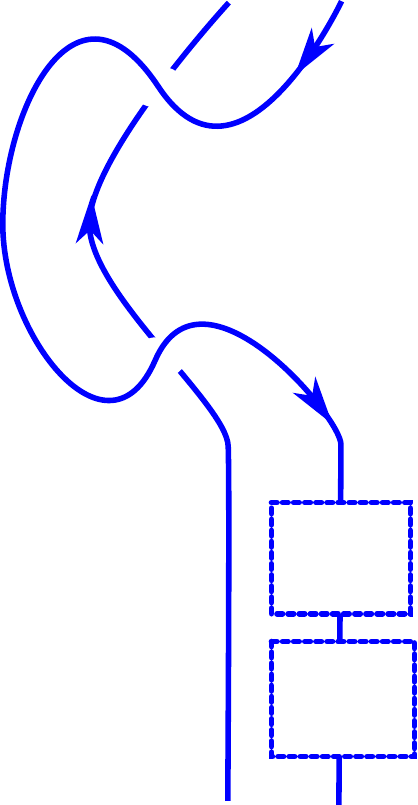}{18ex}\stackrel\lambda=
\epsh{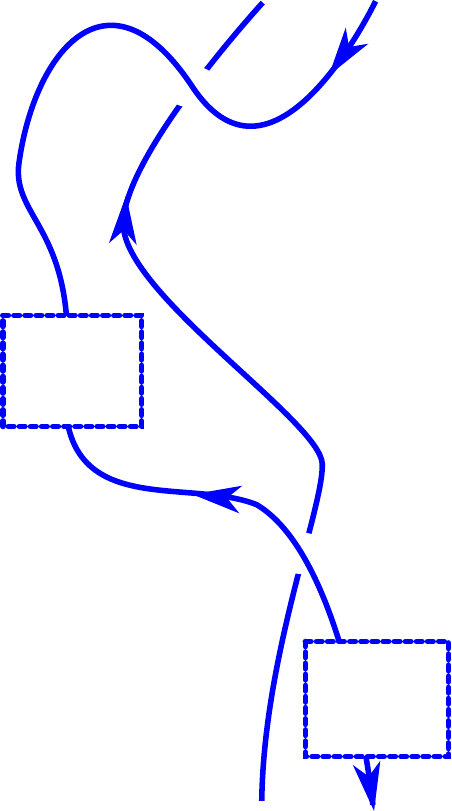}{18ex}\stackrel\lambda=\epsh{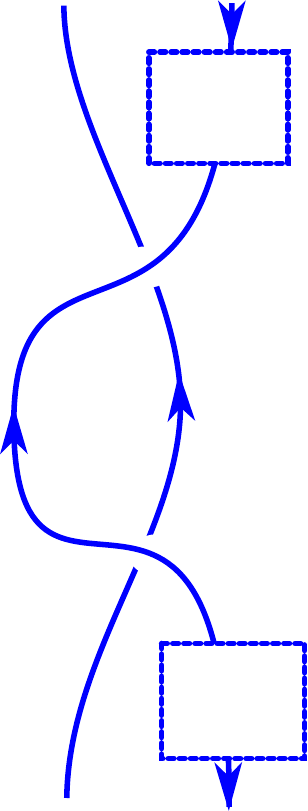}{18ex}\stackrel\lambda=\epsh{fig25}{12ex}
$$
where the last step uses the $RII_{++}$ move.  The other negative $RII$ move is similar.
\end{proof}

\begin{theorem}\label{T:NegReidMoveRoot}
  The braiding
  satisfies the 
  two negative Reidemeister moves $RII_{-+}$ and $RII_{+-}$ holds up to a $\rr$ root of one.  In particular, 
\begin{equation}\label{E:RII-root1}
\epsh{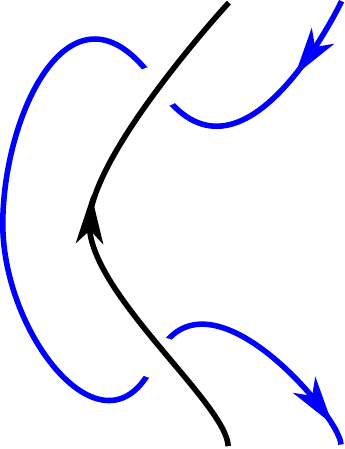}{12ex}\stackrel\rr=\epsh{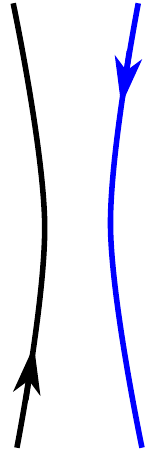}{12ex}
\end{equation} 
where $\stackrel\rr=$ 
means that for any $Y_{\rel}$-coloring the values of diagrams under $F$ are equal up to a $\rr$ root of one. 
\end{theorem}

\begin{proof}
We show $RII_{+-}$ holds up to a $\rr$ root of one.  The proof of $RII_{-+}$ is similar.  Recall the map $F'$ defined in Equation \eqref{eq:md} from the ambi pair coming from the modified trace $\{\mt_{V }\}_{V\in \Proj}$ defined in Subsection \ref{SS:ModTraceQuSl}.
  Proposition \ref{P:RII-scalar} implies the $RII_{+-}$ holds up to a
  scalar.  

    We will use the map $F'$ to compute this scalar.  However, applying $F'$ to the closure of either of the morphisms in Equation \eqref{E:RII-root1} is zero because their closures are both isotopic to split links. 
  To deal with this problem, we use a trick to link the two strands with a closed component colored with the Steinberg module, as follows.  
  
Let $\chi_0$ be the $Z$ character determined by $\chi_0(K^r)=(-1)^{r-1}$, $\chi_0(E^r)=\chi_0(F^r)=0$ and $\chi_0(\Omega) = 2(-1)^{\ell -1}$.  The $U_\xi$-module corresponding to $\chi_0$ is the Steinberg module $V_0$ which is a highest weight module with a highest weight vector $v_0$ such that $Kv_0=q^{r-1}v_0$ and $Ev_0=0$.  This module has two nice properties related with the braiding:

First,  in Section 4.2 of \cite{GP4} it is shown that there is a truncated $R$-matrix in the h-adic completion which can be specialized at the root of unity $\xi$ to an element 
$\check R^<\in U_\xi\otimes U_\xi$.  Since $\RR$ comes from the conjugation of the $h$-adic $R$-matrix and since $F^r$ and $E^r$ act by zero on $V_0$ it follows that for any $\chi\in Y_{\rel}$  we can choose $\bar R$ in the definition of the braiding so that both $c_{{\chi_0},{\chi}} $ and $c_{{\chi},{\chi_0}} $ are given by the action of $\check R^<$.   This is the same braiding used in \cite{GP4} to compute the open Hopf link.   In particular,  from Lemma 1 of \cite{GP4} we have that
 \begin{equation}\label{E:ParClosureDoubBr}(   \tev_{V_\chi} \otimes \Id_{V_{\chi_0}})(\Id \otimes c_{\any,\any} \circ c_{{\chi},{\chi_0}} )(  \coev_{V_\chi} \otimes \Id_{V_{\chi_0}})=r \Id_{V_{\chi_0}}.
 \end{equation}

Second,
given $\chi=(M\bp{\kappa,\ve,\vp},z)\in Y_{\rel} $  (for notation see Equation \eqref{E:Mkappavevp}), a direct computation using Equation \eqref{E:Formulasy4y3} shows 
$$ B(\chi_0, \chi)=(\chi^-,\chi_0) \text{ and } B(\chi,\chi_0)=(\chi_0, \chi^-)$$
where $\chi^{-}=(M\bp{\kappa,(-1)^{r-1}\ve,(-1)^{r-1}\vp},z)\in Y_{\rel}$.    In particular,
$$B\circ B(\chi_0, \chi)=(\chi_0, \chi) \text{ and } B\circ B(\chi,\chi_0)=(\chi,\chi_0).$$ 

Therefore, we can add a strand colored with $V_0$ to a link and control the new induced colors.  In particular, a strand colored with $V_0$ does not change when it is braided with any other strand.  With this in mind, we close the first tangle in Equation \eqref{E:RII-root1} with two parallel strands encircled by a closed component colored with $V_0$ (in red) to obtain the first tangle in Equation~\eqref{E:FcloserVSc}.  Now we compute the value of this tangle under $F'$ in two ways.

In the following diagrams, the coloring of the red strand is always $V_0$ and from above does not change in any move.  However, the colorings of the blue and black strands can change but as we will see it is not important to know the exact coloring.   
Let us compute: 
\begin{equation}\label{E:FcloserVSc}
F'\bp{\epsh{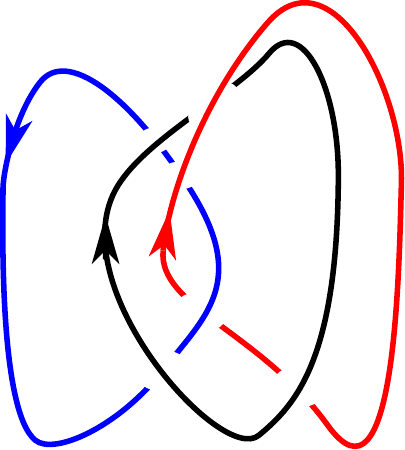}{16ex}}=\lambda F'\bp{\epsh{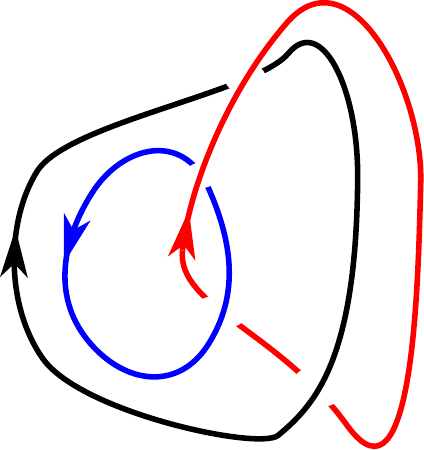}{16ex}}\stackrel\rr=
\lambda\brk{\epsh{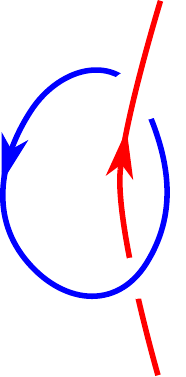}{12ex}}\brk{\epsh{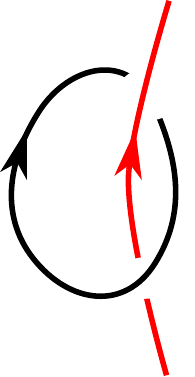}{12ex}}\qd(V_0)
\end{equation}
where $\lambda$ is the scalar coming from Proposition
\ref{P:RII-scalar} and $\brk{\;}$ is defined in Subsection \ref{SS:GaugeInvF}.  On the other hand, we have 
$$
F'\bp{\epsh{fig57}{16ex}}\stackrel\rr=F'\bp{\epsh{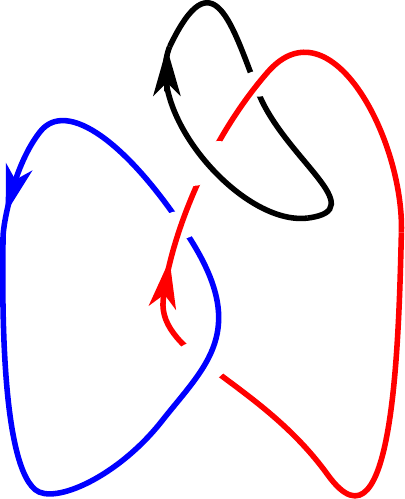}{16ex}}\stackrel\rr=\brk{\epsh{fig60}{12ex}}\brk{\epsh{fig61}{12ex}}\qd(V_0)
 $$
 where the first equality (modulo a $\rr$ root of unity) comes from doing
 positive RII and RIII moves.
 But it follows from Equation \eqref{E:ParClosureDoubBr} that for
any coloring of the closed blue and black stands we have
$$\brk{\epsh{fig60}{12ex}}\stackrel\rr=\brk{\epsh{fig61}{12ex}}\stackrel\rr=r.$$
Thus, we have $\lambda$ is a $\rr$ root of unity.
\end{proof}

\subsection{Left and right twist are equal up to a $\rr$ root of unity}\label{SS:Twist}
We now show that the left and right twists coincide.
\begin{theorem}\label{P:YBInducesTwist}
  The family $(\{V_\chi\}_{\chi\in Y_{\rel}},c)$ induces a twist, up to a $\rr$ root of unity.  

\end{theorem}
\begin{proof}
Let $\alpha$ be the partial map defined for $x\in Y_{\rel}$ by
$B^{\pm1}(x,\alpha(x))=(x,\alpha(x))$ and $\alpha^{-1}$ its partial
inverse.  For $x\in Y_{\rel}$, if there exists $y=\alpha(x)$ and $z=\alpha^{-1}(x)$,
we can define the right and the left twist:
$$\theta^R_x=\brk{\ \ \epsh{fig29}{8ex}\put(-35,-18){$x$}\put(1,0){$y$}\ }
\et\theta^L_x=\brk{\
  \epsh{fig30}{8ex}\put(0,-18){$x$}\put(-36,0){$z$}\ \ }$$
We will show that $\theta^R_x\stackrel\rr=\theta^L_x$.

Lemma \ref{L:dGaugeInvInYrel} implies the  modified dimensions of
$V_x$, $V_y$ and $V_z$ are equal.  Then using the left and right
partial trace property of the modified trace, we can compute $\mt(c_{x,y})$ in
two different ways:
$$\qd(x)\theta^R_x\stackrel\rr=F'\bp{\ \epsh{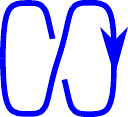}{5ex}\put(-30,0){$x$}\put(0,0){$y$}\ }
\stackrel\rr=\qd(y)\theta^L_y.$$
This implies that $\theta^R_x=\theta^L_y$.  On the other side, we have
$$\theta^L_x\stackrel\rr=\brk{F\bp{\ \epsh{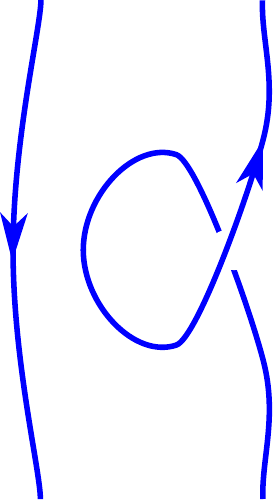}{18ex}\put(-48,2){$x$}
    \put(-8,-30){$x$}\put(-27,0){$z$}\
  }}\stackrel\rr=
\brk{F\bp{\epsh{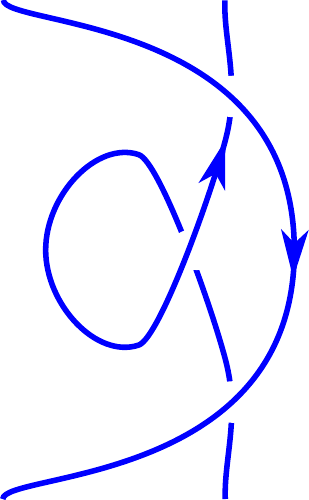}{18ex}\put(-48,2){$x$}
    \put(-10,-32){$x$}\put(-50,-32){$x$}\put(0,0){$y$}\put(-15,6){$y$}\
  }}\stackrel\rr=\theta^L_y\brk{\!F\bp{\epsh{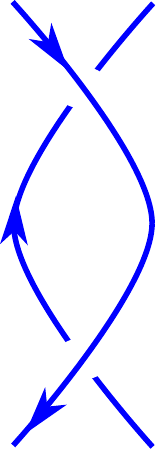}{12ex}\put(-11,-27){$x$}\put(-11,27){$x$}\put(-11,0){$y$}}}\stackrel\rr=\theta^L_y$$
where the second and fourth equalities hold because the diagrams are related by
Reidemeister moves colored by $x,y$ and $z$.
\end{proof}

\subsection{Conclusion of main example}
Let $\catro$ be the pivotal category with the same objects as $\cat$
and whose morphisms are the orbits of morphisms of $\cat$ for
 the
action of $\Z/\rr\Z$ given by multiplication by an $\rr$-root of
unity.  Then 
$$\kk=\End_{\catro}(\unit)$$ is the set of complex numbers
up to a $\rr$-root of unity (in bijection with $\C$ through
$z\mapsto z^{\rr}$).  
 \begin{proposition} The family $(\{V_\chi\}_{\chi\in Y_{\rel}},c)$  is a simple regular representation in $\catro$ of the generic biquandle factorization $(Y_{\rel}, \gen_{\rel},B_{\rel},\Qm_\ell,\gm)$.  
 \end{proposition}
 \begin{proof}
Since each morphism in $\cat$ has an image in $\catro$, the proof is a consequence of the above results about $\cat$:  If $\chi\in Y_{\rel}$ then $V_\chi$ is simple and regular in $\cat$ and it follows that it is  simple and regular in $\catro$.  Theorems \ref{L:BraidingForCyclic} implies the model $(\{V_\chi\},c)$ satisfies the colored braid relation in $\catro$.  Theorem \ref{T:NegReidMoveRoot} implies it is sideways invertible (see Equation \eqref{eq:R2-a}).  Finally, Theorem \ref{P:YBInducesTwist} give the twist.  \end{proof}

Let $\qd:  Y_{\rel}\to \kk$ be the function defined in Equation \eqref{E:modNilCyc}.  
\begin{proposition}\label{P:AmbiPairCatro}
  The pair $(\{V_\chi\}_{\chi\in Y_{\rel}},\qd)$ is an ambi pair in $\catro$ and $\qd$ is gauge invariant.
\end{proposition}
\begin{proof}
 Since $\qd$ comes from the modified trace in $\cat$ given in Subsection \ref{SS:ModTraceQuSl} we have $(\{V_\chi\}_{\chi\in Y_{\rel}},\qd)$ is an ambi pair in $\cat$.  
Then the natural map $\D_\cat \to  \D_{\catro}$ implies the proposition (for notation of $\D_\cat$ see Subsection \ref{SS:ModInv}).  
 \end{proof}

The last two propositions immediately imply:
\begin{corollary}\label{C:CycHolInv}
The family of cyclic modules and holonomy braidings $(\{V_\chi\}_{\chi\in Y_{\rel}},c)$  is a representation in $\catro$ of the  generic biquandle factorization $(Y_{\rel}, \gen_{\rel},B_{\rel},\Qm_\ell,\gm)$ of the quandle $\QX_{\rel}$ which is a sub-quandle of $\Conj(\SL_2(\C))\times_{\Ch_r} \C$.  Moreover, Theorem \ref{T:DefF'Generic} gives a gauge invariant map 
$$\wt F':=F'\circ\wt \Qm_\ell^{-1}:\Link_{\QX_{\rel}}\to\kk.$$  
\end{corollary}

\begin{remark} In Subsection \ref{SS:conclusionGenericInv} we summarized how to compute $\wt F'$.  Also, let us give the following interpretation of the invariant in Corollary \ref{C:CycHolInv}.   An element of
$\Link_{\QX_{\rel}}$ modulo gauge equivalence can be interpreted as 1)
a link $L\subset S^3$, 2) an equivalence class of flat $\SL_2(\C)$-bundle
on the complement of $L$ and 3) a choice for each component $L_i$ of $L$
of a $r$\textsuperscript{th} root of the monodromy $g_i$ around $L_i$
(conjugacy class in $\SL_2(\C)$).   The flat bundles with projective
parabolic monodromy are excluded (i.e.\  no $g_i$ with 
$\operatorname{trace}(g_i)=\pm2$).
\end{remark}

\appendix

\section{Proof related with generically define functor}\label{A:ProofPropTheorGat}
In this appendix we prove Proposition \ref{P:guitarRM} and Theorem \ref{T:stableRM}. 

\begin{proof}[Proof of Proposition \ref{P:guitarRM}]
  Assume the underlying diagrams of $D$ and $D'$ are related by a
  Reidemeister move.  By functoriality (and injectivity) of $\Qm$ it suffices to assume  $D$ and $D'$ are  diagrams  of the form 
  $\Id_{w_1}\otimes E\otimes \Id_{w_2}:w_1w_3w_2\to w_1w_4w_2$ and
  $\Id_{w_1'}\otimes E'\otimes \Id_{w_2'}:w_1'w_3'w_2'\to
  w_1'w_4'w_2'$
  where $E$ and $E'$ are the diagrams of the Reidemeister
  move and $w_i, w_i'$ are words in $Y$.  Since $\Qm$ is injective on $Y$-colored diagrams, it is clear that the source and target
  of $D$ and $D'$ are equal if and only if the source and target of
  $\Qm(D)$ and $\Qm(D')$ are equal.  
  Since $\QX$ is a quandle then in particular it is a biquandle (where $B_1(a,b)=a  \rhd b$ and $B_2(a,b)=a$ for $a,b\in \QX$).  It follows that the coloring of the incoming boundary  of a braid extends uniquely to a $\QX$-coloring of the complete braid.  Thus, the coloring $\Qm(D')$ exists,
it has the same shape as $D'$ and the same boundary colors as $\Qm(D)$.
But there is an unique such $\QX$-colored diagram and it is also obtained by doing a $\QX$-colored Reidemeister move to $\QX(D)$.  Therefore, we obtain the if and only if stated in the proposition.  
\end{proof}

To prove Theorem \ref{T:stableRM} we first state and prove a key lemma.   
Recall the $(m,n)$-cable of the positive and negative crossings given in Figure \ref{fig:DiagramWords}.  Given words,  $w,w'\in W_Y$ one can try to use the generic biquandle structure to color one of these diagrams (coloring may not exist because the biquandle maps are not defined everywhere).  If such a coloring exists we denote the corresponding $Y$-colored diagram by $\chi^+_{w,w'}$ and $\chi^-_{w,w'}$, respectively.  Suppose $\chi^-_{w,w'}$ and $\chi^+_{w'',w'''}$ exists and are ``inverses,'' i.e.  $ \chi^+_{w'',w'''}\circ \chi^-_{w,w'}$ and $\Id_w \otimes \Id_{w'} $ are related by a sequence of 
  $Y$-colored Reidemeister moves.
  In this situation we use a slight abuse of notation and denote $\chi^+_{w'',w'''}$ by $\left(\chi^-_{w,w'}\right)^{-1}$.

\begin{lemma}\label{L:stab}
   Let $D:w\to w'\in\Diag_Y$, then for generic $x\in Y$ the morphism
\begin{equation}\label{E:chiSlideDchi}
\left(\chi^-_{x,w}\right)^{-1}\circ \left[\Qm^{-1}\left(x\gm \Qm(D) \right) \otimes \Id \right]\circ\chi_{x,w}^{-}
\end{equation}
  exists and is related to $\Id_{(x,+)}\otimes D$ by a sequence of
  $Y$-colored Reidemeister moves.
\end{lemma}
\begin{proof}  
 The composition of generic
bijections of $Y$ is
  a generic
bijection of $Y$.  
   Thus for any $w\in W_Y$, 
  Axiom \ref{I:BBSSGenDef} of Definition \ref{d:genbirack}
implies that  for generic
  $x\in Y$, there exists a (unique) $Y$-coloring of the diagrams $\chi^-_{x,w}$
  and
  $(\chi^-_{x,w})^{-1}$.  We conclude that the morphism in Equation~\eqref{E:chiSlideDchi} exist for generic $x\in Y$.  

We prove the second statement in the lemma holds in two steps as follows. \\ \textbf{Step 1}: We will show that statement holds for $D=\Id\otimes E\otimes \Id:w\to w'$ be where $E$ is
  an elementary diagram made by a crossing or one cup or cap.  From Axiom \ref{I:QinvxD} of Definition   \ref{d:genbirack}, for generic $x\in Y$ there exist $D'\in \Diag_Y $ such that  
  $\Qm(D')=x\gm \Qm(D)$.   As explained in the previous paragraph,  for generic $x\in Y$ there exists  
  $$\chi^-_{x,w}:x\otimes w\to w_x\otimes y \et  \left(\chi^-_{x,w}\right)^{-1} :w_x\otimes y \to x\otimes w $$ where
  $w_x\in W_Y$ and $y\in Y$.  The  image of the morphism $\chi^-_{x,w}$ under $\Qm$ is a map with domain $\Qm(x\otimes w)=\Qm(x)\otimes (x\gm \Qm(w))$ and range $\Qm(w_x\otimes y)=\Qm(w_x)\otimes q_y$  for some $q_y\in\QX$.  Since $\chi^-_{x,w}$ comes from a negative crossing the quandle color of the over-strand does not change and so we have  $\Qm(w_x)=x\gm \Qm(w)$.
The last two sentences imply that the two $\QX$-diagrams 
$$
\Qm(D'\otimes \Id_y)\; \text{ and } \; \left(x\gm\Qm(D)\right)\otimes \Id_{q_y}
$$
 have the same incoming $\QX$-colored boundaries.   Now, since both of these $\QX$-diagrams have the same underlying  diagram which is elementary and so determined by its incoming $\QX$-colored boundary,  we conclude the diagrams are equal in $\Diag_\QX$.
 In particular, taking the preimage under $\Qm$ we have
 $$
 D'\otimes \Id_y =\Qm^{-1}\left(x\gm\Qm(D)\right) \otimes \Id_{y}.
 $$
 Finally, in $\Diag_\QX$ we have
\begin{align*}   \Qm\bp{\left(\chi^-_{x,w}\right)^{-1}} \Qm(D'\otimes \Id_y)  \Qm\bp{\chi_{x,w}^{-}}&\stackrel1\equiv\Qm\bp{\left(\chi^-_{x,w}\right)^{-1}} \Qm\bp{\chi_{x,w'}^-}\Qm\bp{\Id_x\otimes D}\end{align*}
 where the Reidemeister equivalence $\stackrel1\equiv$ comes from the fact that underlying diagrams differ by a single Reidemeister move which is  $\QX$-colored.  Since both the diagrams in the above equation are  in the image of $\Qm$ then Proposition~\ref{P:guitarRM} implies their preimage are related by a $Y$-colored  Reidemeister move.  Thus, the last two equation combined with the facts that $\Qm$ is a functor and $\bp{\chi_{x,w'}^-}^{-1}\circ \bp{\chi_{x,w'}^-}$ is related by $Y$-colored  Reidemeister moves to the identity we can conclude the lemma is true when $D=\Id\otimes E\otimes \Id$.

 \textbf{Step 2}:
  We will show, if the lemma is true for $D_1:w_1\to w_2$ and 
  $D_2:w_2\to w_3$ then  it is true for $D=D_2\circ D_1$.   To simplify notation, for $x\in Y$, we let  $ H_x: \Diag_Y\to \Diag_Y$ be the partial map given by 
  $$H_x(D)=\Qm^{-1}(x\gm\Qm(D)).$$  
  Definition \ref{d:genbirack} implies, for generic $x\in Y$ all six diagrams
  $\chi^-_{x,w_1}$, $\chi^-_{x,w_2}$, $(\chi^-_{x,w_1})^{-1}$, $(\chi^-_{x,w_2})^{-1}$, $H_x(D_1)$ and $H_x(D_2)$ exist.   Also,  
  for such $x$, we have the range of $(\chi^-_{x,w_1})^{-1}$ is $x\otimes w_1$ and the diagram 
  $\chi^-_{x,w_2}\circ(\chi^-_{x,w_1})^{-1}$ is related to the identity by a sequence of $Y$-colored Reidemeister moves (since colorings of Reidemeister II moves reducing the number of crossing always exist).   Therefore, we have 
  \begin{align*}
    \Id_{(x,+)}\otimes D&\equiv\bp{\Id_{(x,+)}\otimes D_2}\bp{\Id_{(x,+)}\otimes D_1}\\
                    &\equiv \bp{(\chi^-_{x,w_2})^{-1} \bp{H_x(D_2)\otimes \Id}
                 \chi_{x,w_2}^{-}} \bp{(\chi^-_{x,w_1})^{-1} 
                 \bp{H_x(D_1)\otimes \Id}\chi_{x,w_1}^{-}}\\
                             &\equiv(\chi^-_{x,w_2})^{-1} \bp{H_x(D_2)\otimes \Id}
                           \bp{\chi_{x,w_2}^{-}(\chi^-_{x,w_2})^{-1}} 
                          \bp{H_x(D_1)\otimes \Id}\chi_{x,w_1}^{-}\\
                      &\equiv(\chi^-_{x,w_2})^{-1} 
                          \bp{ \bp{H_x(D_2)\circ H_x(D_1)}\otimes \Id}
                          \chi_{x,w_1}^{-}\\
                        &\equiv(\chi^-_{x,w_2})^{-1} 
                           \bp{H_x(D)\otimes \Id}
                          \chi_{x,w_1}^{-}.
  \end{align*}
\end{proof}
\begin{proof}[Proof of Theorem \ref{T:stableRM}]
Since $\Qm(D)$ and
  $\Qm(D')$ represent isotopic $\QX$-tangles there exists a sequence of $\QX$-colored Reidemeister moves
 $$\Qm(D)= D_0^\QX\stackrel1\equiv D_1^\QX\stackrel1\equiv\cdots\stackrel1\equiv 
   D_n^\QX=\Qm(D').$$
It is possible that $\Qm^{-1}(D_i^\QX)$ does not exist for some $i\in \{1,...,n-1\}$.  To address this we use a gauge transformation by a generic $x\in Y$.  
Axiom \ref{I:QinvxD} of Definition \ref{d:genbirack}, implies that 
  for generic $x\in Y$, the diagrams
  $$D_i=\Qm^{-1}\bp{x\gm D_i^\QX}\in\Diag_Y$$ exists
 for all $i\in \{0,\cdots, n\}$.   
  Since the diagrams underlying  $\Qm(D_i)=x\gm D_i^\QX$ and $\Qm(D_{i+1})=x\gm D_{i+1}^\QX$ are related by a single Reidemeister move, it extends to a $\QX$-colored move. Then Proposition
  \ref{P:guitarRM} ensures that $ D_0\stackrel1\equiv
  D_1\stackrel1\equiv\cdots\stackrel1\equiv D_n$ is a sequence of
  $Y$-colored Reidemeister moves, so the same is true for
$$(\chi_{x,\any}^-)^{-1} \bp{D_0\otimes \Id_\any}\chi_{x,\any}^{-}\stackrel1\equiv\cdots\stackrel1\equiv(\chi_{x,\any}^-)^{-1} \bp{  D_n\otimes\Id_\any}\chi_{x,\any}^{-}$$ 
  where $\any$ represents the appropriate word in $W_Y$.  
  Lemma \ref{L:stab} implies that
    $$\Id_{(x,+)}\otimes D \equiv(\chi_{x,\any}^-)^{-1} \bp{ D_0\otimes   \Id_\any}\chi_{x,\any}^{-}\ets \Id_{(x,+)}\otimes D'  \equiv(\chi_{x,\any}^-)^{-1} \bp{ D_n\otimes   \Id_\any}\chi_{x,\any}^{-}.$$
    Combining the above equivalence we have 
   $$\Id_{(x,+)}\otimes D\equiv\Id_{(x,+)}\otimes D'$$ for
generic $x\in Y$.
\end{proof}

\section{Proof of Theorem \ref{T:YisGenericBiquConjSL2C}}\label{Ap:ProofSL2CGenericFact}
Here we prove Theorem \ref{T:YisGenericBiquConjSL2C}.  We need to show  $(Y,\gen, B, \Qm, \gm)$ as defined in Subsection \ref{ss:glgf} satisfy the axioms of Definition \ref{d:genbirack}.   The maps $B, B^{- 1}, S$ and $S^{- 1}$ defined in Subsection \ref{ss:glgf} by construction satisfy Axiom \ref{I:BandSrelateGenDef}.  

 For $x\in Y$, the maps
  $$B_1^\pm(x,\any),\; S_1^\pm(x,\any),\; B_2^\pm(\any,x),\; S_2^\pm(\any,x)$$
  are all rational maps (as $\psi^{\pm1},\vp_\pm$ and matrix inversion are rational maps).  In fact these maps are pairwise inverse and so they
  send Zariski open dense sets to Zariski open dense sets.  Thus, Axiom \ref{I:BBSSGenDef} of Definition \ref{d:genbirack} is satisfied.

Next we recall the definition of the functor $\Qm$, see \cite{KR} and \cite{GP}.  
Let $T$ be a standard tangle in $(0,+\oo)\times\R\times[0,1]$.  Let $D$ be a $Y$-colored diagram in $(0,+\oo)\times\{0\}\times[0,1]$ which we assume is a projection of $T$ in the $y$-coordinate direction.  We assume $T$ is close to $D$.  To define the functor we consider two kinds of paths: (1) \emph{positive paths} which are in $(0,+\oo)\times [0, +\oo)\times[0,1]$ and to the right of $T$ and (2) \emph{negative paths} in $(0,+\oo)\times (-\oo,0]\times[0,1]$ and to the left of $T$.  Let $M_+$ (resp.\ $M_-$) be the set of points traced out by all positive (resp.\ negative) paths.  Then $M_+\cap M_-$ is a subset of $(0,+\oo)\times\{0\}\times[0,1]$ whose connected components $r_i$ are delimited by the diagram $D$.   The closure of one of these components contains $\{0\}\times\{0\}\times [0,1]$, we denote this component by $r_0$.  We fix a point $P_i$ in each region $r_i$.   
Recall that a quandle $\Gd$-coloring of $D$ is
  equivalent to the data of a 
  representation 
  $\rho:\pi_1(M_T,P_0)\to \Gd$, see Remark \ref{R:QtanGtan} and Theorem~\ref{T:QD}.

 Let $(\overrightarrow{P_iP_j})_\pm$ be a path in $M_\pm$ from $P_i$ to $P_j$.    The space $M_\pm$ is contractible so a path
  $(\overrightarrow{P_iP_j})_\pm$ in the groupoid $\pi_1(M_\pm,\{P_i\})$
  is uniquely determined by its end points.  The positive and negative paths to and from adjacent regions generate the
  groupoid $\pi_1(M,\{P_i\})$ (for relations see Lemma 3.4 of \cite{GP}) so a representation of the groupoid
  in $\Gd$ is determined by their image.  If an edge $e$ of $D$
  with $Y$-color $x$ has region $r_i$ on the left and $r_j$ on the
  right, we assign to the path $(\overrightarrow{P_iP_j})_\pm$ the
  element $\vp_\pm(x)$.  The properties of the $Y$-coloring imply that this assignment satisfies the defining relations of the groupoid $\pi_1(M_T,\{P_i\})$ given in Lemma 3.4 of \cite{GP}.  Thus,   
  this assignment extends (uniquely) to a
  representation $\wb \rho$ of $\pi_1(M_T,\{P_i\})$ in $\Gd$.  The
  meridian of the edge $e$, as above, is the loop
  $(\overrightarrow{P_iP_j})_+.(\overrightarrow{P_iP_j})_-$ and its image
  by $\wb\rho$ is $\psi(x)=\vp_+(x)\vp_-(x)^{-1}\in \Gd$.  As the fundamental group
  $\pi_1(M_T,P_0)$ is generated by the meridians of the edges we have
  the restriction of $\wb\rho$ to $\pi_1(M_T,P_0)$ takes values in $\Gd$
  and thus is a $\Gd$-tangle structure $\rho$ on $T$.

  The partially defined inverse map is constructed as follows: we want to extend a
  representation
  $$\rho\in\Hom_{\text{\tiny{quandle}}}(Q(T,P_0),\Conj(\Gd))
  \cong\Hom_{\text{\tiny{group}}}(\pi_1(M_T,P_0),\Gd)$$
   to a
  representation
  $\wb\rho\in\Hom_{\text{\tiny{groupoid}}}(\pi_1(M_T,\{P_i\}),\Gd)$
  that will send $(\overrightarrow{P_iP_j})_\pm$ to
  $\vp_\pm(x)$.  This can be done if and only if 
  \begin{equation}
    \label{eq:rho-adm}
   \rho\left((\overrightarrow{P_0
      P_i})_+.(\overrightarrow{P_iP_0})_-\right)\in\psi(\Grc)
  \end{equation}
  for all points $P_i$.  
  Assuming that this is true, let
  $$g_i=\psi^{-1}\left(\rho\left((\overrightarrow{P_0
    P_i})_+.(\overrightarrow{P_iP_0})_-\right)\right)$$ then by defining 
  $\wb\rho((\overrightarrow{P_iP_j})_\pm)=\vp_\pm(g_i^{-1}g_j)$ we 
  obtain the desired representation. 
   If an edge $e$ of $D$ has
  region $r_i$ on the left and $r_j$ on the right, the $Y$-coloring
  associated to $\wb\rho$ assign to $e$ the color $g_i^{-1}g_j$.

  Hence we obtain an injective functor $\Qm$ whose image is formed by
  $\Gd$-colorings whose underlying quandle map $\rho$ satisfy Equation
  \eqref{eq:rho-adm}.  Next we check that this functor satisfies Equation \eqref{E:Qmxw}.  
  Let $x\in Y$ and $w\in W_Y$.  As above let $T_1$ and $T_2$ be two standard tangles close the trivial $Y$-colored diagrams $\Id_w$ and $\Id_{x\otimes w}$, respectively.   Let $\rho_i: \pi_1(M_{T_i},P_0)\to \Gd$ be their representation constructed above.  Let $e$ be a strand corresponding to a letter in the word $w$ (we think of $e$ in both $T_1$ and $T_2$).  Since $T_2$ is just $T_1$ with and extra strand on the left we see that to $\rho_2(e)=\rho_2((\overrightarrow{P_0P_1})_+)\rho_1(e)\rho_2((\overrightarrow{P_0P_1})_+)^{-1}$ where the conjugation of $\rho_2((\overrightarrow{P_0P_1})_+)$ come from passing over first strand of $T_2$ and back.  But by definition $\rho_2((\overrightarrow{P_0P_1})_+)=\vp_+(x)$.  Thus,  $\rho_2(e)=x\gm \rho_1(e)$ and the functor satisfies Equation \eqref{E:Qmxw}.

  Finally, for fixed $g\in\Gd$, the set of $h$ in
  $\vp_+(\Grc)$ such that $hgh^{-1}\in\psi(\Grc)$ is the complement of
  an algebraic hypersurface.  The preimage by $\vp_+$ of this set is
  an open set $Z_g\in\gen$.  Then the set of $x\in\Grc$ such that
  $ \Qm^{-1}(x\gm D)$ exists is 
  $\bigcap
  \left\{Z_g:g\in\{\rho\bp{\overrightarrow{P_\oo
        P_i}_+.\overrightarrow{P_iP_\oo}_-}\}\right\}\in\gen$.  This
  proves Axiom \eqref{I:QinvxD} and completes the proof of Theorem \ref{T:YisGenericBiquConjSL2C}.
  
  \section{Proof of Theorem \ref{c:ginv}}\label{Ap:ProofThmFBgauge}

 To prove the theorem it is convenient to prove a more general proposition about the disjoint union of diagrams of  links and 1-1 tangles.  With this in mind, consider the disjoint union of $\QX$-colored diagram $\sqcup_i D_i$ where $D_i$ is a diagram 
  representing a $\QX$-link or a 1-1 $\QX$-tangle;  
 let  $\Dk$ be the set of all such unions.  
  Also, let $\DQ$ be the set of all $\QX$-colored diagram that are in the image of
  $\Qm$.  
    For $D\in\DQ\cap\Dk$, let
  $$\FQ(D)=\brk{F\bp{\Qm^{-1}(D)}}\in\kk,$$
  where the bracket is defined as the scalar corresponding to the scalar endomorphism $F\bp{\Qm^{-1}(D)}$.
    Remark that tensoring on the right by
  any identity does not affect $\FQ$, i.e.\ $\FQ(D\otimes\Id)=\FQ(D)$.  Remark
  also that for a $Y$-colored diagram $D$ with $\Qm(D)\in \Dk$ then $\FQ(\Qm(D))=\brk{F(D)}$.   In particular, if $D$ is a $Y$-colored diagram of a link then $\FQ(\Qm(D))=F(D)$.  
  Thus, Theorem~\ref{c:ginv}.
 is a corollary of the following proposition:
  \begin{proposition} The partial map $\FQ:  \Dk\to \kk$ is $B$-gauge invariant, i.e.\ invariant under the equivalence generated by $(x\gm\any)_{x\in Y}$ and $(a\rhd\any)_{a\in \QX}$.  
  \end{proposition}
  \begin{proof}
  We prove the following five claims:
\begin{enumerate}
\item If $D\in\DQ\cap\Dk$ then for generic $x\in Y$,
  $\FQ(D)=\FQ(x\gm D)$.
\item Let $D_1, D_2\in\DQ\cap\Dk$ and $D'_1,D'_2$ be two $\QX$-colored
  diagrams of the form
  $D'_i=\Qm(\Id\otimes \Qm^{-1}(D_i))\otimes \Id$.  Suppose that there
  exists a $\QX$-colored braid diagram $\sigma$ such that
  $D'_2\circ\sigma\equiv \sigma\circ D'_1$, then
  $\FQ(D_1)=\FQ(D_2)$.
\item For any $x,y\in Y$ and any $D\in\Dk$, if $x\gm D,y\gm D\in\DQ$
  then $\FQ(x\gm D)=\FQ(y\gm D)$.
\item The partial map $\FQ$ on $\Dk$ is invariant for the
  equivalence relation generated by $(x\gm\any)_{x\in Y}$.
\item
  The partial map $\FQ$ on $\Dk$ is invariant for the
  equivalence relation generated by $(a\rhd\any)_{a\in \QX}$.
\end{enumerate}
$\bullet$
To  prove Claim (1) we see that Lemma \ref{L:stab} implies: 
 If $D=\Qm(D')$ with
$D':w\to w$ then
$$\Id_{V_x}\otimes\FQ(D).\Id_{F(w)}=F\bp{\chi^-_{x,w}}^{-1}\circ\bp{\FQ(x\gm
  D).\Id_{V_{x'}}\otimes\Id_{F(w')}}\circ F\bp{\chi^-_{x,w})}.$$
The claim then follows because $V_x\otimes F(w)$ is regular.\\

\noindent
$\bullet$ For Claim (2), let $D^Y_i= \Qm^{-1}(D_i)$ for i=1,2.  We have that for
generic $x\in Y$, the diagrams $x\gm (D'_2\circ\sigma)$ and $x\gm (\sigma\circ D'_1)$
are in $\DQ$ and they represent isotopic tangles so
$F\circ\Qm^{-1}(x\gm (D'_2\circ\sigma))=F\circ\Qm^{-1}(x\gm
(\sigma\circ D'_1))$.
Thus,
$$\FQ(x\gm D'_2).F\circ\Qm^{-1}(x\gm \sigma)=\FQ(x\gm
D'_1).F\circ\Qm^{-1}(x\gm \sigma)$$ and as
$F\circ\Qm^{-1}(x\gm \sigma)$ is invertible and
objects
are regular,
we get $\FQ(x\gm D'_1)=\FQ(x\gm D'_2)$.  Next
$$\FQ(x\gm D'_i)=\FQ(x\gm [\Qm(\Id\otimes D^Y_i)\otimes\Id])=\FQ(x\gm [\Qm(\Id\otimes D^Y_i)]),$$
then as $\Qm(\Id\otimes D^Y_i)\in\DQ\cap\Dk$ we can apply Claim (1)
twice and we get that for generic $x$,
$$\FQ(\Qm(\Id\otimes D^Y_1))=\FQ(x\gm [\Qm(\Id\otimes D^Y_1)])=\FQ(x\gm [\Qm(\Id\otimes D^Y_2)])$$$$\ \hspace*{38ex}=\FQ(\Qm(\Id\otimes D^Y_2)).$$
Finally, we see Claim (2) follows from
$\FQ(\Qm(\Id\otimes D^Y_i))=\FQ(D_i)$.
\\

\noindent
$\bullet$ We are now ready to prove Claim (3): Let $D:w\to w$ with
$D\in\Dk$.  Define $q_x,q_y\in \QX$ by
$$\Qm(\Id_{(x,-)}\otimes \Qm^{-1}(x\gm D))=\Id_{(q_x,-)}\otimes D$$
 and
$$\Qm(\Id_{(y,-)}\otimes \Qm^{-1}(y\gm D))=\Id_{(q_y,-)}\otimes D.$$
Let $\sigma$ be the $\QX$-colored braid
$$\sigma=\epsh{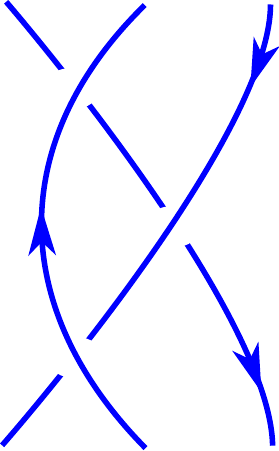}{10ex}\put(-30,-25){$q_x$}\put(-30,27){$q_y$}
\put(-15,-25){$w$}\put(-15,27){$w$}$$
then
$(\Id\otimes D\otimes \Id)\circ\sigma\equiv\sigma\circ(\Id\otimes D\otimes
\Id)$ and Claim (2) implies that $\FQ(x\gm D)=\FQ(y\gm D)$.
\\

\noindent
$\bullet$ To prove Claim (4) let us first extend the partial map $\FQ$ to a
total map $\wh F$ on $\Dk$ by $\wh F(D):=\FQ(x\gm D)$ for any $x\in Y$
such that $x\gm D\in\DQ$.  Claim (3) implies that this extension is well
defined and Claim (1) implies that $\FQ$ is indeed the restriction of
$\wh F$ on $\DQ\cap \Dk$.  Now Claim (4) is equivalent to
\begin{equation}
  \label{eq:hatFinv}
  \wh F(x\gm D)=\wh F(D), \quad \forall (x,D)\in Y\times\Dk.
\end{equation}
If $x\gm D\in \DQ$, this follows from the definition of $\wh F$.  Else, we
have 
$$\wh F(x\gm D)=\FQ(z\gm(x\gm D))$$
 for generic $z\in Y$.  But for
generic $z$, $(z\gm\any)\circ(x\gm\any)=(x'\gm\any)\circ(z'\gm\any)$ where
$(x',z')=B^{-1}(z,x)$ and in particular $z'$ is the image of $z$ by the
generic bijection $B^{-1}_2(\any,x)$.  Thus for generic $z$, $z'\gm D\in \DQ$
and then $\FQ(z\gm(x\gm D))=\FQ(x'\gm(z'\gm D))=\FQ(z'\gm D)=\wh F(D)$.
\\

\noindent
$\bullet$ Finally we prove 
$$ \wh F(q\rhd D)=\wh F(D), \quad \forall (q,D)\in \QX\times\Dk,$$
which is equivalent to Claim (5).  By \eqref{eq:hatFinv}, we can replace $q$ with
$x\gm q$ and $D$ with $x\gm D$ because $x\gm(q\rhd D)=(x\gm q)\rhd (x\gm D)$.
Doing this for generic $x$ we are left with the case where $D,q\rhd D\in\DQ$.
Then, if $D:x\to w$, the result follows by applying Claim (2) for $D_1=D$,
$D_2=q\rhd D$ and
$$\sigma=\qquad\epsh{fig3}{10ex}
\put(-18,-25){$w$}\put(-38,22){$q\rhd w$}\put(1,22){$q$}.$$
This completes the proof of the proposition.  
\end{proof}

\section{Proof of colored braid relation} \label{Ap:Proofr2root}
In this appendix we prove Theorem \ref{L:BraidingForCyclic}.
Let $\bar{R}$ be the matrix chosen when defining $c_{{\chi_1},{\chi_2}}$.
Since $\RR$ satisfies the YB equation then
$\bar{R}$
 satisfies a holonomy YB
equation up to a scalar.  Again, as all determinants are $1$ then the
holonomy YB equation is true up to a $r^3$-root of unity.
The following lemma shows that  the
this equation is actually  true up to a $r^2$-root of unity (and thus implies Theorem \ref{L:BraidingForCyclic}).
\begin{lemma}\label{L:holonomyYBr2}
  The holonomy YB equation is true up to a $\rr$-root of unity.
\end{lemma}
\begin{proof}
For $n\in\N$, let $\Theta_n$ be the set of complex
$n$\textsuperscript{th} roots of unity which acts by multiplication on
complex vector spaces.  If $f$ is a bijection between two
$n$-dimension
$\C$-vector spaces equipped with volume forms then, up
to $\Theta_n$,  there is a unique normalization $\lambda.f$ which send one
volume form to the other.  The $R$-matrix is
such a
renormalization
where the volume forms are induced from the maps $\phi_\chi$.  Hence,
if two representations
 $V,V'$  equipped with volume forms 
 have the same character then they are isomorphic by an unique volume
 preserving isomorphism up to $\Theta_{r}$.  Furthermore, as the
 tensor product of volume preserving isomorphisms is a volume
 preserving isomorphism, the $R$-matrices constructed from $V$ and
 those constructed from $V'$ correspond through these isomorphisms.

  $\bullet$ Consider now the subset $A_3$ of $Y_{\rel}^3$ formed by triplet
  $(x,y,z)$ where the set YB equation is defined
  \begin{equation}
    \label{eq:YB2}
    \begin{array}{rl}
    (\Id_{V_x}\otimes c_{y,z})^{-1}\circ (c_{\any,\any}\otimes \Id_\any)^{-1}
    \circ (\Id_\any\otimes c_{\any,\any})^{-1}\circ&\!\!\!(c_{\any,\any}\otimes \Id_\any)
    \\\circ (\Id_\any\otimes c_{\any,\any})\circ (c_{x,y}\otimes \Id_{V_z})
    =&\lambda(x,y,z)\Id_{V_x\otimes V_y\otimes V_z}   
    \end{array}
  \end{equation}
  where the $\any$ objects are completed with the biquandle structure
  $B$ and where $\lambda:A_3\to\C$ is a function.  The set $A_3$ is a 9
  dimensional complex variety obtain from $\C^9$ by removing complex
  hyper-surfaces and taking finite coverings.  In particular, it is
  path connected.  Moreover, the $R$-matrix has determinant $1$, so we have
  that $\lambda^{r^3}=1$ and $\lambda\in\Theta_{r^3}$ has discrete
  values.  Remark that the previous paragraph implies that $\lambda$
  induces a partial map
  $\wb\lambda:A_3\to \Theta_{r^3}/\Theta_{r^2}$
  which is independent of the representation
  $(V_\chi)_{\chi\in Y_\rel}$.
  
  $\bullet$ We now construct a bundle of $r$-dimensional
  $U_\xi$-module $\VV\to P'$ where $P'\to Y_{\rel}$ is onto and is a
  local homeomorphism.  
  
  Let $P=\{(y,k)\in Y_{\rel}\times \C^*: k^r=(-1)^{r-1}y(K^r)\}$ be the $r$
  fold cover of $Y_{\rel}$.  If an element $(g,\omega)\in Y_{\rel}$ satisfies
  $g(F^r)=0$, then $\Ch_r\bp{\omega}=-(-1)^\ell g(K^r+K^{-r})$ and
  there exists $k\in\C$ such that $k^r=(-1)^{r-1}g(K^r)$ and
  $\omega=(-1)^{\ell-1}\bp{k+k^{-1}}$.  We denote by $P^0$ the set of
  such $(g,\omega,k)$.  Remark that if $(g,\omega,k)\in P^0$, then
  $k\neq \pm\xi^i$ for $i=1\cdots r-1$ else one would have
  $\Ch_r\bp{\omega}=\pm2$ and $\omega=2(-1)^{\ell-1}$.

  Let $P'=P\setminus\{(g,\omega,k):g(F^r)=0\}\cup P^0$ and let
  $\ve:P'\to\C$ be defined on $p=(g,\omega,k)$ by
  $$\ve(p)=\left\{
    \begin{array}{l}
      \dfrac{\omega+(-1)^\ell\bp{k+k^{-1}}}{\qn1^2g(F^r)}\text{ if }p\notin P^0\\
      \dfrac{(-1)^{r-1}\qn1^{2r}g(E^r)}{\prod_{i=1}^{r-1}\qn i(k\xi^{-i}-k^{-1}\xi^i)}
      \text{ if }p\in P^0
    \end{array}\right..$$
  One easily checks that $\ve$ is continuous on $P'$.  Let
  $Z_\rel\subset\mathcal C^0(P',\C)$ be the polynomial algebra
  $Z_\rel=\C[\ve,k^{\pm1},\vp]$ where $\vp:(g,\omega,k)\mapsto g(F^r)$.
  Following \cite[Section VI.5]{Kas} we define a bundle of
  $U_\xi$-module $\VV=P'\times\C^r$ by giving an algebra map
  $\rho_\VV:U_\xi\to \Mat_r(Z_\rel)$ as follows:
  $$\rho_\VV(K)_{i,i}=k\xi^{r+1-2i},
  \quad \rho_\VV(F)_{i+1,i}=1,\quad \rho_\VV(F)_{1,r}=\vp$$
  $$\rho_\VV(E)_{i,i+1}=\vp\ve-(-1)^\ell\dfrac{(k\xi^{-i}-k^{-1}\xi^i)\qn i}{\qn1^2},\quad 
  \rho_\VV(E)_{r,1}=\ve$$ and other coefficients are $0$.
  
  The obvious map $\pi:P'\to Y_{\rel}$ is surjective and it is locally an
  homeomorphism.  For any $(p_1,p_2,p_3,p_4)\in (P')^4$, if
  $B(\pi(p_1),\pi(p_2))=(\pi(p_4),\pi(p_3))$ then the $R$-matrix is up
  to a scalar the unique solution in
  $$\Hom_\C\bp{\VV_{p_1}\otimes\VV_{p_2},\VV_{p_4}\otimes\VV_{p_3}}$$ of
  the six linear equations
  $$\left\{
    \begin{array}{l}
      \rho_{\VV_{p_3}\otimes\VV_{p_4}}\bp{\RR(X\otimes
      1)}R=R\rho_{\VV_{p_1}\otimes\VV_{p_2}}\bp{1\otimes X}\\
      \rho_{\VV_{p_3}\otimes\VV_{p_4}}\bp{\RR(1\otimes
      X)}R=R\rho_{\VV_{p_1}\otimes\VV_{p_2}}\bp{X\otimes 1}
    \end{array},\,X\in\qn{K,E,F}.\right.$$
  Hence the $R$-matrix is continuous in 
  the following way: there exists open neighbors $N_i$ of $p_i$ such 
  that $B$ is a bijection 
  $\pi(N_1)\times\pi(N_2)\to\pi(N_4)\times\pi(N_3)$ which lift to 
  a continuous isomorphism of bundle $\VV_{|N_1}\otimes\VV_{|N_2}
  \to\VV_{|N_4}\otimes\VV_{|N_3}$.  
  This implies that the map $\wb\lambda$ defined above is continuous thus locally constant. 
  
  $\bullet$ Finally as $\wb\lambda$ is locally constant and $A_3$ is
  connected, $\wb\lambda$ is constant and this
  constant is $\wb\lambda(\chi_0,\chi_0,\chi_0)=1$ where $\chi_0$ is the character of
  the Steinberg module.

This finishes the proof of Lemma \ref{L:holonomyYBr2}.  
\end{proof}

\section{Proof of Lemma \ref{L:PictPropV}}\label{A:ProofPictPropV}
 We need the following lemma which appears in \cite{DGP}.
\begin{lemma}[Cutting Coupon Lemma]\label{L:CuttingCouponLemma}
Let $H$ be a pivotal Hopf algebra over a field $\kk$. Let $\cat$ be the pivotal $\kk$-category of finite dimensional $H$-modules.  If $f:V_1\otimes V_2 \to V_3\otimes V_4$ is a morphism in $\cat$ such that 
$$
f (hx \otimes y) =(h \otimes 1) f (x\otimes y) \;\; \text{ or } \;\; f (x \otimes hy) =(1 \otimes h) f (x\otimes y)
$$
for all $h\in H$ and $x\otimes y\in V_1\otimes V_2$ then there exists $a_i:V_1\to V_3$ and $b_i:V_2\to V_4$ such that $f=\sum_i a_i \otimes b_i$.  
\end{lemma}

Let $V$ and $W$ be generic simple $U_\xi$-modules.
 Recall the morphism $R$ given in Equation \eqref{E:braidingR}.  Let $V'$  be a  simple $U_\xi$-modules such that there exists an
$\cat$-isomorphism
 $\beta: V'\to V^*$.     Consider the following 
morphisms
$$ f_1:=\tau_{12}R_{12}\tau_{23}R_{23}=\tau_{12}\tau_{23}R_{13}R_{23}: V'\otimes V\otimes W\to W\otimes  U'_1\otimes U_1$$
and
$$f_2:=\tau_{23}R_{23}\tau_{12}R_{12}=\tau_{23}\tau_{12}R_{13}R_{12}: W\otimes V'\otimes V\to  U'_2\otimes U_2 \otimes W$$
where the colors $U_i$ and $ U'_i$ are determined by the morphisms.   
 There exists isomorphisms  $\alpha_i:U'_i \to U_i^*$ for $i=1,2$.  Consider the following 
 morphisms:   
$$g_1=(\Id\otimes\ev_{U_1})(\Id\otimes \alpha_1 \otimes \Id)f_1: V'\otimes V\otimes W\to W$$
and
$$
g_2=(\ev_{U_2} \otimes\Id)( \alpha_2\otimes \Id \otimes \Id )f_2:W\otimes V'\otimes V\to W.
$$
\begin{lemma}\label{L:g=lambdaf}
There exist a scalar $\lambda_1, \lambda_2 \in \C$ such that 
$$g_1=\lambda_1 (\ev_V \otimes\Id)( \beta \otimes \Id\otimes \Id) \;\; \text{ and } \;\; g_2=\lambda_2 (\Id \otimes\ev_{V})( \Id \otimes \beta\otimes \Id).
$$
\end{lemma}
\begin{proof}
We prove the theorem for $g_1$ the case for $g_2$ is similar.  To do this we need to derive several equations.  Let $v\in V'\otimes V\otimes W$ and $x,y\in U_\xi$.  Then 
\begin{align}
f_1
\big((\Delta (x)\otimes y) . v\big) 
&= \tau_{12}\tau_{23}R_{13}R_{23} \big((\Delta (x)\otimes y) . v\big) \notag \\
&= \tau_{12}\tau_{23}\RR_{13}\RR_{23}\big(\Delta (x)\otimes y\big) R_{13}R_{23} (v) \notag \\
&= \tau_{12}\tau_{23} \Delta_1 \big( \RR(x\otimes y)\big) R_{13}R_{23} (v) \notag \\
&=   \Delta_2 \big(\tau \big( \RR(x\otimes y)\big)\big).
f_1
( v) \label{E:fDeltaxyv}
\end{align}
where the
second equality comes Equation \eqref{E:RRRvu} and the third equality
follows from Equation~\eqref{E:DeltaRR}.  Consider the morphism $\alpha'=
\ev_{U_1}(\alpha_1 \otimes \Id)$.
We have  
\begin{align*}
g_1\big((\Delta (x)\otimes y) . v\big) 
&=(\Id\otimes \alpha') \Big(  \Delta_2 \big(\tau \big( \RR(x\otimes y)\big)\big) .
f_1
(v) \Big)\\
&= \epsilon_2 \big( \tau \big( \RR(x\otimes y)\big) \big) .\big((\Id\otimes \alpha')
f_1
 \big)(v)\\
& =  \epsilon_1 \big(  \RR(x\otimes y)\big).
g_1(v)\\
&= \epsilon(x)y.
g_1(v)
\end{align*}
where the first equality comes from Equation \eqref{E:fDeltaxyv}, the second since $\alpha'$ is an invariant morphism and the fourth from  Equation \eqref{E:epsilonRR}.  

Thus, we have proved that
$\cat$-morphism
 $g:V'\otimes V\otimes W\to \C\otimes W$ satisfies the hypothesis of Lemma \ref{L:CuttingCouponLemma} and so there exists
$\cat$-morphisms
$a_i:V'\otimes V\to \C$ and $b_i :W\to W$ so that $g=\sum_i a_i\otimes b_i$.  But
$$\Hom_{\cat}(V'\otimes V, \C)\cong \End_{\cat}(V)\cong \C$$ and $\End_{\cat}(W)\cong \C$
 since $V$ and $W$ are  simple.  Thus, 
 for each $i$, 
 $a_i$ and $b_i$ are proportional to 
  $\ev_V(\beta \otimes \Id_V)$ and $\Id_W$, respectively.
\end{proof}
Lemma \ref{L:g=lambdaf} is an algebraic version of Lemma \ref{L:PictPropV}.

\section{Proof semi cyclic module give a representation}\label{Ap:ProofSemiCyclic}
Here we discuss why the biquandle give in 
 Subsection \ref{SS:ExampleSemicyclic} has a representation coming from the semi cyclic modules.   Let  $U_\xi\slt$ be quantum quantum $\slt$ where $\xi$ is the $2r^{th}$-root of unity.   Let $\cat$ be the category of $U_\xi\slt$ weight modules has defined in Subsection 6.1.3 of \cite{GP}.  Recall the set $X$ given in Subsection \ref{SS:ExampleSemicyclic}.   For  $x\in X$ there exists a semi cyclic module $V_x$ in $\cat$, see Theorem 6.6 of \cite{GP}.   Moreover, Equation (23) of \cite{GP} defines a Yang-Baxter model for  $\{V_x\}_{x\in X}$:   for all $x,y\in X$
\begin{equation}\label{E:BraidingSemiCyclic}
 c_{x,y}:V_x\otimes V_y\to B_1(V_x,V_y)\otimes B_2(V_x,V_y)
\end{equation}
 which satisfies the colored braid relations.  Remark that the semi cyclic modules are cyclic modules as in Subsection \ref{SS:CyclicMod}.  The braiding given in Equation \eqref{E:BraidingSemiCyclic} intertwines the actions of the generators $K$, $E$ and $F$ of $U_\xi\slt$ and their image by $\RR$.  Thus, up to a scalar, this braiding is equal to the braiding given in Theorem \ref{L:BraidingForCyclic} when the modules are semi cyclic.  Thus, the braiding in Equation \eqref{E:BraidingSemiCyclic} is sideways invertible, up to a root of unity.  

In Subsections \ref{SS:NegReidIIMove} and \ref{SS:Twist} we showed the Yang-Baxter model associated to the cyclic modules was sideways invertible and induced a twist.  Since the Yang-Baxter model of the cyclic modules was only defined up to a root of unity we were only able to show the sideways invertibility and twist held up to a root of unity.  However, the proof works in the context of semi cyclic modules and in this situation we have equalities independent of any root of unity.  We explain this now.  

As mentioned above the braiding is sideways invertible, up to a root of unity.  In other words, Yang-Baxter model satisfies the two negative Reidemeister moves $RII_{-+}$ and $RII_{+-}$ up to a scalar.  Then the proof of Theorem \ref{E:RII-root1} holds in for the semi cyclic.  The only difference is that our braiding is defined on the nose and not just up to a scalar.  Thus, we get that $\{V_x, c_{x,y}\}$ satisfies the $RII_{-+}$ and $RII_{+-}$.  Also, the proof of Theorem \ref{P:YBInducesTwist} show that $\{V_x, c_{x,y}\}$ induces a twist.

\end{document}